\documentclass[11pt,a4paper]{amsart}

\RequirePackage[OT1]{fontenc}
\RequirePackage{amsthm,amsmath}
\RequirePackage[numbers]{natbib}
\RequirePackage[colorlinks=false,citecolor=blue,urlcolor=blue]{hyperref}

\usepackage{tikz}
\usepackage{enumitem}

\usepackage{graphicx}

\newtheorem{theorem}{Theorem}
\newtheorem{lemma}[theorem]{Lemma}
\newtheorem*{lemma*}{Lemma}
\newtheorem{proposition}[theorem]{Proposition}
\newtheorem{corollary}[theorem]{Corollary}

\newtheorem*{fact*}{Fact}
\theoremstyle{definition}
\newtheorem{remark}[theorem]{Remark}
\newtheorem{definition}[theorem]{Definition}
\newtheorem*{example*}{Example}
\newtheorem{example}[theorem]{Example}

\newtheorem{question*}[theorem]{Question}
\newtheorem{questions*}[theorem]{Questions}

\usepackage[utf8]{inputenc}
\usepackage[english]{babel}
\usepackage{latexsym}
\usepackage{amssymb}
\usepackage{amsmath}

\def\vint_#1{\mathchoice%
      {\mathop{\kern 0.2em\vrule width 0.6em height 0.69678ex depth -0.58065ex
              \kern -0.8em \intop}\nolimits_{\kern -0.4em#1}}%
      {\mathop{\kern 0.1em\vrule width 0.5em height 0.69678ex depth -0.60387ex
              \kern -0.6em \intop}\nolimits_{#1}}%
      {\mathop{\kern 0.1em\vrule width 0.5em height 0.69678ex depth -0.60387ex
              \kern -0.6em \intop}\nolimits_{#1}}%
      {\mathop{\kern 0.1em\vrule width 0.5em height 0.69678ex depth -0.60387ex
              \kern -0.6em \intop}\nolimits_{#1}}}
\def\vintslides_#1{\mathchoice%
      {\mathop{\kern 0.1em\vrule width 0.5em height 0.697ex depth -0.581ex
              \kern -0.6em \intop}\nolimits_{\kern -0.4em#1}}%
      {\mathop{\kern 0.1em\vrule width 0.3em height 0.697ex depth -0.604ex
              \kern -0.4em \intop}\nolimits_{#1}}%
      {\mathop{\kern 0.1em\vrule width 0.3em height 0.697ex depth -0.604ex
              \kern -0.4em \intop}\nolimits_{#1}}%
      {\mathop{\kern 0.1em\vrule width 0.3em height 0.697ex depth -0.604ex
              \kern -0.4em \intop}\nolimits_{#1}}}

\usepackage{fullpage}

\def\Z{\mathbf{Z}}
\def\reals{\mathbf{R}}
\def\R{\reals}

\def\D{\mathbb{D}}

\def\C{\mathbf{C}}

\def\integers{\mathbf{Z}}

\def\E{\mathbb{E}\,} 
\def\P{\mathbb{P}}

\def\1{\mathbf{1}}

\DeclareMathOperator{\dist}{dist}
\DeclareMathOperator{\supp}{supp}

\renewcommand{\Im}{\operatorname{Im}}

\numberwithin{equation}{section}
\numberwithin{theorem}{section}

\author{Janne Junnila}
\address{University of Helsinki, Department of Mathematics and Statistics, P.O. Box 68, FIN-00014 University of Helsinki, Finland}
\email{janne.junnila@helsinki.fi}
\author{Eero Saksman}
\address{Department of Mathematical Sciences, Norwegian University of Science and Technology (NTNU), NO-7491 Trondheim, Norway}
\address{University of Helsinki, Department of Mathematics and Statistics, P.O. Box 68, FIN-00014 University of Helsinki, Finland}
\email{eero.saksman@helsinki.fi}
\author{Christian Webb}
\address{Department of mathematics and systems analysis, Aalto University, P.O. Box 11000, 00076 Aalto, Finland}
\email{christian.webb@aalto.fi}

\calclayout

\begin{document}

\title[Imaginary multiplicative chaos and the Ising model]{Imaginary multiplicative chaos: moments, regularity and connections to the Ising model}

\begin{abstract}
In this article we study imaginary Gaussian multiplicative chaos -- namely a family of random generalized functions which can formally be written as $e^{i X(x)}$, where $X$ is a log-correlated real-valued Gaussian field on $\R^d$, i.e. it has a logarithmic singularity on the diagonal of its covariance. We study basic analytic properties of these random generalized functions, such as what spaces of distributions do these objects live in, along with their basic stochastic properties, such as moment and tail estimates.

After this, we discuss connections between imaginary multiplicative chaos and the critical planar Ising model, namely that the scaling limit of the spin field of the so called critical planar XOR-Ising model can be expressed in terms of the cosine of the Gaussian free field, i.e.  the real part of an imaginary multiplicative chaos distribution. Moreover,  if one adds a magnetic perturbation to the XOR-Ising model, then the scaling limit of the spin field can be expressed in terms of the cosine of the sine-Gordon field, which can also be viewed as the real part of an imaginary multiplicative chaos distribution.

The first sections of the article have been written in the style of a review, and we hope that the text will also serve as an introduction to imaginary chaos for an uninitiated reader.
\end{abstract}

\maketitle

\section{Introduction}\label{sec:intro}

We begin this introduction with Section \ref{sec:background}, where we informally review what log-correlated fields and multiplicative chaos are as well as their role in modern probability theory and applications. Then in Section \ref{subsec:ichaosresults}, we state our main results concerning the existence and basic properties of imaginary multiplicative chaos. After this, we move to Section \ref{subsec:isingresults}, where we discuss our results concerning the Ising model. Finally in Section \ref{sec:outline}, we give an outline of the remainder of the article.

\subsection{Background on log-correlated fields and multiplicative chaos}\label{sec:background}

Log-correlated fields -- namely real-valued random generalized functions on $\R^d$ with a logarithmic singularity on the diagonal of the covariance kernel\footnote{For precise definitions in the Gaussian setting, see Section \ref{sec:logcor}.} -- have emerged as an important class of objects playing a central role in various probabilistic models. For example, one encounters them when studying the statistical behavior of the Riemann zeta function on the critical line \cite{ABBRS,FK,N,SW}, characteristic polynomials of large random matrices \cite{HKO,FKS,RidVir}, combinatorial models for random partitions of integers \cite{IO}, certain models of mathematical finance \cite[Section 5]{RhVa}, lattice models of statistical mechanics \cite{Kenyon}, construction of conformally invariant random planar curves such as Stochastic Loewner evolution \cite{AJKS,Sheffield}, the random geometry of two-dimensional quantum gravity and scaling limits of random planar maps \cite{DKRV,KRV,DuSh,MS}, growth models \cite{BF}, and statistical properties of disordered systems \cite{CLD}. A typical example of a log-correlated field is the two-dimensional Gaussian free field, namely the centered Gaussian process on a planar domain with covariance given by the Green's function of the domain with some prescribed boundary conditions. In the planar case, a log-correlated field can be seen as a model for a generic random surface.

In many of the above cases, a central goal is to understand geometric properties of the object described in terms of the log-correlated field. One might for example be interested in understanding how the maximum of the field behaves or one might be interested in the Hausdorff dimensions of level sets of the field. As the field is a rough object -- a random generalized function instead of a random function -- it is not obvious that any of these notions make sense. Nevertheless, in some specific situations, precise sense can be made of such questions  -- for such studies, see e.g. \cite{ABB,ABBRS,CLD,CMN,DRZ,FHK,FK,LOS,LP,M,N,PZ}. In some approaches to such geometric questions, an important role is played by a family of random measures which can be formally written as the exponential of the log-correlated field multiplied by a real parameter: $e^{\beta X(x)}dx$, where $X$ is the log-correlated field and $\beta>0$. The rigorous construction of these measures requires a regularization and renormalization procedure since a priori, one can not exponentiate a generalized function. The theory of these random measures goes under the name of multiplicative chaos and its foundations were laid by Kahane \cite{Kahane}; see also \cite{RhVa} for a recent review of the theory and \cite{Berestycki} for an elegant and concise construction of the family of measures. The connection between multiplicative chaos and the geometry of the field can be seen e.g. in \cite[Section 4]{RhVa} or the approach of \cite{Berestycki}. In addition to being of importance in geometric studies, multiplicative chaos measures also a play an important role in a rigorous definition of so-called Liouville field theory -- an example of a quantum field theory with certain symmetries under conformal transformations; for details, see e.g. \cite{DKRV,KRV}.

The importance of these multiplicative chaos measures suggests posing the question of whether one can make sense of similar objects for complex values of the parameter $\beta$ in the definition of $e^{\beta X(x)}$, or more generally can one consider similar objects for complex log-correlated fields\footnote{That is random fields whose real and imaginary parts are real-valued log-correlated fields.}. Moreover, if one can make sense such objects, what properties do they have, where do they arise, and do they perhaps say something about the geometry of the field $X$? Indeed, such objects have been studied -- see e.g. \cite{BJM,LRV} -- and also show up naturally when studying the statistics of the Riemann zeta function on the critical line and characteristic polynomials of random matrices --  see \cite{SW}. We also point out that, at least on a formal level, the situation where the parameter $\beta$ is purely imaginary plays a central role in the study of so-called imaginary geometry -- see \cite{Miller_Sheffield}. 

The purpose of this article is to study in more detail a particular case of such complex multiplicative chaos in that we consider the situation where the relevant parameter is purely imaginary:  we consider objects formally written as $e^{i\beta X(x)}$, where $\beta\in \R$, and $X$ is a real-valued Gaussian log-correlated field -- imaginary Gaussian multiplicative chaos. We have two primary goals for this article. The first one is to study  the basic properties of these objects as random generalized functions. Thus we investigate their analytic properties -- namely show that the relevant objects exist as certain random generalized functions and study their smoothness properties -- and also their basic probabilistic properties -- namely provide moment and tail estimates for relevant quantities built from this imaginary chaos. In this latter part the main novelty of our results is that they deal with general log-correlated fields, contrary to previous studies dealing with the Gaussian free field, as it turns out that the general case requires new tools. Our second goal is to prove that imaginary multiplicative chaos is a class of probabilistic objects arising naturally e.g. in models of statistical mechanics.\footnote{On a possibly related issue, we remark that we suspect that as suggested in the theoretical physics literature, imaginary multiplicative chaos can be used to give a rigorous definition of the Coulomb gas formulation of some conformal field theories, though we do not discuss this further here.} In addition to these primary goals, we use an example from random matrix theory to illustrate that there are some subtleties in constructing multiplicative chaos, both in the real and imaginary case.

As we suspect that imaginary multiplicative chaos will play a prominent role in different types of probabilistic models, we have tried to write this article in a format similar to a survey article. In particular, we review basic properties of imaginary multiplicative chaos and discuss different types of results in a style which is hopefully accessible to readers of various backgrounds and interests.

We now turn to discussing more precisely our main results.

\subsection{Main results on basic properties of imaginary multiplicative chaos}\label{subsec:ichaosresults} $ $

Naturally the starting point in discussing basic properties of imaginary multiplicative chaos is the existence of imaginary multiplicative chaos. That is, given a centered Gaussian process $X$ taking values in some space of generalized functions, and (formally) having a covariance kernel of the form 
\begin{equation}\label{eq:logcov}
C_X(x,y)=\E X(x)X(y)=\log |x-y|^{-1}+g(x,y),
\end{equation}
where $g$ is say locally bounded (see Section \ref{sec:logcor} for details), we want to make sense of $e^{i\beta X(x)}$ in some way. Some results of this flavor actually exist already -- see e.g. \cite{BJM,LRV},\footnote{In the setting of the Gaussian free field, a very similar question though with a different emphasis has been considered already in \cite{frolich}; a study related to the sine-Gordon model -- see Section \ref{subsec:isingresults} for further discussion about the sine-Gordon model and its relationship to imaginary chaos.} but there are many natural questions that remain unanswered about the objects. More precisely, \cite{BJM,LRV} impose some assumptions on the function $g$, that one would expect to be rather unnecessary, and based on their results, very little can be said about the analytic properties of the objects $e^{i\beta X(x)}$ -- are they possibly random smooth functions, are they random complex measures, or are they random generalized functions? Also probabilistic properties such as precise tail estimates are not studied in \cite{BJM,LRV}, though we do refer to \cite[Appendix A]{LSZ}, where such questions are studied in the setting of the Gaussian free field. Hence, one of our main goals is to address these issues, namely to study imaginary chaos for a rather general class of covariances $C_X$, to describe nearly optimal regularity results, as well as probabilistic results such a moment and tail estimates.

We now describe the setting of our first result concerning existence and uniqueness of imaginary chaos for rather general covariances $C_X$. As $X$ is a random generalized function instead of an honest function, $e^{i\beta X(x)}$ can not be constructed in a naive way. Instead, one must construct it through a regularization and limiting procedure. More precisely, we introduce suitable approximations to $X$, which are honest functions and which we call standard approximations $X_n$ -- see Definition \ref{def:standard} for a precise definition. Standard approximations always exist -- a typical example of a standard approximation is convolving $X$ with a smooth bump function; see Lemma \ref{le:standard} for details. One would then expect that the correct way to construct $e^{i\beta X(x)}$ is as a limit of the sequence $e^{i\beta X_n(x)+\frac{\beta^2}{2}\E [X_n(x)^2]}$. This turns out to be partially true -- as proven in \cite{LRV} under some further assumptions on $C_X$ and for a rather particular approximation, this sequence has a non-trivial limit for\footnote{Note that as we are dealing with a centered Gaussian field, $-X\stackrel{d}{=}X$ so results for $-\sqrt{d}<\beta<0$ can be obtained from the $0<\beta<\sqrt{d}$ case.} $0<\beta<\sqrt{d}$.  For larger $\beta$, it was shown in \cite{LRV} that (once again under certain assumptions on $C_X$) that one can multiply $e^{i\beta X_n(x)}$ by a suitable deterministic $n$-dependent factor to obtain convergence to complex white noise. As white noise is a well understood object, we have chosen to focus on the regime $0<\beta<\sqrt{d}$ in this article. In addition to being able to construct $e^{i\beta X(x)}$ as a limit of $e^{i\beta X_n(x)+\frac{\beta^2}{2}\E X_n(x)^2}$, one might hope that the limiting object would not depend very much on how we approximated $X$ -- this is indeed confirmed by one of our results. Finally, as mentioned in the previous section, the limiting object is rather rough; a generalized function instead of an honest function so we formulate convergence in a suitable Sobolev space of generalized functions -- we refer the reader wishing to recall the definition of the Sobolev space $H^s(\R^d)$ to the beginning of Section \ref{sec:spaces}. The precise statement concerning all these issues is the following theorem.

\begin{theorem}\label{th:existuniq}
Let  $(X_n)_{n\geq 1}$ be a standard approximation of a given log-correlated field $X$ on a domain $U\subset \R^d$ satisfying the assumptions \eqref{eq:cov} and \eqref{eq:assumptions} $($see also Definition~\ref{def:standard} for a precise definition of a standard approximation$)$.  When $0<\beta <\sqrt{d}$, the functions 
$$
\mu_n(x)=e^{i\beta X_n(x)+\frac{\beta^2}{2}\E[X_n(x)^2]},
$$
understood as zero outside of $U$, converge in probability in $H^s(\reals^d)$ for $s < -\frac{d}{2}$. The limit $\mu$ is a non-trivial random element of $H^s(\R^d)$, supported on $\overline{U}.$

Moreover, suppose that $X_n$ and $\widetilde{X}_n$ are two sequences of standard approximations of the same log-correlated field $X$ $($satisfying assumptions \eqref{eq:cov} and \eqref{eq:assumptions} below$)$, living on the same probability space,  and satisfying
\begin{equation}\label{eq:compat}
\lim_{n\to\infty} \E X_n(x) \widetilde{X}_n(y) = C_X(x,y),
\end{equation}
  where convergence takes place in measure on $U\times U.$
  Then the corresponding imaginary chaoses $\mu$ and $\widetilde{\mu}$ are equal almost surely.
\end{theorem}
\noindent Our proof of this theorem, which does not rely on martingale theory as in \cite{BJM,LRV}, is a rather basic probabilistic argument involving calculating second moments of objects such as $\int f(x)\mu_n(x)dx$ for suitable $f$ -- the proof is the main content of Section \ref{subsec:construction}. We wish to point out here that one can show that different convolution approximations satisfy the condition \eqref{eq:compat} so this theorem shows that the limiting random variable $\mu$ is indeed unique at least if one restricts one's attention to convolution approximations.

As discussed earlier, one of our main goals in this article is to understand (essentially optimal) regularity properties of the object $\mu$. While convergence in the space $H^{s}(\R^d)$ with $s<-\frac{d}{2}$ means that $\mu$ can not be terribly rough, it does not say that $\mu$ isn't say a $C^\infty$-function. The following result, which is our main result concerning analytic properties of imaginary multiplicative chaos, rules out this kind of possibility, or even the possibility that $\mu$ would be a complex measure. As this means that $\mu$ is a true generalized function, we also study more extensively to which spaces of generalized functions does $\mu$ belong to and essentially extract its optimal regularity. For a reminder of the relevant spaces of generalized functions: $B_{p,q}^s$, Triebel spaces, etc., along with their uses, we refer the reader again to Section \ref{sec:spaces}.

\medskip

\begin{theorem}\label{th:regularity}
Let  $\mu$ be the imaginary multiplicative chaos given by Theorem \ref{th:existuniq} and let $1\leq p,q\leq \infty$. Then the following are true. 
 \begin{enumerate}[label={\rm (\roman*)}]
\item  $\mu$  is almost surely not a complex measure.

\item We have almost surely $\mu \in B_{p,q,loc}^s(U)$ when $s < -\frac{\beta^2}{2}$ and $\mu\notin B_{p,q,loc}^s(U)$ when $s > -\frac{\beta^2}{2}$. 

\item Assume moreover that the function $g$ from \eqref{eq:logcov} satisfies $g\in L^\infty (U\times U)$ or that $X$ is the GFF with zero boundary conditions -- see Example \ref{ex:fields}. Then almost surely $\mu \in B_{p,q}^s(\R^d)$ when $s < -\frac{\beta^2}{2}$. 

\item Analogous statements hold for  Triebel spaces with $p,q\in [1,\infty)$.
\end{enumerate}
\end{theorem}

\noindent For obtaining upper bounds on regularity, our proof of this theorem relies on estimating low order moments of $\mu(f_k)$ for a suitable sequence of (random) test functions, while for lower bounds we combine the Fourier-analytic definition of Besov spaces with moment estimates of $\mu(f)$ for general deterministic $f$ -- the details of the proof are presented in Section~\ref{subsec:regularity}.

\smallskip

Having described our main results concerning analytic properties of imaginary multiplicative chaos, we move onto basic probabilistic properties of it. The main question we wish to answer is what can be said about the law of $\mu(f)=\lim_{n\to \infty}\int e^{i\beta X_n(x)+\frac{\beta^2}{2}\E [X_n(x)^2]}f(x)dx$ for a given $f\in C_c^\infty(U)$. Our study of this question is through analysis of moments of $\mu(f)$. The existence of all positive moments is one of the main things that makes imaginary multiplicative chaos special compared to real or general complex chaos. More precisely, if one considers general complex multiplicative chaos, formally written as $e^{\beta X(x)}$ with $\mathrm{Re}(\beta)\neq 0$, then it is known that generically $\E|\int f(x)e^{\beta X(x)}|^k$ will be finite only for $0\leq k\leq k_0$ for some finite $k_0$. We will show that for purely imaginary chaos, all moments exist. Moreover, as we will see, the moments grow slowly enough for the law of the random variable $\mu(f)$ to be characterized by the moments $\E \mu(f)^k\overline{\mu(f)}^l$, with $k,l$ non-negative integers.  A similar phenomenon has been observed for a particular model of what might be called signed multiplicative chaos -- see \cite{BM}. 

The fact that the moments $\mu(f)$ grow slowly enough to determine the distribution for a particular variant of the Gaussian free field (corresponding to $g=0$ in \eqref{eq:logcov}) follows from the work in \cite{GP,LSZ}. Interesting related estimates in connection with the sine-Gordon model are obtained in \cite{HS}. However, the case of general $g$ in \eqref{eq:logcov} leads to surprising complications.
Our analysis of moments is based on variants and generalizations of a famous inequality originally due to Onsager \cite{O}, that is often called Onsager's lemma (see e.g. \cite{DL}), or the electrostatic inequality (see e.g. \cite{GP}), as it involves the Green's function of the Laplacian in its original form. As we are not focusing on the Green's function, we find it more suitable to simply refer to our inequalities as Onsager (type) inequalities. As these Onsager inequalities are not directly properties of multiplicative chaos, we don't record them in this introduction, but refer the reader to Section \ref{subsec:moments} -- see Proposition \ref{prop:electrostatic_inequality}, Theorem \ref{th:onsager_general}, and Proposition \ref{prop:electrostatic_GFF}. Of these results, we prove in Section \ref{subsec:moments} Proposition \ref{prop:electrostatic_inequality} and Proposition \ref{prop:electrostatic_GFF} which apply in the case $d=2$. We prove Theorem \ref{th:onsager_general}, which applies for $d\neq 2$, in a separate article \cite{JSW}. There the proof of Theorem~\ref{th:onsager_general} is based on a non-trivial decomposition result for log-correlated fields, which has several other applications as well, and we hence find it more suited for a separate publication.  While one might argue that the most interesting log-correlated fields are variants of the Gaussian free field in two dimensions, we chose to present results for general $d$ as there are natural one-dimensional log-correlated fields arising e.g. in random matrix theory \cite{FKS,HKO} and also four-dimensional ones arising in the study of the uniform spanning forest -- see \cite{LSW}. 

Given our Onsager inequalities, we may then deduce that all positive integer moments of imaginary multiplicative chaos are finite, and in fact grow slowly enough to determine the law of imaginary chaos  -- which can be seen as another kind of uniqueness result. More precisely, we have the following result, which requires some further regularity from the covariance of our log-correlated field.

\begin{theorem}\label{th:moments}
Let $X$ be a log-correlated field on $U\subset \R^d$ satisfying the conditions \eqref{eq:cov} and \eqref{eq:assumptions}. Let $0<\beta<\sqrt{d}$ and $\mu$ be the random generalized function provided by Theorem \ref{th:existuniq}. Then for $f\in C_c^\infty(U)$, $\E|\mu(f)|^k<\infty$ for all $k>0$. For $d=2$, assume further that the function $g$ from \eqref{eq:cov} satisfies $g\in C^2(U\times U)$ and for $d\neq 2$, assume that $g\in H_{loc}^{d+\varepsilon}(U\times U)$  for some $\varepsilon>0$.\footnote{For the definition of the Sobolev space $H_{loc}^{d+\varepsilon}(U\times U)$, see Section \ref{sec:spaces}.}  Then there exists a constant $C>0$ independent of $f$ and $N$ such that for $N\in \Z_+$
$$
\E|\mu(f)|^{2N}\leq \|f\|_\infty^{2N} C^N N^{\frac{\beta^2}{d}N}.
$$
In particular, the law of $\mu(f)$ is determined by the moments $\E \mu(f)^k\overline{\mu(f)}^{l}$ with $l,k$ non-negative integers and $\E e^{\lambda |\mu(f)|}<\infty$ for all $\lambda>0$.
\end{theorem}
\noindent In the special case of $d=2$, $f=1$\footnote{Note that we require test functions to have compact support so in our setting $f=1$ is not strictly speaking a valid test function for $\mu$, but if one were not interested in realizing $\mu$ as a random generalized function, one could simply consider the sequence of random variables $\mu_n(1)$, which are perfectly well defined, and show that these converge to something non-trivial. Such a phenomenon of being able to make sense of a random generalized function acting on a single test function which is not a priori a valid test function is common, and occurs e.g. for white noise.}, and $g=0$, such moments can in fact be interpreted as the canonical partition function of the so-called two-dimensional two-component plasma or neutral Coulomb gas. The connection between this model and imaginary multiplicative chaos was noted in \cite[Appendix A]{LSZ}, where using the main results of \cite{LSZ}, very precise asymptotics for these moments were derived.  Moreover, using these precise asymptotics, precise estimates for the tail of the distribution of the random variable one might formally write as $|\mu(1)|$ were derived.  In this spirit, we combine Theorem \ref{th:moments} and Proposition \ref{prop:moment_lower_bound} to obtain similar but slightly weaker results for general $d,g,f$:

\begin{theorem}\label{th:tails}
Let $X$ be a log-correlated field on $U$ satisfying the conditions \eqref{eq:cov} and \eqref{eq:assumptions}. Assume further that the function $g$ from \eqref{eq:cov} satisfies the following condition$:$ if $d=2$, $g\in C^2(U\times U)$ and if $d\neq 2$, $g\in H_{loc}^{d+\varepsilon}(U\times U)$ for some $\varepsilon>0$. Now let $0<\beta<\sqrt{d}$ and $\mu$ be the random generalized function from Theorem \ref{th:existuniq}. Then for $f\in C_c^\infty(U)$
$$
\limsup_{\lambda\to \infty}\frac{\log \P(|\mu(f)|>\lambda)}{\lambda^{\frac{2d}{\beta^2}}}<0.
$$
Let us now assume further that  $f\geq 0$ and $f$ is not identically zero. Then for any $\varepsilon>0$ we have
$$
\liminf_{\lambda\to \infty}\frac{\log \P(|\mu(f)|>\lambda)}{\lambda^{\frac{2d}{\beta^2}+\varepsilon}}>-\infty.
$$
\end{theorem}
\noindent We also point out that in \cite[Proposition 17]{NW}, similar tail bounds in the setting of the Gaussian free field (or more precisely, $g=0$) were used to establish a Lee--Yang property for imaginary multiplicative chaos.

Proving the results discussed here is the main content of Section \ref{sec:ichaos}. In addition to these results, we also consider what we call universality properties of imaginary chaos in Section \ref{subsec:universality}, where we show that through a similar regularization/renormalization scheme, one can make sense of $H(X)$ for a large class of periodic functions $H$, and the relevant object can be expressed in general in terms of imaginary multiplicative chaos -- see Theorem \ref{th:kosini}. In Section \ref{subsec:critical}, we study how the objects $\mu$ behave in the vicinity of the critical point $\beta=\sqrt{d}$. More precisely, we prove in Theorem \ref{thm:crit} that once one multiplies $\mu=\mu_\beta$ by a suitable deterministic quantity tending to zero as $\beta\nearrow \sqrt{d}$, one has convergence to a weighted complex white noise. 
 
This concludes the summary of our results concerning the basic analytic and probabilistic properties of imaginary multiplicative chaos. We now turn to the connection between the Ising model and the random generalized functions $\mu$ of Theorem \ref{th:existuniq}.

\subsection{Main results on the Ising model and multiplicative chaos}\label{subsec:isingresults} 

In this section we review our basic results concerning the Ising model and imaginary chaos, beginning with some background to the problem we study. The Ising model is one of the most studied models of statistical mechanics, where the object of interest is a random spin configuration on some graph, or in other words, a random function defined on say the vertices of the graph and taking values $\pm 1$. The model is known to describe  certain aspects of ferromagnets -- for its definition (in two dimensions and $+$ boundary conditions), see Section \ref{sec:Isingdef} and for an extensive introduction to it, see e.g. \cite{Baxter}. A particularly important property of the Ising model on say $\Z^d$ with $d\geq 2$, is that at a certain temperature, known as the critical temperature, the model undergoes a phase transition, and the behavior of the correlation functions of the spin configuration change abruptly.  It has been recently proven in \cite{CHI} that for $d=2$ and precisely at this critical temperature, these correlation functions have a non-trivial scaling limit and this scaling limit possesses certain conformal symmetries -- see Theorem \ref{th:CHI} where we recall this result. Indeed, physicists know that quite generically, models of statistical physics at their critical points\footnote{Namely at a point of a phase transition where e.g. the correlation lengths of quantities of interest diverge.} have scaling limits which can be described by  quantum field theories behaving nicely under conformal transformations. While rigorously proving such statements has turned out to be very challenging for mathematicians, there has been rather spectacular progress in this direction in the case of the two-dimensional Ising model over the past two decades.

A particularly successful method for making precise mathematical sense of quantum field theories has been constructing probability measures on suitable spaces of generalized functions and proving that the relevant quantum field theory can be constructed from these random generalized functions -- we refer the interested reader to \cite{GJ} for further details about this construction. This kind of procedure has in fact more or less been carried out for the critical planar Ising model: in \cite{CGN1}, the authors proved that the random spin configuration of a critical Ising model on $\Z^2$ has as a scaling limit a certain random generalized function (whose correlation functions are closely related to those of \cite{CHI}), and it more or less follows that this gives rise to an operator and Hilbert space representation of the corresponding quantum field theory. This being said, as a probabilistic object, the scaling limit constructed in \cite{CGN1} is perhaps slightly poorly understood. Essentially no other characterisation for it is known besides being the scaling limit of the critical Ising model, or equivalently the unique random generalized function whose correlation functions are the scaling limit of the Ising ones. For example, if one wished to simulate it, to our knowledge, the easiest way is to simply simulate an Ising model on a domain with a fine mesh.

One of our goals is to show that if we change the model slightly, then one ends up with a random generalized function which can be constructed also in other ways -- in particular, simulating it boils down to simulating a sequence of independent standard Gaussian random variables. The model we consider is the so-called XOR-Ising model (see e.g. \cite{BdT,Wilson} and references therein for studies related to it), whose spin configuration is a pointwise product of two independent Ising spin configurations. Our main result concerning the XOR-Ising model is that for $d=2$ and at the critical point, the spin configuration has a scaling limit which is the real part of an imaginary multiplicative chaos. The precise result is the following (for relevant definitions and notation concerning the XOR-Ising model, see Section \ref{sec:XOR} and for the Gaussian free field, see Example \ref{ex:fields}).

\begin{theorem}\label{th:ising}
Let $X$ be the zero boundary condition Gaussian free field on a simply connected bounded planar domain $U\subset \R^2$ and let $\mathcal{S}_\delta$ denote the spin field\footnote{We find it convenient to define spin configurations as functions on faces of the lattice $\delta \Z^2$, or alternatively on the dual graph of $\delta \Z^2$, and by a spin field, we mean a function defined on $U$ which is constant on these lattice faces and in each face, it agrees with the value of the spin configuration on that face.} of the XOR-Ising model on a lattice approximation of $U$ with $\delta$-mesh and $+$ boundary conditions. Then for any $f\in C_c^\infty(U)$, 
$$
\delta^{-1/4}\int_U f(x)\mathcal{S}_\delta(x)dx\stackrel{d}{\to}\mathcal{C}^2\int_U f(x) \left(\frac{2|\varphi'(x)|}{ \mathrm{Im}\,\varphi(x)}\right)^{1/4}\cos\left(2^{-1/2}X(x)\right)dx 
$$

\noindent as $\delta\to 0$, where $\mathcal C=2^{5/48}e^{\frac{3}{2}\zeta'(-1)}$, $\varphi$ is any conformal bijection from $U$ to the upper half plane, and $\cos \frac{1}{\sqrt{2}}X(x)$ denotes the real part of the random generalized function $\mu$ constructed in Theorem \ref{th:existuniq} from convolution approximations of the random generalized function $X$ with $\beta=\frac{1}{\sqrt{2}}$, and the integral on the right hand side is formal notation meaning that we pair this random generalized function with the test function $f(x)(\frac{2|\varphi'(x)|}{ \mathrm{Im}\,\varphi(x)})^{1/4}$.
\end{theorem}
\noindent We prove this theorem in Section \ref{sec:Isingconv}. The proof follows rather easily from the strong results of \cite{CHI}, some rather rough estimates following arguments in \cite{FM}, and the method of moments which is justified by Theorem \ref{th:moments}. Rather interestingly, we note that our proof doesn't rely on anything converging to the GFF.

We emphasize here that the interpretation of Theorem \ref{th:ising} that one should have in mind is that if $\sigma_\delta(x)$ and $\widetilde\sigma_\delta(x)$ are the spin fields of two independent critical Ising realizations, then 
\begin{equation*}
\delta^{-1/4}\sigma_\delta(x)\widetilde\sigma_\delta(x)\stackrel{d}{\approx}\mathcal C^2\left(\frac{2|\varphi'(x)|}{\mathrm{Im}\, \varphi(x)}\right)^{1/4}\cos \big(2^{-1/2}X(x)\big).
\end{equation*}

While studying the XOR-Ising model might seem like an artificial idea at first, it is in fact a model both physicists and mathematicians have studied and used to derive the scaling limit of the correlation functions of the critical Ising model and is referred to as bosonization of the Ising model. More precisely, in the physics literature, a connection between squared full plane Ising correlation functions and correlation functions of the cosine of the GFF were observed in \cite{IZ} -- for a review of later developments and more on the conformal field theory of the Ising model, see e.g. \cite[Chapter 12]{DFMS}. This connection was given a rigorous basis in \cite{Dubedatbos} where the author proved an exact identity between squares of Ising correlation functions and suitable correlation functions of the dimer model and then performed asymptotic analysis of these correlation functions. Intuitively, the connection to the free field comes from the fact that the relevant dimer correlation functions can be expressed in terms of the height function of the dimer model and it is known that this converges to the free field in the fine mesh limit. 

Admittedly, for readers interested purely in the critical Ising model, our Theorem \ref{th:ising} is perhaps not much more than a curiosity showing that this notion of bosonization also makes rigorous probabilistic sense on the level of the scaling limit. This being said, we hope that from the perspective of better understanding scaling limits of critical models of statistical mechanics, Theorem \ref{th:ising} might be of some use, in that the cosine of the free field, interpreted in terms of imaginary multiplicative chaos, is a rather concrete object which might serve as a test case where proving some conjectured properties of scaling limits might be simpler than for other models -- even the Ising model as everything is constructed in terms of Gaussian random variables. Although, we do concede that analytic and probabilistic results similar to those discussed in Section \ref{subsec:ichaosresults} have largely been proven for the critical Ising model; see \cite{CGN1,CGN2,FM}. Simulation on the other hand is certainly simpler for the scaling limit of the XOR-Ising model: see Figure \ref{fig:simu} for a simulation of $\cos (2^{-1/2}X(x))$ in the unit square.

\begin{figure}\label{fig:simu}
    \centering
    \begin{minipage}{0.45\textwidth}
        \centering
        \includegraphics[width=0.9\textwidth]{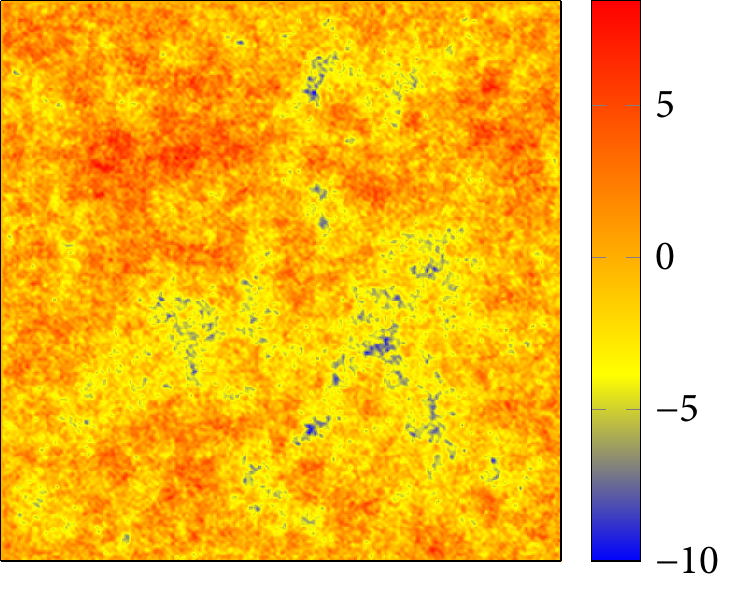} 
    \end{minipage}\hfill
    \begin{minipage}{0.45\textwidth}
        \centering
        \includegraphics[width=0.9\textwidth]{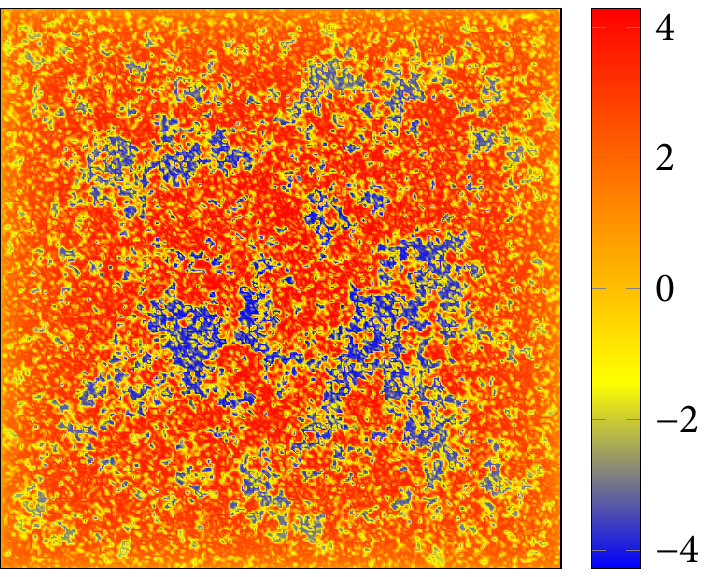} 
    \end{minipage}
        \caption{Left: A simulation of the Gaussian free field in the unit square with zero boundary conditions. The approximation is obtained by truncating the expansion in terms of Laplacian eigenfunctions at level $200^2$ -- see Example \ref{ex:fields}.{\protect\footnotemark} Right: A simulation of the cosine of the GFF obtained from the realization of the GFF in the left figure (with parameter $\beta = 1/\sqrt{2}$) -- see Lemma \ref{le:KLdef}.}
\end{figure}
\footnotetext{More precisely, the eigenfunctions are of the form $\sin(k \pi x) \sin(\ell \pi y)$, $k,\ell \ge 1$, and we have used those for which $1 \le k, \ell \le 200$.}

\smallskip

We now discuss an application of Theorem \ref{th:ising} to a model which is a perturbation of the critical XOR-Ising model. In addition to the connection between scaling limits of critical models of statistical physics and conformal field theory, physicists have argued that suitable perturbations of critical two-dimensional models of statistical mechanics should still have scaling limits described by quantum field theories which have an integrable structure despite loosing a conformal structure. For example, there exist fantastic conjectures concerning the scaling limit of the critical Ising model perturbed by a small magnetic field -- see e.g. \cite{Zamolodchikov}. Another model where this type of structure is believed to exist is the so-called sine-Gordon model, which has been studied extensively in the physics literature (see e.g. \cite{Coleman,LZ,ZZ}) and in the mathematical physics literature (see e.g. \cite{frolich,NRS,DH}). Formally, the probabilistic representation of the sine-Gordon model is a probability distribution on a suitable space of random generalized functions which is absolutely continuous with respect to the law of the full plane Gaussian free field $X$, with Radon--Nikodym derivative $\frac{1}{Z_{\beta,\mu}}e^{\mu\int \cos (\beta X(x))dx}$,\footnote{The precise definition of this is slightly delicate as the whole plane Gaussian free field is well defined only up to a random additive constant. Moreover, it is by no means clear that the ``integral" here is convergent, or more precisely that the constant function one is a valid test function for the distribution, but as we are reviewing non-rigorous results due to physicists, we ignore this issue. A rigorous construction would involve first restricting to a bounded domain and then trying to take an infinite volume limit.} where $\mu,\beta\in \R$ and $Z_{\beta,\mu}$ is a normalizing constant. The conjectural integrable structure of this model is evident e.g. in \cite{LZ}, where it is conjectured that if $X_{\mathrm{sG(\beta,\mu)}}$ is distributed according to this law, then (formally -- a rigorous statement would involve regularizing and taking a limit) for $0<\beta<2$ and $|\mathrm{Re}(\alpha)|<2/\beta$.
\begin{align*}
\E e^{i\alpha X_{\mathrm{sG}(\beta,\mu)}(x)+\frac{\alpha^2}{2}\E [X_{\mathrm{sG}(\beta,\mu)}(x)^2]}&=\left(\frac{\mu\pi\Gamma(1-\beta^2/4)}{2\Gamma(\beta^2/4)}\right)^{\frac{\alpha^2}{4-\beta^2}}e^{\int_0^\infty \left[\frac{\sinh^2\frac{\alpha \beta t}{2}}{2\sinh \frac{\beta^2 t}{4}\sinh t \cosh[1-\frac{\beta^2}{4}]t}-\frac{\alpha^2}{2}e^{-2t}\right]\frac{dt}{t}}.
\end{align*}
Note that here truly one has $0<\beta<2$ instead of $\beta<\sqrt{2}$ as one would expect from e.g. Theorem \ref{th:existuniq}. This is due to the fact that one can make sense of the sine-Gordon model also in this regime; the partition function $Z_{\beta,\mu}$ diverges, but correlation functions should be finite. We note that while slightly related to the convergence of imaginary chaos to white noise outside of the $L^2$-regime, this is a more delicate issue. At $\beta=2$, there is a far more interesting transition for the sine-Gordon model than this $L^2$-boundary at $\beta=\sqrt{2}$ for imaginary chaos. This transition is known  as the Kosterlitz--Thouless transition. We refer to \cite{DH,NRS} and references therein for further information. We also point out that the condition $|\mathrm{Re}(\alpha)|<2/\beta$ is simply the condition that the integral above converges. 

While it currently seems that proving results of this flavor, or perhaps ones involving more complicated correlation functions are out of reach, we point out that this is surprisingly similar to quantities arising in Liouville field theory where significant progress has been made recently -- compare e.g. with quantities appearing in the so-called DOZZ-formula in \cite{KRV}.

Our contribution to questions about near critical models of statistical mechanics and integrable quantum field theories is rather modest. First of all, we point out in Section \ref{sec:sg}, that in a finite domain and for suitable values of $\alpha,\beta$, using results from Section \ref{subsec:ichaosresults}, one can make sense of objects defined in the spirit of $e^{i \alpha X_{\mathrm{sG}(\mu,\beta)}(x)+\frac{\alpha^2}{2}\E [X_{\mathrm{sG}(\mu,\beta)}(x)^2]}$ -- note that as the field $X_{\mathrm{sG}(\mu,\beta)}$ is non-Gaussian, this is an instance of non-Gaussian imaginary multiplicative chaos appearing naturally in a model of mathematical physics. After this, we observe that if one adds a (non-uniform) magnetic perturbation\footnote{As can be seen from the definition in Section \ref{sec:XOR}, adding a magnetic perturbation to the XOR-Ising model is different from taking pointwise products of two independent magnetically perturbed Ising models. Thus in this near critical case, one can't expect e.g. the correlation functions of the magnetically perturbed XOR-Ising model to be related to the original Ising model in any simple way.} to the critical planar XOR-Ising model -- see Section \ref{sec:XOR} for proper definitions -- then the spin field converges to the cosine of the sine-Gordon field in the scaling limit. More precisely, we have the following theorem (for proper definitions, see Section \ref{sec:XOR} and Section \ref{sec:sg}).

\begin{theorem}\label{th:perturb}
Let $U$ be a bounded and simply connected domain and $f,\psi\in C_c^\infty(U)$.  Let $\mathcal{S}_\delta$ be distributed according to the magnetically perturbed critical XOR-Ising model with magnetic field $\psi$ on a lattice approximation of $U$ with mesh $\delta$ and $+$ boundary conditions. Also write 
$$
\widetilde{\psi}(x)=\mathcal C^2\left(\frac{2|\varphi'(x)|}{\mathrm{Im}\, \varphi(x)}\right)^{1/4}\psi(x)
$$
where $\mathcal C$ and $\varphi$ are as in Theorem \ref{th:ising}, and  let $X_{\mathrm{sG}(\widetilde\psi,1/\sqrt{2})}$ be distributed according to the sine-Gordon measure on $U$, written formally as 
$$
\frac{1}{Z_{\widetilde{\psi},\beta}}e^{\int_U \widetilde{\psi}(x)\cos [2^{-1/2}X(x)]dx}\P_{\mathrm{GFF}}(dX),
$$
where $\P_{\mathrm{GFF}}(dX)$ denotes the law of the zero boundary condition Gaussian free field on $U$ interpreted as a probability measure on say $H^{-\varepsilon}(\R^d)$.

Then as $\delta\to 0$, $\delta^{-1/4}\int_U \mathcal{S}_\delta(x)f(x)dx$ converges in law to a random variable written formally as
$$
\mathcal C^2\int_U f(x) \left(\frac{2|\varphi'(x)|}{\mathrm{Im}\, \varphi(x)}\right)^{1/4}\cos \left(2^{-1/2}X_{\mathrm{sG}(\widetilde\psi,1/\sqrt{2})}(x)\right)dx.
$$
\end{theorem}
We prove this theorem in Section \ref{sec:pertconv}. The proof follows rather easily from Theorem \ref{th:ising} and standard probabilistic arguments. The result is not very surprising given Theorem \ref{th:ising} and is certainly known in the physics literature\footnote{We also note that perhaps slightly curiously, if one perturbs the temperature of the critical XOR-Ising model suitably, then again physicists expect a connection to the sine-Gordon model -- see \cite{IZ}.}, but we do point out that it seems difficult to prove a result of this flavor only from knowledge of the scaling limit of the critical correlation functions. Again our hope is that this type of result could be interesting as it provides a rather concrete case of a near critical model of statistical mechanics which has a scaling limit, conjectured to have an integrable structure and which is concrete enough that one might hope to be able to prove results that might be out of reach in more general models, or for example for the scaling limit of the magnetically perturbed Ising model.

Finally we conclude this introduction with an outline of the remainder of the article.

\subsection{Outline of the article and acknowledgements}\label{sec:outline} 

In Section \ref{sec:logcor}, we discuss some background material concerning log-correlated fields and their approximations and remind the reader about some basic definitions and properties of spaces of generalized functions. Then in Section \ref{sec:ichaos}, we prove our results from Section \ref{subsec:ichaosresults} concerning basic properties of imaginary multiplicative chaos. In Section \ref{sec:ising}, we prove our results on the Ising model while in Section \ref{sec:rmt}, we describe briefly how imaginary multiplicative chaos arises in a model of random matrix theory and use this example to illustrate some of the subtleties of multiplicative chaos. In Appendix \ref{app:moments}, we record some basic moment bounds for imaginary chaos as well as a combinatorial counting argument we make use of in Section \ref{sec:ichaos}. 

\smallskip

{\bf Acknowledgements: } We wish to thank Antti Kupiainen for interesting discussions and references for the sine-Gordon model. We are also grateful to Julien Barral and Vincent Vargas for their comments on the manuscript. The first author was supported by the Doctoral Programme in Mathematics and Statistics at University of Helsinki. The first and second authors were supported by the Academy of Finland CoE \lq Analysis and Dynamics\rq, as well as the Academy of Finland Project \lq Conformal methods in analysis and random geometry\rq. C.W. was supported by the Academy of Finland grant 308123.

\section{Preliminaries: Introduction to log-correlated fields}\label{sec:logcor}

In this section we give a precise definition of log-correlated Gaussian fields as random generalized functions, and discuss the type of approximations or regularizations of them that we shall use to construct our imaginary multiplicative chaos. More precisely, we realize log-correlated fields as random elements of suitable Sobolev spaces of generalized functions and define a class of approximations, containing e.g. convolution approximations, that are convenient for proving the existence of imaginary multiplicative chaos.  The results of this section will probably look familiar to readers acquainted with basic facts about the Gaussian free field, as discussed e.g. in \cite[Section 4]{DubedatGFF} or \cite{SheffieldGFF}, but unfortunately the definition and study of general log-correlated fields requires slightly heavier analysis than the GFF, especially in view of applications to imaginary chaos. In addition to discussing basic facts about log-correlated fields, we review in Section \ref{sec:spaces} the basic definitions and properties of spaces of generalized functions that we will need in this article. We have intended this section as an introduction to log-correlated fields for readers interested in generalities. Readers interested only in multiplicative chaos constructed from the Gaussian free field can skip  the technical details of this section rather safely. 

\subsection{Log-correlated fields}\label{subse:log correlated}

 Intuitively, we wish to construct a centered Gaussian process $X$ on a domain $U\subset \R^d$ with covariance (kernel)
\begin{equation}\label{eq:cov}
C_X(x,y)=\E X(x)X(y)=\log |x-y|^{-1}+g(x,y),
\end{equation}
where we make the basic assumptions (used throughout the paper unless otherwise stated) that
\begin{align}\label{eq:assumptions}
\begin{cases} \quad g \in L^1(U \times U)\cap C(U\times U), \quad  g\;\;\textrm{is bounded from above in}\;\; U\times U, \quad\textrm{and}\\
 \quad U\subset\R^d\quad \textrm{is a simpy connected and bounded domain}.\end{cases}
\end{align}
These assumptions cover some of the most common examples of log-correlated fields, but we expect that many of our results hold more generally too -- in particular, one might hope to be able to relax the assumption of $g$ being bounded from above to some degree. To avoid discussing in great detail generalized functions on domains with boundaries, we find it convenient to extend $C_X(x,y)$ to $\reals^d \times \reals^d$ by setting $C_X(x,y) = 0$ whenever $(x,y) \notin U \times U$. 
In addition, we also need to of course require that $C_X$ is a covariance kernel, namely that it is symmetric and positive semi-definite: $C_X(x,y)=C_X(y,x)$ and
\[\int C_X(x,y) f(x) \overline{f(y)} \, dx \, dy \ge 0\]
for all $f \in C^\infty_c(\R^d)$. When a  result needs more regularity to be assumed of $g$ or $U$, this will be stated separately. 

We note first that actually our conditions on $C_X$ imply much stronger integrability of the covariance -- we will make use of this to realize our process $X$ as a random element in a suitable Sobolev space.

\begin{lemma}\label{le:integrability}
Assume that $C_X$ is a covariance kernel satisfying \eqref{eq:cov} and \eqref{eq:assumptions}.
Then $C_X\in L^{p}(U \times U)$ for all $p<\infty.$
\end{lemma}
\begin{proof}
Let $\psi_\varepsilon:= \varepsilon^{-d}\psi(\cdot/\varepsilon)$, where $\psi\in C^\infty_c(\R^d)$ is a standard smooth, non-negative bump function with integral 1. We denote the mollified covariance by
$$
  C_{X_\varepsilon}(x,y):= \int_{\R^{2d}}\psi_\varepsilon(x-x')\psi_\varepsilon(y-y')C_X(x',y')dx'dy'.
$$ 
From the definition of $C_X$, it easily follows that $C_{X_\varepsilon}$ is a  smooth honest covariance function (it will actually turn out to be the covariance of the mollified field $\psi_\varepsilon*X,$ but  we do not need this here). 
By smoothness, for any integer $p\geq 1$ also the power $(C_{X_\varepsilon})^p$ is a covariance, as is seen by considering products of independent copies of corresponding Gaussian fields.
We apply the covariance condition on a smooth test function that is 1 on $U+B(0,1)$  and obtain for $\varepsilon \in (0,1)$ and any integer $p\geq 1$ the inequality
$
\int_{\R^{2d}}(C_{X_\varepsilon}(x,y))^p dxdy \geq 0.
$
By decomposing the covariance $C_{X_\varepsilon}$ into its positive and negative part: $C_{X_\varepsilon}=(C_{X_\varepsilon})_+ -(C_{X_\varepsilon})_- $ and noting that
$(C_{X_\varepsilon})_+\leq [(C_{X})_+]_\varepsilon,$ it follows\footnote{Here in the first step we use the fact that since $C_+\cdot C_-=0$ and $p$ is odd, $0\leq \int(C_+-C_-)^p=\int C_+^p-\int C_-^p$, which is the first inequality. Bounding the norm of $(C_+)_\varepsilon$ with the norm of $C_+$ is justified e.g. by Young's convolution inequality.} that for any positive odd integer $p$
\begin{eqnarray*}
&&\int_{\R^{2d}}\left[(C_{X_\varepsilon})_-(x,y)\right]^p dxdy\leq \int_{\R^{2d}}\left[(C_{X_\varepsilon})_+(x,y)\right]^p dxdy\; \leq \; \int_{\R^{2d}}\left[((C_X)_+)_\varepsilon(x,y)\right]^p dxdy\\
&\leq &  \int_{\R^{2d}}\left[(C_X)_+(x,y)\right]^p dxdy
\; =:\;  c_p <\infty,
\end{eqnarray*}
where the last step follows by Minkowski's inequality and the assumption that $g$ is bounded from above.
Since $C_{X_\varepsilon}\to C_X$ almost everywhere as $\varepsilon \to 0^+$, we also see that almost everywhere, $(C_{X_\varepsilon})_-\to (C_X)_-$, and we may use Fatou's lemma to deduce that $\int_{\R^{2d}}((C_X)_-(x,y))^p dxdy\leq c_p<\infty.$ Again, since $g$ is bounded from above, Minkowski's inequality implies now that $C_X\in L^p(U\times U)$ for arbitrary positive odd integers $p$ and hence for all real $p\geq 1$.
\end{proof}

\begin{remark}
  Using our assumption that $(C_X)_+(x,y)\leq c_0+\log(1/|x-y|)$, the moment bound  obtained in the proof may be used to deduce the stronger integrability $e^{(d - \varepsilon)|C_X|} \in L^1(U \times U)$ for every $\varepsilon > 0$.\hfill $\blacksquare$
\end{remark}

The previous lemma verifies in particular that $(x,y)\mapsto C_X(x,y) \in L^2(\reals^d \times \reals^d)$, whence the operator $C_X \colon L^2(\reals^d) \to L^2(\reals^d)$ with the integral kernel $C_X(x,y)$ is Hilbert--Schmidt. In particular, it is symmetric and compact, so by the spectral theorem there exists a sequence $\lambda_1 \ge \lambda_2 \ge \dots > 0$ of strictly positive eigenvalues and corresponding eigenfunctions $\varphi_n$, that together with those eigenfunctions that correspond to the eigenvalue $0$ form an orthonormal basis for $L^2(\reals^d)$. We will now formally define $X$ via the (generalized) Karhunen--Lo\`{e}ve expansion
\begin{equation}\label{eq:def_X}
  X(x) := \sum_{n=1}^\infty A_n \sqrt{\lambda_n} \varphi_n(x), \qquad x\in\R^d,
\end{equation}
where $A_n$ are i.i.d. $N(0,1)$ random variables. Note that the functions $\varphi_n$ are supported on $U$. Let us now show that this sum converges in a suitable Sobolev space of generalized functions -- we refer the reader to Section \ref{sec:spaces} for the definition of the $L^2$-based standard Sobolev spaces $H^{s}(\reals^d)$. While this result is well known for the GFF, and probably not very surprising to readers familiar with log-correlated fields,  we choose to give a detailed proof of it here as it does not seem to appear in the literature. 
\begin{proposition}\label{prop:karhunen_loeve}
  The series on the right-hand side of \eqref{eq:def_X} converges in $H^{-\varepsilon}(\reals^d)$ for any $\varepsilon > 0$ to a $H^{-\varepsilon}(\reals^d)$-valued Gaussian random variable with covariance kernel $C_X$ satisfying \eqref{eq:cov} and \eqref{eq:assumptions}.
\end{proposition}

\begin{proof}
  We start by showing that the series converges in $H^{-d/2-\varepsilon}(\reals^d)$ for any $\varepsilon > 0$. Let $X_n(x) := \sum_{k=1}^n A_k \sqrt{\lambda_k} \varphi_k(x)$ denote the $n$th partial sum of \eqref{eq:def_X}. Then $X_n$ form a $H^{-d/2-\varepsilon}(\reals^d)$-valued martingale. As $H^{-d/2-\varepsilon}(\reals^d)$ is a Hilbert space, it is enough to show that 
  \begin{equation}\label{eq:unisobolev}
  \sup_{n \ge 1} \E \|X_n\|^2_{H^{-d/2-\varepsilon}} < \infty
  \end{equation}
  in view of the almost sure convergence of Hilbert space valued $L^2$-bounded martingales (see e.g. \cite[Theorem 3.61, Theorem 1.95]{HNVW}).  For $f \in L^1(\reals^d)$ we denote its Fourier transform by $\widehat{f}(\xi) := \int_{\reals^d} f(x) e^{-2\pi i \xi \cdot x} \, dx$. 
  Using elementary bounds along with orthogonality of the eigenfunctions, we may compute
  \begin{align*}
    \E \|X_n\|^2_{H^{-d/2-\varepsilon}} & = \int_{\reals^d} \frac{\E |\widehat{X}_n(\xi)|^2}{(1 + |\xi|^2)^{d/2 + \varepsilon}} \, d\xi \le \int_{\reals^d} \frac{\int_{U \times U} |\E X_n(x) X_n(y)| \, dx \, dy}{(1 + |\xi|^2)^{d/2 + \varepsilon}} \, d\xi \\
    & \le  C_\varepsilon\int_{U \times U} \big|\sum_{k=1}^n \lambda_k \varphi_k(x) \varphi_k(y)\big| \, dx \, dy \\
    & \le C_\varepsilon |U| \left(\int_{\R^d \times \R^d} \Big|\sum_{k=1}^n \lambda_k \varphi_k(x) \varphi_k(y)\Big|^2 \, dx \, dy\right)^{1/2}\\
    & =C_\varepsilon |U| \sqrt{\sum_{k=1}^n \lambda_k^2}\;\;  \le\;\; C_\varepsilon |U| \|C_X\|_{HS}<\infty
  \end{align*}
  for some constant $C_\varepsilon > 0$ and $\|C_X\|_{HS}$ denoting the Hilbert--Schmidt norm of $C_X$. This proves \eqref{eq:unisobolev}.

  Next we show that $X$ actually takes values almost surely in $H^{-\varepsilon}(\reals^d)$. We denote by $X_\delta:=\psi_\delta*X$ a standard mollification of the field $X$ (here $\psi_\delta$ is as in the proof of Lemma \ref{le:integrability}) whose covariance satisfies $C_{X_\delta} \in C_c^\infty(\reals^{2d})$. Moreover, writing $a_\delta(x) := \int_{\reals^d} C_{X_\delta}(u, u-x) \, du$ we have $a_\delta \in C_c^\infty(\reals^d)$ and
  \[\E |\widehat X_\delta (\xi)|^2=\int_{\reals^{2d}} C_{X_\delta}(x,y) e^{2\pi i \xi \cdot (y-x)} \, dx \, dy = \widehat{a}_{\delta}(\xi).\]
  We compute for large enough $p$ and small enough $\delta > 0$ that
  \begin{align*}
    \E \|X_\delta\|_{H^{-\varepsilon}(\reals^d)}^2 & = \int_{\reals^d} \frac{\E |\widehat{X_\delta}(\xi)|^2}{(1 + |\xi|^2)^{\varepsilon}} \, d\xi \le  \int_{\reals^d} \frac{\E |\widehat{X_\delta}(\xi)|^2}{|\xi|^{2\varepsilon}} \, d\xi  \\
    & = \int_{\reals^d} \frac{\widehat{a}_\delta(\xi)}{|\xi|^{2\varepsilon}} \, d\xi = c_\varepsilon \int_{\reals^d} \frac{a_\delta(x)}{|x|^{d - 2\varepsilon}} \, dx = c_\varepsilon \int_{U'^2} \frac{C_{X_\delta}(x, y)}{|x - y|^{d - 2\varepsilon}} \, dx \, dy \\
    & \le c'_{\varepsilon,p} \|C_{X_\delta}\|_{L^p(U'^2)} \le c'_{\varepsilon,p} \|C_X\|_{L^p(U^2)} <\infty,
  \end{align*}
  where the last inequality is due to Lemma~\ref{le:integrability} and the second to last from Young's convolution inequality.
  Above $U' = U + B(0,1)$ and we used the fact that $(|\cdot|^{-2\varepsilon})^{\widehat{\;}} = c_\varepsilon |\cdot|^{-d+2\varepsilon}$. We then obtain
  \[\E \|X\|^2_{H^{-\varepsilon}(\reals^d)} = \E \lim_{\delta \to 0} \|X_\delta\|^2_{H^{-\varepsilon}(\reals^d)} \le \liminf_{\delta \to 0} \E \|X_\delta\|^2_{H^{-\varepsilon}(\reals^d)} \le c'_{\varepsilon,p} \|C_X\|_{L^p(U^2)}<\infty.\]

  Finally, we lift the convergence $X_n \to X$ from $H^{-d/2-\varepsilon}(\reals^d)$ to $H^{-\varepsilon}(\reals^d)$. By the previous argument and by construction, the  $H^{-\varepsilon}(\reals^d)$-valued random variables $X_n$ and $X-X_n$ are symmetric, independent, and their norms have finite variance. By considering  $H^{-\varepsilon}(\R^d)$ as a real Hilbert space, the symmetry and independence yield for any $n\geq 1$
\begin{align*}
  \E\| X\|^2_{H^{-\varepsilon}(\reals^d)} &=  \E\| X-X_n\|^2_{H^{-\varepsilon}(\reals^d)}+2\E \langle X-X_n,X_n\rangle_{H^{-\varepsilon}(\reals^d)}+\E\| X_n\|^2_{H^{-\varepsilon}(\reals^d)}\; \\ &\geq \; \E\| X_n\|^2_{H^{-\varepsilon}(\reals^d)}.
\end{align*}
Thus $(X_n)$ is a $L^2$-bounded ${H^{-\varepsilon}(\reals^d)}$-valued martingale, which again yields the stated convergence.
\end{proof}

\begin{remark}
  The existence of $X$ as say a random tempered distribution could also be deduced by many other ways, e.g. it is a rather direct consequence of Bochner--Minlos' theorem (see e.g. \cite[Theorem 2.3]{Simon}). However, we wanted to avoid the more abstract framework and obtain directly the optimal Sobolev regularity. \hfill $\blacksquare$
\end{remark}

To give the reader a sharper picture of what kind of objects log-correlated fields are, we discuss a bit further their smoothness properties. It is well-known and easy to show that the field $X$ is almost surely not a Borel measure. However,  it only barely fails  being one, or even a function, since an arbitrarily small degree of  smoothing makes $X$ a continuous function. In order to make this precise, we recall that given $\delta\in\R$ there is a standard $\delta$-lift operator $I^\delta$ that smoothes a given tempered distribution ``by an amount of $\delta$'', see \eqref{eq:lift} below.
Here is the exact statement concerning $X$ being nearly a continuous function:
\begin{lemma}\label{le:tarkempi}
  Let us assume that $C_X$ is as in \eqref{eq:cov} and \eqref{eq:assumptions}.  For any $\delta > 0$ there is an $\varepsilon >0$ so that almost surely $I^\delta X \in C^{\varepsilon}(\R^d)$ -- the space of $\varepsilon$-H\"older continuous functions. A fortiori, $X \in C^{-\varepsilon}(\reals^d)$ for any $\varepsilon > 0$. \end{lemma}
\begin{proof}
  We assume that $\delta\in (0,1).$ The covariance of $I^\delta X$ is given by $C_\delta := (G_\delta \otimes G_\delta)* C_X$, where $G_\delta$ is the so-called Bessel kernel, which is the integral kernel of the operator $(I-\Delta)^{-\delta/2}$ --  see \eqref{eq:lift}. Classical representations  (see  \cite[(3,1)--(3,5), (4,1)]{AS}) of the Bessel-kernel  $G_\delta$  imply that 
 $$
 G_\delta (x-y)=|x-y|^{\delta-d}H(|x-y|),
 $$
where $H$ is an entire analytic function (as a side remark one may note that the main term in the resulting asymptotics has the same behaviour as the Riesz potential). 
  Using this representation one can verify that given  any $\delta>0$, there is a $p_0(\delta)>1$ and $\alpha>0$ such that for $p\in (1,p_0(\delta))$ it holds that
  $$
  \|(G_\delta \otimes G_\delta)(\cdot-x)- (G_\delta \otimes G_\delta) (\cdot)\|_{L^p(B\times B)}\; \lesssim  |x|^{\alpha}
  $$
for any ball $B\subset\R^d$ and $x\in B\times B$. When this is combined with the fact that $C_X$ has compact support and $C_X\in L^q(\R^{2d})$ for all $q<\infty$ by Lemma \ref{le:integrability}, one obtains by H\"older's inequality that the Gaussian field $I^\delta X$ has a H\"older-continuous covariance. In turn, this is well-known \cite[Theorem 1.4.1]{AT} to imply that the realizations of $I^\delta X$ can be taken to be H\"older continuous.

The final statement then follows from basic properties of the operator $I^{\delta}$, see the discussion around \eqref{eq:lift}.
\end{proof}
\noindent In comparison, Proposition \ref{prop:karhunen_loeve} states that $X$ only barely fails being an $L^2$-function, while Lemma \ref{le:tarkempi} states that $X$ only barely fails being a H\"older continuous function, which is of course a stronger claim.

\medskip

We now point out two examples of log-correlated Gaussian fields which will also play a role in our applications later on.

\begin{example}\label{ex:fields}
{\rm{Most common examples of log-correlated fields involve the two-dimensional Gaussian free field. While there are many related examples, we will consider the following two as they will be important in our applications to the Ising model and random matrices.

\begin{itemize}[leftmargin=0.5cm]
  \item[1.] \label{ex:fields1}Let $U\subset \R^2$ be a bounded simply connected domain. Then the Gaussian free field on $U$ with zero boundary conditions is the $\mathcal{D}'(\reals^2)$-valued Gaussian random field with covariance 
\begin{equation}\label{eq:GFFcovariance}
    C_X(x,y)=G_U(x,y)=\log \left|\frac{1 - \varphi(x) \overline{\varphi(y)}}{\varphi(x) - \varphi(y)}\right|,
\end{equation}

    \noindent where $G_U$ is the Green's function of the Laplacian in $U$ with zero Dirichlet boundary conditions, and $\varphi:U\to \D$ is any conformal bijection. We could equivalently write $G_U(x,y)=\log \left|\frac{\psi(x) -\overline{\psi(y)}}{\psi(x) - \psi(y)}\right|$, where now $\psi:U\to \mathbb{H}^+$ is any conformal bijection  from $U$ to the upper half-pane. The generalized Karhunen--Lo\`{e}ve expansion obtained in Proposition~\ref{prop:karhunen_loeve} lets us write
    \[X(x)=\sum_{k=1}^\infty\frac{1}{\sqrt{\lambda_k}}A_k \varphi_k(x)\]
    with convergence in $H^{-\varepsilon}(\reals^d)$ in the norm-topology. Here $(\lambda_k)_{k=1}^\infty$ are the eigenvalues of $-\Delta$, $\varphi_k$ the associated eigenfunctions with unit $L^2$-norm (interpreted as zero outside of $U$), and $(A_k)_{k=1}^\infty$ i.i.d. standard Gaussians. 
    
 The   covariance given by the Green's function $G_U$ satisfies condition \eqref{eq:assumptions} which may seen by applying the standard comparison    $0\leq G_U(z,w)\leq G_{U'}(z,w)$, where $U'\supset U $ is any larger simply connected domain and $z,w\in U$. The integrability and the needed upper bound are  obtained via this inequality by picking a ball $B$ such that $U\subset B$ and setting $U'=2B$.
  \item[2.] The trace of the whole plane Gaussian free field on the unit circle $\mathbb{T}$ is the $\mathcal{D}'(\mathbb{T})$-valued Gaussian random variable with covariance 
$$
C_X(z,w)=-\log |z-w|
$$

    \noindent with $|z|=|w|=1$. Again $X$ can be expressed in terms of a sum. Let $(W_k)_{k=1}^\infty$ be i.i.d. standard complex Gaussian random variables, i.e. $W_k = \frac{1}{\sqrt{2}} A_k + i \frac{1}{\sqrt{2}} B_k$ with $A_k,B_k \sim N(0,1)$ and i.i.d.. Then one has
$$
X(z)=\sqrt{2}\mathrm{Re}\sum_{k=1}^\infty \frac{1}{\sqrt{k}}z^kW_k,
$$

    \noindent where the sum converges pointwise almost surely in $\mathcal{D}'(\mathbb{T})$ (again actually in $H^{-\varepsilon}(\mathbb{T})$ with respect to the norm topology for any $\varepsilon>0$).
    
While the unit circle $\mathbb T$ is not an open subset of $\R^d$, we can say write $z=e^{ix}$ and take $x\in(-\pi,\pi)$ or something similar and see that the conditions \eqref{eq:cov} and \eqref{eq:assumptions} can be verified with various interpretations.    \hfill $\blacksquare$
\end{itemize} }}
\end{example}

As $X$ is a random generalized function and not an honest function, we need to define the exponential $e^{i\beta X}$ in terms of a renormalization procedure, where we smooth $X$ into a function, exponentiate and then remove the smoothing. We will require our smoothing to have  particular properties that are usually satisfied by  most  natural approximations of log-correlated fields (and are typical  in the general theory of multiplicative chaos). We will call this type of an approximation  a \emph{standard approximation}:

\begin{definition}[Standard approximation]\label{def:standard}
    Let the covariance $C_X$ be as in \eqref{eq:cov} and \eqref{eq:assumptions}. We say that a sequence $(X_n)_{n\geq1}$ of continuous jointly Gaussian centered fields on $U$ is a standard approximation of $X$ if it satisfies:

\begin{itemize}[leftmargin=0.5cm]
\item[(i)] One has
$$
\lim_{(m,n)\to \infty}\E X_m(x)X_n(y)=C_X(x,y),
$$

\noindent where convergence is in measure with respect to the Lebesgue measure on $U\times U$. 

\item[(ii)] There exists a sequence $(c_n)_{n=1}^\infty$ such that $c_1\geq c_2\geq ...>0$, $\lim_{n\to\infty}c_n=0$,  and for every compact $K\subset U$
$$
\sup_{n\geq 1}\sup_{x,y\in K}\left|\E X_n(x)X_n(y)-\log \frac{1}{\max(c_n,|x-y|)}\right|<\infty.
$$

\item[(iii)] We have 
$$
\sup_{n\geq 1}\sup_{x,y\in U}\left[\E X_n(x)X_n(y)-\log\frac{1}{|x-y|}\right]<\infty.
$$ \hfill $\blacksquare$
\end{itemize}
\end{definition}

\smallskip

There can of course be various standard approximations. For example, one can check that for the GFF restricted to the unit circle from Example \ref{ex:fields}, one could take $X_n$ to be the truncation of the sum at $k=n$ -- see Example~\ref{ex:fourier_standard}. Perhaps the  most important class is provided by the usual mollifications of the field:

\begin{lemma}\label{le:standard}
Let $X$ be as in Proposition \ref{prop:karhunen_loeve}, and let $\eta\in C_c^\infty(\R^d)$ be non-negative, radially symmetric, with unit mass$:$ $\int_{\R^d}\eta(x)dx=1$, and with support ${\rm supp}(\eta)\subset B(0,1).$  For $x\in U$, $y\in \R^d$, and $\varepsilon>0$ define $\eta_{\varepsilon}(y)=\varepsilon^{-d}\eta(y/\varepsilon)$ and  set $X_{\varepsilon}(x):=X*\eta_\varepsilon(x)\times \1_U(x)$ for $x\in \reals^d$.\footnote{Recall that $X\in H^{-s}(\R^d)\subset \mathcal S'(\R^d)$ for any $s>0$, so as $\eta_\varepsilon\in \mathcal S(\R^d)$, this convolution makes sense.}

Let $K\subset U$ be a compact set, $0<\varepsilon<\delta$, and $x,y\in U$. We then have the estimates 
\begin{equation}\label{eq:conv1}
\sup_{0<\varepsilon<\delta<1}\sup_{x,y\in K}\left|\E X_{\varepsilon}(x)X_{\delta}(y)-\log \frac{1}{\max(|x-y|,\delta)}\right|<\infty,
\end{equation}
\begin{equation}\label{eq:conv2}
  \lim_{\delta\to 0}\E X_{\varepsilon}(x)X_{\delta}(y)=C_X(x,y) \qquad \text{for } x\neq y \text{ fixed},
\end{equation}
\begin{equation}\label{eq:conv3}
\sup_{\varepsilon>0}\sup_{x,y\in U}\left[\E X_{\varepsilon}(x)X_{\varepsilon}(y)-\log \frac{1}{|x-y|}\right]<\infty,
\end{equation}

\noindent and finally there exists a constant $C>0$ depending only on $K$ and $\eta$ so that for $x,y\in K$ 
\begin{equation}\label{eq:conv4}
\E (X_{\varepsilon}(x)-X_{\varepsilon}(y))^2\leq C|x-y| \varepsilon^{-1}.
\end{equation}
Especially, for any sequence $\delta_n\searrow 0$ the convolutions $X_{\delta_n}$, $n\geq 1$, provide a standard approximation.
\end{lemma}

\begin{proof} We begin with the proof of \eqref{eq:conv1} and observe that by definition
\begin{equation}\label{eq:eka}
\E X_{\varepsilon}(x)X_{\delta}(y) = \big((\eta_\varepsilon\otimes\eta_\delta)* C_X\big) (x,y)\mathbf{1}_{U\times U}(x,y).
\end{equation}
  Note that by our definition, $C_X$ is extended to be zero outside $U\times U$, and $C_X$ is integrable (actually belongs to all $L^p$-spaces by Lemma \ref{le:integrability}), so the convolution is well-defined in all of $\R^{2d}$. In turn, the factor $\mathbf{1}_{U\times U}$ verifies that the approximations are supported on $U$. 
Pick an open set $V$ such that $K\subset V\subset \overline{V}\subset U.$ Denote $a=a_K:={\rm dist}(K,\partial V)>0$. 
Locally the function $g$ is bounded uniformly from above and below on $\overline{V}$ by the assumed continuity, so its contribution to the convolution \eqref{eq:eka} is also uniformly bounded if $x,y\in K$ and $\varepsilon,\delta\leq a$. For other values of $\delta,\varepsilon$ the contribution of $g$ is upper bounded by   $\lesssim a^{-2d}$ from the integrability of $g$. Hence it remains to verify \eqref{eq:conv1} just  for the logarithmic term.

As the logarithmic term depends only on the difference $x-y$ we may write
\begin{equation}\label{eq:toka}
 \big((\eta_\varepsilon\otimes\eta_\delta)* \log (|\cdot -\cdot |^{-1})\big)(x,y)= \big((\eta_\varepsilon* \eta_\delta)*\log(|\cdot |^{-1})\big)(x-y).
\end{equation}
Given any differentiable function $h:\R^d\to \R$ we have the easy estimate
\begin{equation}\label{eq:apu}
\|\eta_\varepsilon*h-h\|_{L^\infty (B(x,r-\varepsilon))}\lesssim \varepsilon\|Dh\|_{L^\infty (B(x,r))}
\end{equation}
for any $0<\varepsilon<r$ and $x\in\R^d$.
Let us denote $H:=\eta_1*\log(1/|\cdot|)$. As a smooth function $H$ is uniformly bounded near the origin. Moreover, $|D\log (1/|x|)| \leq 1$ for $|x|\ge 1$, whence \eqref{eq:apu} yields that $|H(x)-\log(1/|x|)|\leq C$ for
$|x|\geq 1$. These observation may be combined as follows:
\begin{equation}\label{eq:apu1}
  \sup_{x\in \R^d}\left| H(x)- \log (1\wedge |x|^{-1})\right| \le C.
\end{equation}
Using the smoothness of $H$ and again the bound $|D\log (1/|x|)| \leq 1$ for $|x|\geq 1$, we see that $|DH|$ is uniformly bounded in $\R^d$, and hence
\eqref{eq:apu} implies  the inequality
$\|\eta_\varepsilon*H -H\|_{L^\infty (\R^d)}<C$ uniformly in   $\varepsilon\in (0,1)$. Putting things together we have shown that
$$
\Big|\big((\eta_1* \eta_\varepsilon)*\log(|\cdot |^{-1})(x)- \log (1\wedge |x|^{-1})\Big| \leq C \qquad \textrm{for all}\quad \varepsilon\in (0,1)\;\;\textrm{and}\;\;  x\in \R^d.
$$
This is \eqref{eq:conv1} for $1=\delta >\varepsilon >0$, and scaling yields the general case
\begin{equation}\label{eq:stand_gen}
  \big|(\eta_\varepsilon * \eta_\delta)*\log(|\cdot|^{-1})(x) - \log\big(\frac{1}{\varepsilon \vee \delta \vee |x|}\big)\big| \le C.
\end{equation}

  The convergence in \eqref{eq:conv2} is immediate from standard properties of convolution and  the continuity of $C_X$ outside the diagonal.  Next,  \eqref{eq:conv3} follows  from \eqref{eq:eka}, \eqref{eq:toka}, \eqref{eq:stand_gen} and the upper boundedness of $g$. Finally, for  \eqref{eq:conv4} we may clearly assume  that $\varepsilon\leq a/2$ (where $a$ depends on $K$ as was defined in the beginning of the proof) and that $g$ is continued as a uniformly bounded measurable function to the whole of $\R^d$ (the extension need not to be a covariance). For  \eqref{eq:conv4}  it is enough to prove the derivative bounds $|D_x C_{X_\varepsilon}|,|D_y C_{X_\varepsilon}| \lesssim \varepsilon^{-1}$. Since  $\int_{\R^d} |D\eta_\varepsilon|\lesssim \varepsilon^{-1}$, we obtain the stated bounds for the contribution of $g$ to the derivative.  In turn, for the contribution of the logarithm one assumes first that $\varepsilon =1$. Then the uniform boundedness of the derivatives follow from \eqref{eq:toka} and the fact that $\|DH\|_\infty <\infty$, where $H$ is as before. The case of general $\varepsilon \in (0,1)$ is again obtained by scaling.

Finally we note that conditions (i), (ii), and (iii) of a standard approximation follow from \eqref{eq:conv1}, \eqref{eq:conv2}, and \eqref{eq:conv3}. Thus we only need to check that $(X_{\delta_n})$ are jointly Gaussian and continuous. We recall the simple argument for the convenience of a reader unfamiliar with such matters. By construction, all of the processes $(n,x)\mapsto X_{\delta_n}(x)$ live on the same probability space. Moreover, for any fixed $N\in \Z_+$, $x_1,...,x_N\in U$, $n_1,...,n_N\in \Z_+$, and $t_1,...,t_N\in \R$,

$$
\sum_{k=1}^N t_k X_{\delta_{n_k}}(x_k)=X\left(\sum_{k=1}^N t_k\eta_{\delta_{n_k}}(\cdot-x_k)\right)
$$

\noindent and as we have e.g. $\sum_{k=1}^N t_k\eta_{\delta_{n_k}}(\cdot-x_k)\in H^\varepsilon(\R^d)$, this is a Gaussian random variable by definition, so indeed we have joint Gaussianity. Finally  continuity follows by observing that $\eta_{\delta}(\cdot-x')\to \eta_{\delta}(\cdot-x)$ in $H^\varepsilon(\R^d)$ as $x'\to x$ and using the duality between $H^{-\varepsilon}$ and $H^\varepsilon$.
\end{proof}
The proof of this result can be used to prove that other natural approximations are also standard approximations. As an example, we give the following one.

\begin{example}\label{ex:fourier_standard}
  Let $X_n(x) = \sqrt{2} \mathrm{Re}\sum_{k=1}^n\frac{1}{\sqrt{k}}e^{ikx}W_k$, where $W_k$ are as in Example~\ref{ex:fields} part 2.
  Then the sequence $(X_n)_{n\ge 1}$ forms a standard approximation. Intuitively, this follows, since for the approximation 
    $$
    \widetilde X_n(x) :=  \sqrt{2} \mathrm{Re}\left(\sum_{k=1}^{n} \frac{\sqrt{n-k}}{\sqrt{nk}}e^{ikx}W_k  \right)
  $$
  we have $\E \widetilde X_n(x) \widetilde X_n(y) = \sum_{k=1}^n \frac{n-k}{nk}\cos(k(x-y))$. 
    The last written sum is a convolution of the logarithmic  kernel with a standard Fej\'er kernel,  and the difference between the Fej\'er partial sum and Fourier partial sum is uniformly bounded by direct inspection. Finally, the Fej\'er partial sum of the logarithm is essentially a convolution approximation which behaves like the covariance of a standard approximation by the proof of Lemma \ref{le:standard}. For a detailed argument, see e.g. the beginning of the proof of Lemma~6.5 in \cite{JS}.
  \end{example}

To conclude this preliminary section, we discuss briefly the spaces of generalized functions that we will discuss in this article.

\subsection{Classical function spaces }\label{sec:spaces}

Realizations of the  imaginary chaos that we define in the next section  are rather singular objects and one can't have convergence in any space of honest functions or even complex measures, so we must study convergence in suitable spaces of distributions. In fact this holds even true for log-correlated fields that were defined in the previous subsection, and therein we used  the basic negative index Sobolev Hilbert spaces as a suitable tool. Here we recall for the convenience of readers less familiar with various spaces of generalized functions the definition of Sobolev spaces as well as of the other function spaces we use in the article.  

For any smoothness index $s\in \R$ we define
\begin{equation}\label{eq:sobo}
  H^{s}(\R^d)=\left\lbrace \varphi\in \mathcal{S}'(\R^d): \|\varphi\|_{H^s(\R^d)}^2=\int_{\R^d}(1+|\xi|^2)^s \big|\widehat{\varphi}(\xi)\big|^2 \, d\xi <\infty\right\rbrace,
\end{equation}
where $\widehat{\varphi}$ stands for the Fourier transform of the tempered distribution $\varphi$ -- our convention for the Fourier transform is 
$$
\widehat{\varphi}(\xi)=\int_{\R^d}e^{-2\pi i \xi \cdot x}\varphi(x)dx
$$
for any Schwartz function $\varphi\in \mathcal{S}(\R^d)$. Some basic facts about the spaces $H^s(\R^d)$ are e.g. that they are Hilbert spaces,  for $s>0$, $H^{-s}(\R^d)$ is the dual of $H^s(\R^d)$ with respect to the standard dual pairing, $H^s(\R^d)$ is a subspace of $C_0(\R^d)$ for $s>d/2$, i.e. there is a continuous embedding into the space of continuous functions vanishing at infinity, and for $s<-d/2$, compactly supported Borel measures (especially $\delta$-masses) are elements of $H^{s}(\R^d)$. 

A more extensive scale of measuring the simultaneous size and smoothness properties of functions is provided by Besov spaces on $\R^d$. In order to recall their definition, fix radial and non-negative Schwartz test functions $\phi_0,\phi_1\in \mathcal S(\R^d)$, denote  $\phi_k(x):=2^{kd}\phi_1(2^kx)$ and assume that 
$$
\textrm{ supp}(\widehat \phi_0)\subset B(0,2),\qquad \textrm{supp}(\widehat \phi_1)\subset B(0,4)\setminus B(0,1),
$$
  together with the partition of unity property
$
\sum_{k=0}^\infty \widehat \phi_k(\xi )=1\quad \textrm{for all}\quad \xi\in \R^d.
$ 
Assume that $1\leq p,q\leq\infty.$ A function (or Schwartz distribution) $f$ on $\R^d$ belongs to the Besov space $B^{s}_{p,q}(\R^d)$ if
\begin{equation}\label{eq:besov}
\| f\|_{B^{s}_{p,q}(\R^d)}:=\left(\sum_{k=0}^\infty 2^{qks}\|\phi_k*f\|^q_{L^p(\R^d)}\right)^{1/q} <\infty,
\end{equation}
where the interpretation for $q=\infty$ is $\| f\|_{B^{s}_{p,q}}:=\sup_{k\geq 0}2^{ks}\|\phi_k*f\|_{L^p(\R^d)}$.
These spaces include many standard spaces. First of all,  $B^{s}_{2,2}(\R^d)=W^{s,2}(\R^d)=H^s(\R^d).$ Moreover, if $s\in (0,1)$ we have $B^{s}_{\infty,\infty}(\R^d)=C^s(\R^d)$ (with equivalent norms), where $C^s$ is the well-known space of bounded H\"older continuous functions with the norm
$$
\|f\|_{C^s(\R^d)}:=\| f\|_{L^\infty (\R^d)}+\sup_{x,y\in \R^d}\frac{|f(x)-f(y)|}{|x-y|^s}.
$$
Indeed, as is standard in harmonic analysis,  one defines $C^{s}(\R^d):=B^{s}_{\infty,\infty}(\R^d)$ for arbitrary $s\in\R$.

Our motivation for proving in this paper basically optimal results for  membership of the imaginary chaos in general Besov spaces comes from the fact that this yields considerably more knowledge on the smoothness and size of these objects than is obtained by just using the spaces $H^s(\R^d)$. Recall for example, that in the setting of log-correlated fields, our Proposition \ref{prop:karhunen_loeve} said that the field $X$, if smoothed a little bit, becomes an $L^2$-function, which is far weaker than saying that it becomes continuous as was stated in Lemma \ref{le:tarkempi}. The latter result indeed measures smoothness using the Besov scale $B^s_{\infty,\infty},$ i.e. H\"older-spaces. 

Another scale of function spaces is provided by the Triebel--Lizorkin spaces $F^{s}_{p,q}(\R^d)$, where we assume that
$1\leq p,q<\infty$ and set
$$
\| f\|_{F^{s}_{p,q}(\R^d)}:=\Big\| \left(\sum_{k=0}^\infty 2^{qks}|\phi_k*f|^q\right)^{1/q}\Big\|_{L^p(\R^d)}.
$$
This space contains as special cases e.g. the general Sobolev spaces  $W^{k,p}(\R^d)=F^{k}_{p,2}(\R^d).$
However, we do not need to know more of them, since we will transfer our smoothness  results from the Besov case to the Triebel--Lizorkin scale in view of  the simple embeddings
\begin{equation}\label{eq:betrie}
B^{s+\delta}_{p,p}(\R^d)\subset F^{s}_{p,q}(\R^d)\subset B^{s-\delta}_{p,p}(\R^d)
\end{equation}
which hold for any $\delta >0$,  all $1\leq p,q<\infty$ and $s\in\R.$ This  is easily shown  from the very definitions of the spaces. For example, by H\"older's inequality we have for any sequence $(a_k)_{k\geq 1}$ and $\delta >0$  that
$\|(a_k)_{k\geq 1}\|_{\ell^{q'}}\lesssim \|(2^{k\delta}a_k)_{k\geq 1}\|_{\ell^{q}}$ for any $q,q'\in [1,\infty]$. This shows that $F^{s}_{p,q}(\R^d)\subset F^{s-\delta}_{p,q'}(\R^d)$ for any $q,q'$. By choosing $q'=p$ and noting that $\|f\|_{F_{p,p}^s(\R^d)}=\|f\|_{B_{p,p}^s(\R^d)}$, we obtain the right hand inequality in \eqref{eq:betrie}, and the other one is proven in a similar way.

We need a couple of additional facts about Besov spaces. Fix $K\subset \R^d$ compact. Then for a distribution $f$ in $\R^d$ with support contained in $K$ we have also (now for  the full range $1\leq p,q\leq\infty$)
\begin{equation}\label{eq:bebe}
\|f\|_{B^{s}_{\infty,\infty }(\R^d)}\; {\lesssim}\;   \|f\|_{B^{s'}_{p,p}(\R^d)} \quad \textrm{if}\quad s'\geq s+\frac{d}{p}.
\end{equation}
and
\begin{equation}\label{eq:beem}
\|f\|_{B^{s-\delta}_{1,1}(\R^d)} \; {\lesssim}\; \|f\|_{B^{s}_{p,q}(\R^d)}\; {\lesssim}\; \|f\|_{B^{s+\delta}_{\infty,\infty}(\R^d)} .
\end{equation}
with the implied constants in \eqref{eq:beem} possibly depending on $K$. \eqref{eq:bebe} is found in \cite[Section 2.7.1.]{Tr}, and \eqref{eq:beem} follows by combining the reasoning from the end of the last paragraph with a standard expression for the Besov-norm using wavelets -- see \cite[Chapter 6]{M}. One finally uses the simple fact that for functions $f$ supported in a compact set $K'$ we have by H\"older's inequality that  $\|f\|_{L^{p_1}}\lesssim \|f\|_{L^{p_2}}$ for $1\leq p_1\leq p_2\leq \infty$.

For a subdomain $U\subset\R^d$ (naturally one may have $U=\R^d$) one says that a distribution $\lambda\in \mathcal{D}'(U)$ lies in the space $H^s_{loc}(U)$ if for all $\psi\in C_c^\infty(U)$ one has $\psi\lambda\in H^s(\R^d)$.  In turn, one says that $\lambda\in H^s(U)$ assuming that there is $f\in H^s(\R^d)$ such that $\lambda=f_{|U}$ (then one defines $\|\lambda\|_{H^s(U)}:=\inf\{ \|f\|_{H^s(\R^d)}\;|\;  \lambda=f_{|U}\}$). Similar conventions are used for other function spaces defined initially on $\R^d.$

One final general fact about the function spaces we will use is the standard $\delta$-lift $I^\delta f$ (``smoothing by an amount $\delta$'') of a given $f\in\mathcal{S}'(\R^d)$, which for any fixed $\delta\in\R$ is defined by using the Fourier-transform as follows:
\begin{equation}\label{eq:lift}
I^\delta f :=\mathcal{F}^{-1}\left((1+|\cdot |^2)^{-\delta/2}\widehat f\right)= G_\delta*f,
\end{equation}
where $G_\delta$ is the Bessel potential kernel. For any $\delta,s\in\R$ and $p,q\in [1,\infty]$ the map $I^\delta:
B^{s}_{p,q}(\R^d)\to B^{s+\delta}_{p,q}(\R^d)$ is a continuous, linear and bijective isomorphism (see \cite[Section~2.3.8]{Tr}).

For an  introduction to the basic properties of  the $L^2$-Sobolev spaces, as well as for the Besov and Triebel spaces we refer in general to \cite[Chapter 2]{Gr2}, \cite{Tr}, \cite{M}.

This concludes our preliminary discussion about log-correlated fields and spaces of generalized functions. We will now move onto imaginary chaos.

\section{Basic properties of imaginary multiplicative chaos}\label{sec:ichaos}

In this section, we prove our results stated in Section \ref{subsec:ichaosresults} concerning basic properties of imaginary multiplicative chaos as well as prove some auxiliary ones. We begin with Section \ref{subsec:construction} where we construct our imaginary multiplicative chaos and give some uniqueness results. In Section  \ref{subsec:moments}, we discuss stochastic properties of imaginary multiplicative chaos, namely we provide some general moment estimates, based on a generalisation of so-called  Onsager type (electrostatic) inequalities  (they will be discussed in Subsection \ref{subsec:moments} below) for general covariances with a logarithmic singularity on the diagonal. These are used to obtain uniqueness statements in terms of moments  and tail estimates for the law of the imaginary chaos tested against a given test function. We then move on to proving basic estimates for the regularity of imaginary chaos in Section \ref{subsec:regularity}. Section \ref{subsec:universality} verifies that in the definition of ``$e^{i\beta X}$'' there is a lot of freedom in replacing $x\mapsto e^{ix}$ by another periodic function. Finally, in Section  \ref{subsec:critical} we investigate what happens in the limit $\beta\nearrow \beta_{\rm crit}=\sqrt{d}$. It is known from \cite{LRV} that $\beta_{\rm crit}$ is the critical value for $\beta$ beyond which the naive renormalization scheme of dividing $e^{i\beta X_n(x)}$ by $\E e^{i\beta X_n(x)}$ does not produce a non-trivial limiting object, and our Theorem \ref{thm:crit} gives another manifestation of this fact.

\subsection{Construction of imaginary chaos}\label{subsec:construction} We begin by constructing imaginary multiplicative chaos and verifying some uniqueness properties, namely that the constructed object does not depend very much on the approximation used -- see Theorem \ref{th:existuniq}. Before starting, we recall that under a slightly more restrictive class of  covariances $C_X$, the existence of the object follows already from results in \cite{LRV}, where complex multiplicative chaos was studied, but we offer a simple alternative proof here. We also mention that if one were to work for example in the class of tempered distributions, proving existence would be slightly simpler, but this would give very little insight into the regularity of these objects.

Let us start by proving existence. In our approach we are given a sequence of approximations $(X_n)_{n\geq 1}$ of the log-correlated field $X$ on the domain $U$, which we use to define what we hope are approximations to our multiplicative chaos distribution: 
\[\mu_n(x) := \exp\Big(\frac{\beta^2}{2}\E [X_n(x)^2] +i\beta X_n(x)\Big)\mathbf{1}_U(x).\]
We will first prove the convergence of $\mu_n$ in a suitable Sobolev space, assuming that $X_n$ forms a standard approximation sequence as in Definition~\ref{def:standard}. As we will see in Section \ref{subsec:regularity}, the smoothness index we obtain here is not optimal, but we will return to finer regularity properties later. We also mention here that as follows from \cite[Theorem 4.2]{LRV} (under slightly more restrictive assumptions on $g$), one should not expect that $\mu_n$ has a limit for $\beta\geq \sqrt{d}$ unless it is multiplied by a suitable quantity tending to zero, in which case the limit should be proportional to white noise. As this is perhaps not as interesting a limiting object, we choose to focus on the regime $0<\beta<\sqrt{d}$. The following proposition is the first ingredient of Theorem \ref{th:existuniq}.

\begin{proposition}\label{prop:mu_convergence}
  Let  $(X_n)_{n\geq 1}$ be a standard approximation of a given log-correlated field $X$ on a domain $U$ $($see Definition~\ref{def:standard}$)$.  When $0<\beta <\sqrt{d}$, the functions $\mu_n$ converge in probability in $H^s(\reals^d)$ for $s < -\frac{d}{2}$. The limit $\mu$ is a non-trivial random element of $H^s(\R^d)$, supported on $\overline{U}.$
\end{proposition}

\begin{proof}
  Assume first that $\varphi \in L^\infty(\reals^d)$ is positive and let us write $C_{n,m}(x,y)=\E X_n(x)X_m(y)$, whence we have $C_{n,m}(x,y)=C_{m,n}(y,x).$ 
  Then a short calculation shows that
  \begin{align*}
    & \E |\mu_m(\varphi) - \mu_n(\varphi)|^2 
    \; =\;  \int_U \int_U \varphi(x) \varphi(y) \Big( e^{\beta^2 C_{n,n}(x,y)} + e^{\beta^2 C_{m,m}(x,y)} - 2e^{\beta^2 C_{n,m}(x,y)} \Big) \, dx \, dy.
  \end{align*}
  By (iii) of Definition \ref{def:standard}, we have $e^{\beta^2 C_{n,n}(x,y)}=\mathcal{O}( |x-y|^{-\beta^2})$, where the implied constant is independent of $x,y,n$. Note that as $\beta^2<d$, $|x-y|^{-\beta^2}$ is an integrable singularity (this is the role the $0<\beta<\sqrt{d}$ condition plays). Thus by the dominated convergence theorem, 
    \begin{align*}
    0 & \le \limsup_{(n,m) \to \infty} \E |\mu_m(\varphi) - \mu_n(\varphi)|^2 
   \; =\;  \int_U \int_U \varphi(x) \varphi(y) (e^{\beta^2 C_X(x,y)} + e^{\beta^2 C_X(x,y)}) \, dx \, dy \\
    &\qquad\; - \; \liminf_{(n,m) \to \infty} \int_U \int_U \varphi(x) \varphi(y) (e^{\beta^2 C_{n,m}(x,y)} + e^{\beta^2 C_{m,n}(x,y)}) \, dx \, dy \;\le \;0,
  \end{align*}
  where the last inequality follows by Fatou's lemma and property (i) of Definition \ref{def:standard}. Thus we get
  \[\lim_{(n,m) \to \infty} \E |\mu_m(\varphi) - \mu_n(\varphi)|^2 = 0,\]
  implying that $\mu_n(\varphi)$ is a Cauchy sequence in $L^2(\P)$.
  Moreover, by property (iii) of Definition \ref{def:standard}, we have the simple upper bound
  \[\E |\mu_m(\varphi) - \mu_n(\varphi)|^2 \le \|\varphi\|_\infty^2 \int_U \int_U  (e^{\beta^2 C_{n,n}(x,y)} + e^{\beta^2 C_{m,m}(x,y)}) \, dx \, dy \le C \|\varphi\|_\infty^2\]
  for some constant $C > 0$.
  By splitting a complex valued $\varphi$ into positive and negative real and imaginary parts we get the convergence in $L^2(\P)$  of $\mu_n(\varphi)$ for all $\varphi \in L^\infty(\reals^d)$\footnote{Note that this result is essentially  enough to ensure the existence of say a random tempered distribution $\mu_n$ converges to, but as stated before, it gives very little insight into the regularity of the object. Hence we work a bit harder to prove convergence in a Sobolev space, and later to extract the optimal regularity.}, as well as the upper bound
  \begin{equation}\label{eq:321}
  \E |\mu_m(\varphi) - \mu_n(\varphi)|^2 \le 16 C \|\varphi\|_\infty^2.
  \end{equation}

  We next compute 
  \begin{align*}
    \E \|\mu_m - \mu_n\|^2_{H^s} & = \int_{\reals^d} (1 + |\xi|^2)^s \E |\widehat{\mu_m}(\xi) - \widehat{\mu_n}(\xi)|^2 \, d\xi \\
    & = \int_{\reals^d} (1 + |\xi|^2)^s \E |\mu_m(e^{-2\pi i \xi \cdot}) - \mu_n(e^{-2\pi i \xi \cdot})|^2 \, d\xi.
  \end{align*}
  Notice that if $s < -\frac{d}{2}$, then the estimate \eqref{eq:321} and the dominated convergence theorem show us that as elements of $H^s$, the sequence is Cauchy in $L^2(\P)$.
  Thus there exists a random element of $H^s$, say $\mu$, living on the same probability space as our approximations, and satisfying $\E\|\mu\|_{H^s}^2<\infty$ as well as $\lim_{n\to\infty}\E\|\mu_n-\mu\|_{H^s}^2=0$.
  In particular this implies convergence in probability in $H^s$ of $\mu_n$ to $\mu$.
  
  Non-triviality of $\mu$ follows from $L^2$-convergence: one has e.g. 
  
  $$\E |\mu(\varphi)|^2=\int_{U\times U}\varphi(x)\varphi(y)e^{\beta^2 g(x,y)}|x-y|^{-\beta^2}dxdy.$$
  Finally, the claim of the support is evident since all the approximations $\mu_n$ are supported on $\overline{U}$ by definition.
\end{proof}

Having proven that limiting objects exist, the next natural step is to check that the limit $\mu$ does not depend on our approximating sequence $\mu_n$ in some sense. There are various statements of this flavor one could formulate; one example being that the law of the limit would be independent of the standard approximation. We return to such a question later with moments and now show with a simple argument that if there are two standard approximations living on the same probability space and are compatible in a certain way, then they converge in probability to the same random variable. The next proposition is the uniqueness portion of Theorem \ref{th:existuniq}.

\begin{proposition}\label{prop:uniqueness}
  Suppose that $X_n$ and $\widetilde{X}_n$ are two jointly Gaussian sequences of standard approximations of the same log-correlated field $X$  and that
  \[\lim_{n\to\infty} \E X_n(x) \widetilde{X}_n(y) = C_X(x,y),\]
  where the convergence takes place in measure on $U\times U.$
  Then the corresponding imaginary chaoses $\mu$ and $\widetilde{\mu}$ are equal almost surely.
\end{proposition}

\begin{proof}
  It is enough to show that for all $f \in C_c^\infty(\R^d)$ we have
  \[\lim_{n \to \infty} \E |\mu_n(f) - \widetilde{\mu}_n(f)|^2 = 0.\]
  A straightforward computation shows that the expectation equals
  \[\int_U \int_U f(x) \overline{f(y)} \big( e^{\beta^2 \E X_n(x) X_n(y)} + e^{\beta^2 \E \widetilde{X}_n(x) \widetilde{X}_n(y)} - e^{\beta^2 \E X_n(x) \widetilde{X}_n(y)} - e^{\beta^2 \E \widetilde{X}_n(x) X_n(y)} \big) \, dx \, dy.\]
  Notice that since $X_n$ and $\widetilde{X}_n$ are standard approximations, there exists a constant $c > 0$ such that on $U\times U$
  \[e^{\beta^2 \E X_n(x) X_n(y)} + e^{\beta^2 \E \widetilde{X}_n(x) \widetilde{X}_n(y)} \le \frac{c}{|x-y|^{\beta^2}}.\]
  Thus by the reverse Fatou lemma we have
  \begin{align*}
     \limsup_{n \to \infty} \E |\mu_n(f) - \widetilde{\mu}_n(f)|^2 &\le  \int_U \int_U f(x)f(y)\limsup_{n \to \infty} \big( e^{\beta^2 \E X_n(x) X_n(y)} + e^{\beta^2 \E \widetilde{X}_n(x) \widetilde{X}_n(y)}\\
    &\qquad  \qquad - e^{\beta^2 \E X_n(x) \widetilde{X}_n(y)} - e^{\beta^2 \E \widetilde{X}_n(x) X_n(y)} \big) \, dx \, dy\\
    &=\;0. \qedhere
  \end{align*}
\end{proof}
By combining Propositions \ref{prop:mu_convergence} and \ref{prop:uniqueness} we conclude the proof of Theorem \ref{th:existuniq}.

\begin{remark}\label{rem:1} Given a log correlated field $X$ as in Proposition \ref{prop:karhunen_loeve} and $\beta\in (0,\sqrt{d})$, when  we speak of the  imaginary chaos $\mu= ``\exp(i\beta X)"$ we mean the chaos defined via Proposition \ref{prop:mu_convergence} using convolution approximations. The definition is well-posed since convolution approximations yield a standard approximation according to Lemma \ref{le:standard}, and the outcome does not depend on the approximation used as one may easily check that two different sequences of convolution approximations satisfy the conditions of Proposition \ref{prop:uniqueness}.  \hfill $\blacksquare$
\end{remark}

In our application to the Ising model, what will turn out to be important is the real part of imaginary chaos. We now define this properly.

\begin{definition}\label{def:kosini}
Given a log-correlated field $X$, satisfying our assumptions \eqref{eq:cov} and \eqref{eq:assumptions}, and $\beta\in (0,\sqrt{d})$ the cosine of $X$ (simply denoted by ``$\cos(\beta X)$'') is defined as the real part of the imaginary chaos, or in other words, for any test-function $\varphi\in C_c^\infty(\R^d)$ one has
$$
\langle \cos(\beta X),\varphi\rangle := \lim_{n\to\infty} \int_{U} e^{\frac{1}{2}\beta^2\E (X_n(x))^2}\cos(\beta X_n(x))\varphi(x)dx,
$$
where the limit is in probability, and $(X_n)_{n\geq 1}$ is a sequence of convolution approximations of $X$. \hfill $\blacksquare$
\end{definition}
\noindent The most important example of $``\cos(\beta X)"$ is the one corresponding to a Gaussian free field (GFF) on a given simply connected planar domain $U\subset \R^2, $ see the first part of Example \ref {ex:fields}. In Section \ref{subsec:moments} we shall characterise the laws of both $``\exp(i\beta X)"$ and $``\cos(\beta X)"$ via moments.

\smallskip

Before concluding this section about the existence and uniqueness of imaginary chaos, we mention that it is natural to ask whether the definition of the imaginary chaos could be done via the approximations given by the partial sums of the Karhunen--Lo\`{e}ve expansion \eqref{eq:def_X}:
\begin{equation}\label{eq:KLapp}
X_{KL,n}(x) := \sum_{k=1}^n A_k \sqrt{\lambda_k} \varphi_k(x).
\end{equation}
The benefit of such a definition would be that it would allow using powerful probabilistic tools such as martingale theory and the Kolmogorov 0--1 law, which sometimes simplify proofs significantly. Unfortunately, checking even the uniform integrability condition (iii) in Definition \ref{def:standard} appears to be quite complicated in the case of the Karhunen--Lo\`{e}ve approximations $X_{KL,n}(x)$, so we cannot refer to the above statements. However, under a mild further assumption, we will be able to settle the question by a more probabilistic  argument.

\begin{lemma}\label{le:KLdef} Assume that  $\beta\in (0,\sqrt{d})$ and that $X$ is the GFF on a bounded simply connected subdomain of $\C$, or more generally, that $X$ is a log-correlated field on a bounded domain in $\R^d$ with covariance satisfying our basic assumptions \eqref{eq:cov} and \eqref{eq:assumptions}, and the additional size-condition $\sup_{x\in U}\| g(x,\cdot)\|_{L^2(U)} <\infty.$ Denote $\nu_n(x):=\exp \big(\frac{1}{2}\beta^2\E [X_{KL,n}(x)^2]+i\beta X_{KL,n}(x)\big).$ As $n\to\infty,$  $\nu_n$ converges  to the imaginary chaos $\mu$ $($see Remark \ref{rem:1}$)$. More specifically, given
$\phi\in C_c^\infty (\R^d),$ we have as $N\to\infty$
$$
\langle \nu_n,\phi\rangle \to \langle \mu,\phi\rangle,
$$
where the convergence is almost sure. Moreover, $\nu_m\to\mu$ almost surely in the Sobolev space $H^s(\R^d)$ for any $s<-d/2.$
\end{lemma}

\begin{proof} We may assume that $X$ is given by the Karhunen--Lo\`{e}ve decomposition  \eqref{eq:def_X}. Let us denote
$$
Y_n:= \langle \nu_n,\phi\rangle=\int_U \exp \big(\frac{1}{2}\beta^2\E [X_{KL,n}(x)^2]+i\beta X_{KL,n}(x)\big)\phi(x)dx,
$$
whence $Y_n$ is a martingale by construction. Here the integral is well-defined since by Cauchy--Schwarz, the  condition $\sup_{x\in U}\| g(x,\cdot)\|_{L^2(U)} <\infty$ implies that each eigenfunction $\varphi_k$ (corresponding to a non-zero eigenvalue) belongs to $L^\infty (U).$  In order to prove convergence of $Y_n$ to something, as $n\to\infty$, the martingale structure implies that it is enough to verify that $Y_n$ is $L^2$-bounded. Denote by $X_{\delta_k}$ a standard convolution approximation and note that
since $X-X_{KL,n}\perp X_{KL,n}$ we may write $X_{\delta_k}= (X_{KL,n})_{\delta_k}+ (X-X_{KL,n})_{\delta_k},$ where the summands are independent. This implies that

\begin{equation}\label{eq:kaava}
\E \Big(\exp \big(\frac{1}{2}\beta^2\E [X_{\delta_k}(x)^2]+i\beta X_{\delta_k}(x)\big)\big| \mathcal{F}_n\Big)=
\exp \big(\frac{1}{2}\beta^2\E [(X_{KL,n})_{\delta_k}(x)^2]+i\beta (X_{KL,n})_{\delta_k}(x)\big),
\end{equation}
where $\mathcal{F}_n$ is the $\sigma$-algebra generated by $\{A_1,\ldots , A_n\}$, and $A_i$ are the i.i.d. standard Gaussians from \eqref{eq:KLapp}.  By basic  real analysis, as we are convolving $L^1$-functions with nice bump functions, there is a set $E\subset U$ of zero Lebesgue measure so that  we have $(\varphi_j)_{\delta_k}(x)\to\varphi_j(x)$ for each $j$ and  $x\in U\setminus E$. 
Hence, if we denote 
$$
 Y_{n,k}:= \langle \nu_n,\phi\rangle=\int_U \exp \big(\frac{1}{2}\beta^2\E [(X_{KL,n})_{\delta_k}(x)^2]+i\beta (X_{KL,n})_{\delta_k}(x)\big)\phi(x)dx,
$$
then we have
$
Y_{n,k}\to Y_n
$ almost surely as $k\to \infty$. By dominated convergence and \eqref{eq:kaava} it follows for every $n$ that if we write $\mu_k$ for the approximation to $\mu$ given by $X_{\delta_k}$, then 
$$
\E |Y_n|^2\leq \sup_{k} \E|Y_{n,k}|^2 =  \sup_{k} \E\left|\E \big(\langle \mu_k ,\phi\rangle\big|\mathcal{F}_n\big)\right|^2 \leq \sup_{k}\E[|\langle \mu_k ,\phi\rangle|^2] :=C <\infty,
$$
where the last inequality used again the uniform $L^2$-bound on approximations of $\mu$ coming from convolution approximations, which in turn followed from \eqref{eq:conv3}. Further, the above reasoning\footnote{More precisely: multiplying \eqref{eq:kaava} by $\phi(x)$, integrating over $U$, and letting $k\to\infty$, one sees that the left hand side of \eqref{eq:kaava} becomes $\E(\langle \mu,\phi\rangle|\mathcal F_n)$ -- this used the fact that $\mu_k\to \mu$ in $L^2$. On the other hand, before taking the $k\to\infty$ limit, the right hand side equals $Y_{n,k}$ and we saw that this tends to $Y_n$ as $k\to\infty$.}  also verifies that $Y_n=\E (\langle \mu ,\phi\rangle|\mathcal{F}_n)$. Here both sides converge almost surely by the martingale property and $L^2$-boundedness, and the right hand side converges to $\langle \mu ,\phi\rangle$ simply by the fact that $\langle \mu ,\phi\rangle$ is measurable with respect to the $\sigma$-algebra $\sigma (\cup_{j=1}^\infty \mathcal{F}_j).$

The stated convergence in the Sobolev space  now follows  since the above reasoning yields the uniform estimate $\E |Y_n|^2\leq c \|\phi\|_\infty^2$, which leads to $\nu_n$ being a $L^2$-bounded $H^s$-valued martingale. Finally, the GFF on a bounded planar domain $U\subset\C$ satisfies the extra size condition as we then have $0\leq C_X(z,w)\leq c+\log(1/|z-w|)$ for any $z,w\in U.$
\end{proof}

This concludes our basic discussion about existence and uniqueness of imaginary chaos, and we move onto discussing probabilistic properties of imaginary chaos.

\subsection{Moment and tail bounds}\label{subsec:moments} In this section we will prove moment and tail bounds for imaginary chaos, namely Theorem \ref{th:moments} and Theorem \ref{th:tails}.
The situation is quite different from real chaos (or complex chaos in general),
since, as we will see in this section,  for $\mu$ from Theorem \ref{th:existuniq},  the moments $\E |\mu(f)|^{2N}$ are finite
for all $N \ge 1$ and all $f\in C_c^{\infty}(U)$. Moreover, it will turn out that (under minor smoothness assumptions on $g$ from \eqref{eq:cov}) these moments grow slowly enough for one to be able to characterize the law of $\mu(f)$ in terms of its moments. This makes proving that something converges to imaginary chaos rather straightforward since it is then a question about controlling moments -- indeed, this is what we will show for the XOR-Ising model.

Before going into details about the moments, let us point out that a  (formal) straightforward Gaussian computation yields the formula
\begin{equation}\label{eq:formal_moments}
  \E |\mu(f)|^{2N}\ ``=" \ \int_{U^{2N}} \frac{\prod_{1 \le i < j \le N} e^{-\beta^2 C_X(x_i,x_j)} \prod_{1 \le i < j \le N} e^{-\beta^2 C_X(y_i,y_j)}}{\prod_{1 \le i,j \le N} e^{-\beta^2 C_X(x_i,y_j)}} \prod_{i=1}^N f(x_i)\overline{f(y_i)}   dx_i dy_i,
\end{equation}
where we have written $``="$ to indicate that we have not justified this identity beyond $N=1$, or that one would have convergence of say $\mu_\delta$ to $\mu$ in all $L^p$-spaces. Nevertheless, let us not worry about rigor for a moment. The archetypical case of \eqref{eq:formal_moments} would be $C_X(x,y) = \log \frac{1}{|x-y|}$ and $f \equiv 1$ (or more precisely, $f\in C_c^{\infty}(\R^{d})$ and $f|_U=1$), in which case \eqref{eq:formal_moments} becomes the following interesting integral:
\begin{equation}\label{eq:formal_moments2}
\int_{U^{2N}} \frac{\prod_{1 \le i < j \le N} |x_i-x_j|^{\beta^2} \prod_{1 \le i < j \le N} |y_i - y_j|^{\beta^2}}{\prod_{1 \le i,j \le N} |x_i - y_j|^{\beta^2}} \, dx_1 \dots dx_N dy_1 \dots dy_N
\end{equation}
The finiteness of \eqref{eq:formal_moments} for all $\beta \in (-\sqrt{d},\sqrt{d})$ is not completely trivial, although it is well-known to experts and can be proven e.g. by using the techniques in \cite[Appendix A]{LRV}.
Rather precise lower and upper bounds for \eqref{eq:formal_moments2} are known for $d = 2$, see e.g. \cite{GP, LSZ}. As we will see later on, these bounds imply in particular that the law of $\mu(f)$ is determined by its moments.
Our goal in this section is to prove similar bounds in all dimensions and for more general covariance kernels. This is also crucial for us in Section~\ref{sec:ising}, where we deal with the convergence of the XOR-Ising model. Note that in this case, the relevant field is the zero boundary condition GFF from Example \ref{ex:fields1} and moment bounds on the corresponding imaginary chaos do not follow directly e.g. from \cite{GP,LSZ}.

In  \cite{GP} estimates for moments in the case of the purely logarithmic kernel are obtained via first establishing a  2-dimensional version of a famous inequality called Onsager's lemma \cite{O} (also sometimes called the electrostatic inequality). The original 3-d version of Onsager's inequality (where one has the $|x|^{-1}$-kernel instead  of our logarithmic kernel) has been used e.g. in the modern theory of stability of matter  \cite{DL,FL}, and we  refer to \cite{FL} or \cite{So} for a  mathematical proof of the inequality. These proofs do not apply as such for our general logarithmic covariance kernels,  especially in the case of $d\neq 2$, but we will shortly discuss  in more detail how this can be overcome and explain the various versions of the generalised inequality we shall need. 

In any case, after a suitable version of Onsager is at our hand, we may then finish the proof of the desired moment bounds by implementing the combinatorial part of the argument in \cite{GP} as stated in  Lemma \ref{lemma:combinatorial_argument} below. We  include a proof of the lemma in the appendix for the reader's convenience as the proof in \cite{GP} is for $d=2$ and there are cosmetic differences for $d\not=2$. Moreover, we also note that the approach of \cite{GP} for lower bounds of the moments generalizes to some extent, and we record consequences for the tail of the imaginary chaos. 
Finally, it is to be noted that very precise estimates for the moments in the case of $d=2$ and the purely logarithmic kernel were obtained recently in \cite{LSZ}, with applications to the tails of the corresponding imaginary chaos.

Let us then discuss our   versions of Onsager's lemma, of which there are four in total.
Our first version  (see Proposition \ref{prop:electrostatic_inequality} (i) below) takes care of general 2-dimensional covariances for which $g \in C^2(U \times U)$. This generalizes the one in \cite{GP}, which considers just the  purely logarithmic kernel. To achieve this generalization, we need to replace the complex analytic proof of \cite{GP} by a more probabilistic one. The effect  of the term $g$  in the covariance is dealt with by a rather direct error analysis. Surprisingly enough, this proof or the other known ones appear not to work for dimensions $d\not=2$, and for that purpose we require a more complicated approach based on a general decomposition principle of logarithmic covariances -- indeed,  our second version of Onsager's inequality is Theorem \ref{th:onsager_general} below, and its proof will be published elsewhere as it relies on the above decomposition principle whose proof we feel does not belong in this article. The above versions of Onsager are local in the sense that one considers points lying in a fixed subset of $U$. In contrast, our third version (Proposition \ref{prop:electrostatic_GFF} below) is a global result in the case of the $GFF$ on a bounded domain. Finally, our  fourth version  (Proposition \ref{prop:electrostatic_inequality} (ii) below) is an auxiliary result that does not require further regularity from $g$, but comes at the cost of having error of order $O(N^2)$ instead of $O(N)$. Hence  it is not  an 'honest Onsager inequality' from our point of view. In fact, quadratic error in $N$ is too large to prove
that the moments determine the distribution, but we may use this version of the inequality to verify that $\E |\mu_\varepsilon(f)|^{2N}$ converges to \eqref{eq:formal_moments} as $\varepsilon \to 0$, validating our  formal computations and verifying that all moments are finite. 

We start with the first and fourth version of  our Onsager inequalities.

\begin{proposition}\label{prop:electrostatic_inequality}
  Let $K$ be a compact subset of $U$, $N \ge 1$, $q_1, \dots, q_N \in \{-1,1\}$, and $x_1, \dots, x_N \in K$. Assume that the covariance of $X$ is as in \eqref{eq:cov} and that $g$ satisfies the assumptions \eqref{eq:assumptions}. We then have the following two Onsager-type inequalities$:$
 \begin{enumerate}[label={\rm (\roman*)}]
    \item Let $d = 2$ and assume that in addition to \eqref{eq:assumptions} we have $g \in C^2(U \times U)$. Then
      \[-\sum_{1 \le j < k \le N} q_j q_k \E X(x_j) X(x_k) \le \frac{1}{2} \sum_{j=1}^N \log \frac{1}{\frac{1}{2} \min_{k \neq j} |x_j - x_k|} + C N,\]
    for some constant $C > 0$ depending only on $g$ and $K$.
    \item Let $d \ge 1$ be arbitrary. For convolution approximations $X_\varepsilon$ $($as in Lemma \ref{le:standard}$)$ of $X$ we have
      \[-\sum_{1 \le j < k \le N} q_j q_k \E X_\varepsilon(x_j) X_\varepsilon(x_k) \le \frac{1}{2} \sum_{j=1}^N \log \frac{1}{\frac{1}{2} \min_{k \neq j} |x_j - x_k|} + C N^2\]
      for some constant $C > 0$ that is independent of  $\varepsilon > 0$, and depends only on $g$ and $K$. Note that no extra assumptions beyond \eqref{eq:assumptions} on $g$ are required in this case.
  \end{enumerate}
\end{proposition}

\begin{proof}
  Let $r_j = \frac{1}{2} \big(\min_{k \neq j} |x_j - x_k| \wedge \dist(K,\partial U)\big)$ and set (see here Remark~\ref{rem:sloppy})
  \[Z_j = \frac{1}{2\pi} \int_0^{2\pi} X(x_j + r_j e^{i \theta}) \, d\theta.\]
  We have
  \begin{align} \label{eq:tarvis} 
  \E Z_j^2 & = \frac{1}{(2\pi)^2} \int_0^{2\pi} \int_0^{2\pi} \Big( \log\frac{1}{|r_j e^{i \theta} - r_j e^{i \varphi}|} + g(x_j + r_j e^{i \theta}, x_j + r_j e^{i \varphi})\Big) \, d\theta \, d\varphi \\
    & = \log \frac{1}{r_j} + \frac{1}{(2\pi)^2} \int_0^{2\pi} \int_0^{2\pi} g(x_j + r_j e^{i \theta}, x_j + r_j e^{i \varphi}) \, d\theta \, d\varphi\nonumber
  \end{align}
  by harmonicity of $\log(|\cdot|^{-1})$.
  Moreover, for $j \neq k$ we obtain, again using harmonicity of the log,
  \begin{align*}
     \E Z_j Z_k     & = \frac{1}{(2\pi)^2} \int_0^{2\pi} \int_0^{2\pi} \Big( \log\frac{1}{|x_j + r_j e^{i \theta} - x_k - r_k e^{i \varphi}|} + g(x_j + r_j e^{i \theta}, x_k + r_k e^{i \varphi})\Big) \, d\theta \, d\varphi \\
    & = \log \frac{1}{|x_j - x_k|} + \frac{1}{(2\pi)^2} \int_0^{2\pi} \int_0^{2\pi} g(x_j + r_j e^{i \theta}, x_k + r_k e^{i \varphi}) \, d\theta \, d\varphi.
  \end{align*}
  Letting $c_{j,k} = \frac{1}{(2\pi)^2} \int_0^{2\pi} \int_0^{2\pi} g(x_j + r_j e^{i \theta}, x_k + r_k e^{i \varphi}) \, d\theta \, d\varphi$ this means that
  \[\E Z_j^2 = \log \frac{1}{r_j} + c_{j,j} \quad \text{and} \quad \E Z_j Z_k = \log \frac{1}{|x_j - x_k|} + c_{j,k}.\]
 A simple computation  (where we allow also $j = k$) yields that
  \begin{align*}
    c_{j,k} & = \frac{1}{(2\pi)^2} \int_0^{2\pi} \int_0^{2\pi} g(x_j + r_j e^{i \theta}, x_k + r_k e^{i \varphi}) \, d\theta \, d\varphi \\
    & = \frac{1}{(2\pi)^2} \int_0^{2\pi} \int_0^{2\pi} (g(x_j, x_k) + \left(\begin{matrix}r_j e^{i \theta} \\ r_k e^{i \varphi}\end{matrix}\right) \nabla g(x_j, x_k) + \xi(r_j e^{i\theta}, r_k e^{i\varphi})) \, d\theta \, d\varphi \\
      & = g(x_j, x_k) + \int_0^{2\pi} \int_0^{2\pi} \xi(r_j e^{i\theta}, r_k e^{i\varphi}) \, d\theta \, d\varphi, 
     \;  =:\;  g(x_j, x_k) + d_{j,k},
  \end{align*}
  where $\xi$ is the remainder in the Taylor expansion of $g$ at the point $(x_j,x_k)$, and the error $d_{j,k}$ is of  the order $$|d_{j,k}|\lesssim \max(r_j^2,r_k^2).$$ Since $Z_j$ are jointly Gaussian, their covariance is positive definite, and in particular
  \begin{align*}
    0 & \le \sum_{1 \le j,k \le N} q_j q_k \E Z_j Z_k = \sum_{j = 1}^N \E Z_j^2 + \sum_{j \neq k} q_j q_k \E Z_j Z_k \\
    & = \sum_{j=1}^N \log \frac{1}{\frac{1}{2}(\min_{k \neq j} |x_k - x_j| \wedge \dist(K,\partial U))} + 2 \sum_{1 \le j < k \le N} q_j q_k \E X(x_j) X(x_k) \\
    & \quad + \sum_{j=1}^N d_{j,j} + 2 \sum_{1 \le j < k \le N} q_j q_k d_{j,k}.
  \end{align*}
 A key observation for the proof is that by the disjointness of the circles and since $d=2$ we have the area estimate \begin{equation}\label{eq:area}
 |\sum_{j=1}^N d_{j,j}| \lesssim \sum_{j=1}^N r_j^2 \lesssim |U|.
 \end{equation}
 In turn,
  \[\log \frac{1}{\frac{1}{2}\big(\min_{k \neq j} |x_k - x_j| \wedge \dist(K,\partial U)\big)} \le \log \frac{1}{\frac{1}{2}\min_{k \neq j} |x_k - x_j|} + \max\big(\log \frac{1}{\frac{1}{2}\dist(K,\partial U)},0\big).\]
  Moreover, \eqref{eq:area} implies that
  \[|\sum_{1 \le j < k \le N} q_j q_k d_{j,k}| \le \sum_{1 \le j < k \le N} c \max(r_j^2, r_k^2) \le 2 N c |U|\]
  for some constant $c > 0$ that depends on $g$. By putting all the observations together, part (i) of the claim follows.

In order to prove the second inequality, we again employ auxiliary random variables $Z_j$. Letting the radii $r_j$ be as before we set this time
  \[Z_j := X_{\max(\varepsilon,r_j)}(x_j).\]
  By Lemma~\ref{le:standard} we have
  \[\E Z_j^2 = \log \frac{1}{\max(\varepsilon,r_j)} + O(1)\]
  and
  \[\E Z_j Z_k = \log \frac{1}{\max(\varepsilon,|x_j-x_k|)} + O(1) = \E X_\varepsilon(x_j) X_\varepsilon(x_k) + O(1).\]
  Hence
  \begin{align*}
    0 & \le \sum_{1 \le j,k \le N} q_j q_k \E Z_j Z_k = \sum_{j=1}^N \E Z_j^2 + \sum_{j\neq k} q_j q_k \E Z_j Z_k \\
    & \le \sum_{j=1}^N \log \frac{1}{\log(\frac{1}{2} (\min_{k \neq j} |x_k - x_j|))} + 2 \sum_{1 \le j < k \le N} q_j q_k \E X_\varepsilon(x_j) X_\varepsilon(x_k) + C N^2.
  \end{align*}
\end{proof}
\begin{remark}\label{rem:sloppy}Note that the definition of the variables $Z_j$ in the above proof is somewhat formal; we have only defined $X$ as an element of $H^{-\varepsilon}(\R^2)$, so it so it would seem that integrating $X$ over a circle can not be interpreted as $X$ acting on a valid test function. Nevertheless, the probabilistic objects we use are simply a device to obtain covariance inequalities. To make things precise, one might want to rephrase the definition of $Z_j$  as $Z_j:=X(\rho_{\varepsilon,x_j})$, where $\rho_{\varepsilon,x_j}\in C^\infty_c(\R^2)$ is a convolution approximation of uniform probability measure on a circle of radius $r_j$ around $x_j$. Then later in the obtained covariance inequalities, one simply lets $\varepsilon\to 0$ and gets the desired statements. However, we feel that this level of precision could obscure the idea of the proof and hope that the reader will be forgiving us for the slight inaccuracy in  the exposition.

\end{remark}

Let us next state the third version of Onsager's lemma, which is even more
local in nature than Proposition~\ref{prop:electrostatic_inequality} but works in arbitrary dimensions. For a definition of the space $H_{loc}^s$, we refer the reader to Section \ref{sec:spaces}.

\begin{theorem}\label{th:onsager_general}
  Assume that $X$ is a log-correlated field on the domain $U\subset\R^d$ with $0\in U$ and assume that $g \in H^{d + \varepsilon}_{loc}(U\times U)$ for some $\varepsilon > 0$. Then there is a neighbourhood $B_\delta (0)\subset U$ of the origin so that $X$ satisfies the following electrostatic inequality in $B_\delta (0):$\\
  \quad for any
  $N \ge 1$, $q_1, \dots, q_N \in \{-1,1\}$ and $x_1, \dots, x_N \in B_\delta (0)$ it holds that
\begin{equation}\label{eq:osager}
      -\sum_{1 \le j < k \le N} q_j q_k \E X(x_j) X(x_k) \le \frac{1}{2} \sum_{j=1}^N  \log \frac{1}{\frac{1}{2} \min_{k \neq j} |x_j - x_k|} + C N,
 \end{equation}     
where $C$ is independent of the points $x_j$ or $N$, but may depend on the neighbourhood $B_\delta(0).$
\end{theorem}

\begin{proof}
  This is Theorem~7.1 in \cite{JSW}.
\end{proof}
\noindent One should observe that in the above result, in case $d=2$ (disregarding the more local nature that does not affect our moment estimates) the condition on $g$ is certainly satisfied if 
$g\in C^{2+\varepsilon}$. On the other hand, in a certain sense the class  $H_{loc}^{2+\varepsilon}(\R^2\times \R^2)$ is much larger than $C^{2}(\R^2\times \R^2)$, e.g. it allows for local behaviour of type $|x-x_0|^{\delta}$, $\delta>0$, so the conditions are not comparable but extend each other.

All the above results are local in nature. In order to obtain full grip of the moments, or optimal understanding of the imaginary chaos on a two-dimensional bounded domain as a random element in $\mathcal{S}'(\R^d)$, it is desirable to have a global version which is valid for all $x_1,\ldots, x_N\in U.$ This can  be achieved as a consequence of the previous results if $g$ continues with suitable smoothness in a neighbourhood of the closure $\overline{U}$ (by Theorem  \ref{th:onsager_general} the extension needs not to be even a covariance). We next show that one can also obtain a global Onsager inequality in the case of the GFF on a bounded simply connected domain $U\subset \R^2=\C$. For that end let us recall that the density of the hyperbolic metric of $U$ at a point $z\in U$ is given by 
$$
|d_Hz|:=\frac{2|\psi'(z)|}{1-|\psi(z)|^2}|dz|,
$$
where $\psi:U\to\D$ is any conformal map.
The hyperbolic distance between two points in $U$ is obtained by minimizing the integral $\int_\gamma |d_Hz|$ over all rectifiable curves in $U$ joining the given points. In a simply connected domain the classical Koebe estimate (\cite[Theorem 4.3]{GM} -- we refer overall to \cite{GM} on basic facts on hyperbolic metric) says that
\[\frac{1}{2}(d(z,\partial U))^{-1}|dz| \leq |d_Hz|\leq 2(d(z,\partial U))^{-1}|dz|.\]
In particular, the hyperbolic distance dominates a multiple of the standard metric.
The hyperbolic metric is conformally invariant, whence one easily computes that in the unit disc the hyperbolic distance of points $w,z\in \D$ equals 
$$
d_H(w,z)=\log\Big(\frac{1+\rho(w,z)}{1-\rho(w,z)}\Big),\qquad \textrm{with} \quad \rho(w,z):=\left|\frac{z-w}{1-\overline{z}w}\right|,
$$
where $\rho(w,z)$ is called  the pseudo hyperbolic metric between $z$ and $w.$ Also $\rho$ is an honest metric. Given $z_0\in U$ and $r>0$ we denote by $B_\rho(z_0,r)\subset U$ the pseudo-hyperbolic ball of radius $r$. We then have $B_\rho(z_0,r)=B_H(z_0, r')$, and this is the image of the ordinary ball $B(0,R)\subset\D$ under any conformal map $\psi^{-1}:\D\to U$ such that $\psi(z_0)=0$. Here $R=r$ and $r'$ is given by $r'=\log((1+r)/(1-r))$. 
\begin{proposition}\label{prop:electrostatic_GFF}
Assume that  $U\subset\R^2$ is simply connected and bounded and that $X$ is the zero boundary condition GFF on $U$. Let $N \ge 1$, $q_1, \dots, q_N \in \{-1,1\}$, and $x_1, \dots, x_N \in U$ be arbitrary. Then
      \[-\sum_{1 \le j < k \le N} q_j q_k \E X(x_j) X(x_k) \le \frac{1}{2} \sum_{j=1}^N \log \left( \frac{1}{\frac{1}{2} \min_{k \neq j} |x_j - x_k|}\right) \; +\, CN
      \]
    for some constant $C > 0$ depending only on  the domain $U$.
    \end{proposition}
\begin{proof}
We assume first that $U=\D$. Let $ r_j=\frac{1}{2}\inf_{k\not=j}d_\rho(x_j,x_k)$ be half the pseudo hyperbolic distance of $x_j$ to the nearest point. Denote $B_j:=B_\rho(x_j,r_j)$. Let $\nu_j$ stand for the harmonic measure on $\partial B_j$ with respect to the point $x_j$ (computed with respect to the ball $B_j$). We consider the random variables 
$$
Y_j=\int_{\partial B_j} X(z)\nu_j(dz)
$$
(concerning the definition, an analogue of Remark \ref{rem:sloppy} applies). By recalling \eqref{eq:GFFcovariance}, the covariance $C_X(z,w)$ is separately harmonic with respect to  both of the variables. Since the balls $B_j$ are disjoint, a standard limiting argument allows us to use the harmonicity of the Green's function to compute for $k\not=j$
\begin{eqnarray}\label{eq:askel1}
\E Y_jY_k&=& \int_{\partial B_j} \Big( \int_{\partial B_k}C_X(z,w) \nu_k(dw)\Big)\nu_j(dz)= \int_{\partial B_j}C_X(z,x_k)\nu_j(dz)\nonumber\\
&=& C_X(x_j,x_k).
\end{eqnarray}
We next observe that by the conformal invariance of the harmonic measure we have for any $h\in C(\partial B_j)$  that
$$
\int_{\partial B_j}h(z)\nu_j(dz)=\vint_{\partial B_\rho(0,r_j)}h(\tau(w))|dw|,
$$
where $\vint_{}$ stands for the averaged integral and $\tau$ is a conformal self map of $\D$ that carries $B_\rho(0,r_j)\subset \D$ to $B_j$. By applying this formula and the conformal invariance of the GFF covariance we thus obtain
\begin{eqnarray}\label{eq:askel2}
\E Y_j^2&=& \int_{\partial B_j} \Big( \int_{\partial B_j}C_X(z,w) \nu_j(dw)\Big)\nu_j(dz)
\;=\;  \vint_{\partial B_\rho(0,r_j)\times \partial B_\rho(0,r_j)}\log\Big|\frac{1-\overline{z}w}{z-w}\Big||dw||dz|\nonumber\\
&=& \vint_{\partial B_\rho(0,r_j)\times \partial B_\rho(0,r_j)}\log\Big|\frac{1}{z-w}\Big||dw||dz|\;\; =
\;\; \log(1/ r_j),
\end{eqnarray}
where we noted the harmonicity of $\log|1-\overline{z}w|$ and recalled the computation \eqref{eq:tarvis}. We also used the fact that the standard radius of the pseudo hyperbolic ball centred at the origin is the same as the pseudo-hyperbolic one. 

By performing our standard consideration of the expectation $\E \big|\sum_{k=1}^N q_jY_j\big|^2$, in view of \eqref{eq:askel1} we thus obtain the desired inequality with the right hand side
$$
\frac{1}{2} \sum_{j=1}^n  \log \left( \frac{1}{\frac{1}{2} \min_{k \neq j} \rho(x_j, x_k)}\right).
$$
The conformal invariance of both the covariance and the pseudo hyperbolic metric ensures that the stated inequality with the above right hand side is actually true on any simply connected domain. This yields the claim as we finally note that for any bounded domain  there is a constant $a>0$ so that $|z-w|\leq a\rho(z,w)$. This last inequality is seen by noting that Koebe's estimate yields  $|z-w|\leq (2{\rm diam}(U)) d_H(z,w)\approx \rho(z,w)$ for
  $\rho(z,w)\leq 1/2,$ and by boundedness of $U$ this yields the claim. 
\end{proof}

Our goal in this section was to bound the moments of imaginary
chaos. As noted already before, after Onsager's lemma the second ingredient we need for the upper bound
is the following estimate. As the proof is a rather straightforward generalization of the 2-dimensional
result in \cite{GP} it is given in the appendix.

\begin{lemma}\label{lemma:combinatorial_argument}
  Let $B(0,1)$ be the unit ball in $\reals^d$. We have
  \[\int_{B(0,1)^N} \exp\Big(\frac{\beta^2}{2} \sum_{j=1}^{N} \log \frac{1}{\frac{1}{2} \min_{k \neq j} |x_j - x_k|}\Big) \, dx_1 \dots dx_{N} \le c^N N^{N \frac{\beta^2}{2d}}\]
  for some constant $c > 0$.
\end{lemma}

\noindent This lemma and  Proposition \ref{prop:electrostatic_inequality} (ii) yield a bare uniform integrability statement which will be used to show that all the moments exist and that the formula \eqref{eq:formal_moments} is indeed correct. This verifies the part of Theorem \ref{th:moments} which claims that $\E|\mu(f)|^k<\infty$ for all $k$.

\begin{corollary}\label{cor:integrability}
  Let $K$ be a compact subset of $U$ and assume that $x_1,\dots,x_N,y_1,\dots,y_N \in K$. Denote $z_1 = x_1$, \dots, $z_N = x_N$, $z_{N+1} = y_1$, \dots, $z_{2N} = y_N$. We have the uniform bound
  \begin{align*}
    & e^{-\beta^2 \sum_{1 \le j < k \le N} (C_{X_\varepsilon}(x_j, x_k) + C_{X_\varepsilon}(y_j, y_k)) + \beta^2 \sum_{1 \le j, k \le N} C_{X_\varepsilon}(x_j, y_k)} \nonumber \\
    & \le \exp\Big(\frac{\beta^2}{2} \sum_{j=1}^{2N} \log \frac{1}{\frac{1}{2}\min_{k \neq j} |z_j - z_k|} + c N^2\Big) =: \Xi_{N}(z_1,\dots,z_{2N})
  \end{align*}
  for all $\varepsilon > 0$. Here the majorant $\Xi_{N}$ depends on the subset $K$ through the constant $c$, and is integrable over $K^{2N}$.  A fortiori, the formula \eqref{eq:formal_moments} for the moments is valid  for any $f\in C_c^\infty (U)$ under our standard assumptions \eqref{eq:assumptions}.
  \end{corollary}
  \begin{proof}
    We begin by writing out the moment $\E |\mu_\varepsilon(f)|^{2N}$ as a multiple integral
    \begin{align*}
  \E |\mu_\varepsilon(f)|^{2N} &= \int_{U^{2N}} \prod_{j=1}^N dx_j f(x_j) \prod_{j=1}^N dy_j \overline{f(y_j)} \E e^{i \beta \sum_{j=1}^N (X_\varepsilon(x_j) - X_\varepsilon(y_j)) + \frac{\beta^2}{2} \sum_{j=1}^N (\E X_\varepsilon(x_j)^2 + \E X_\varepsilon(y_j)^2)}\\
      & = \int_{U^{2N}} \prod_{j=1}^N dx_j f(x_j) \prod_{j=1}^N dy_j \overline{f(y_j)} e^{-\beta^2 \sum_{1 \le j < k \le N} (C_{X_\varepsilon}(x_j, x_k) + C_{X_\varepsilon}(y_j, y_k)) + \beta^2 \sum_{1 \le j, k \le N} C_{X_\varepsilon}(x_j, y_k)} \\
    & \le \|f\|_{\infty}^{2N} \int_{(\supp f)^{2N}} \exp\Big(\frac{\beta^2}{2} \sum_{j=1}^{2N} \log \frac{1}{\frac{1}{2}\min_{k \neq j} |z_j - z_k|} + c N^2\Big) \, dz_1 \dots dz_{2N}.
  \end{align*}
    Since the upper bound is independent of $\varepsilon$ we may use the dominated convergence theorem to let $\varepsilon\to0$ and deduce that the moments are finite and given by the right formula. 
\end{proof}

Lemma \ref{lemma:combinatorial_argument} combined with our versions of Onsager's inequality allows us to finally prove an
upper bound for the moments of the purely imaginary chaos, verifying the moment bound portion of Theorem~\ref{th:moments}.

\begin{theorem}\label{thm:moment_upper_bound_2}
  Assume that either $d = 2$ and $g \in C^2(U \times U)$, or $d$ is arbitrary and $g \in H_{loc}^{d+\varepsilon}(U \times U)$ for some $\varepsilon > 0$.
  Then for every $N \ge 1$ and $f \in C_c^\infty(U)$ we have for $\mu$ from Theorem~\ref{th:existuniq}
  \[\E |\mu(f)|^{2N} \le \|f\|_\infty^{2N} C^N N^{\frac{\beta^2 N}{d}}\]
  for some constant $C > 0$ $($which may depend on the support of $f)$.
\end{theorem}

\begin{proof}
To obtain the stated upper bounds, assume first that we are in the case $d = 2$ and $g \in C^2(U \times U)$. Then we may use  Corollary \ref{cor:integrability}   to infer
  \begin{align*}
    \E |\mu(f)|^{2N} & = \int_{U^{2N}} \prod_{j=1}^N dx_j f(x_j) \prod_{j=1}^N dy_j \overline{f(y_j)} e^{-\beta^2 \sum_{1 \le j < k \le N} (C_{X}(x_j, x_k) + C_{X}(y_j, y_k)) +\beta^2 \sum_{1 \le j, k \le N} C_{X}(x_j, y_k)} \\
    & \le \|f\|_{\infty}^{2N} \int_{(\supp f)^{2N}} \exp\Big(\frac{\beta^2}{2} \sum_{j=1}^{2N} \log \frac{1}{\frac{1}{2} \min_{k \neq j} |z_j - z_k|} + cN\Big) \, dz_1 \dots dz_{2N},
  \end{align*}
  where the last inequality is a consequence of part (1) of Proposition~\ref{prop:electrostatic_inequality}.
  The claim now follows from Lemma~\ref{lemma:combinatorial_argument}.

  In the case where $d$ is arbitrary and $g \in H^{d + \varepsilon}_{loc}(U \times U)$, we may by using compactness first cover $\supp f$ with a finite number of balls $B(a_1,\delta_1/2),\dots,B(a_m,\delta_m/2)\subset U$, where $\delta_\ell$ are given by Theorem~\ref{th:onsager_general}. Moreover, we can find a smooth partition of unity of non-negative functions $\eta_1,\dots,\eta_m$ such that $\supp \eta_\ell \subset B(a_\ell,\delta_\ell)$ and for any $x$ in a small neighbourhood of $\supp f$ we have $\sum_{\ell=1}^m \eta _\ell(x) = 1$. Then
  \begin{align*}
    \E |\mu(f)|^{2N} & = \E |\sum_{\ell=1}^m \mu(f \eta _\ell)|^{2N} \le m^{2N} \E \max_\ell(|\mu(f \eta_\ell)|^{2N}) \le m^{2N} \sum_{\ell=1}^m \E |\mu(f \eta_\ell)|^{2N},
  \end{align*}
  and each summand may be approximated as in the previous case, replacing the use of Proposition~\ref{prop:electrostatic_inequality} with Theorem~\ref{th:onsager_general}.
\end{proof}

As the final component in the proof of Theorem \ref{th:moments}, we record the following basic fact about the moments from Theorem \ref{thm:moment_upper_bound_2} growing slowly enough for the moments to determine the law of $\mu$.

\begin{corollary}\label{co:detemination}
  Under the conditions of Theorem \ref{thm:moment_upper_bound_2} all the exponential moments $\E e^{\lambda |\mu(\varphi)|}$ for $\lambda \in \reals$ and $\varphi \in C^\infty_c(U)$ are finite and in particular the moments $\E \mu(\varphi)^k \overline{\mu(\varphi)}^l$ for $k,l \ge 0$ exist and they determine the distribution of $\mu$ as a random distribution in $\mathcal{D}'(U).$
\end{corollary}
\begin{proof} As is standard, by linearity, the joint distribution of $(\mu(\phi_1),\ldots \mu(\phi_m))$ for any number of test functions $\phi_j\in C_c^\infty (U)$ is determined as soon as the case of an arbitrary single test function, or $m=1$ is known. This on the other hand, follows from Theorem \ref{thm:moment_upper_bound_2}, since the stated growth rate of the moments is well-known to be small enough to determine the distribution, see e.g. \cite[Theorem~3.3.12]{D}. Finally, the finiteness of exponential moments follows from expanding the exponential as a power series and using Theorem \ref{thm:moment_upper_bound_2} coupled with a standard Jensen estimate.
\end{proof}
As mentioned, the proof of Theorem \ref{th:moments} now follows from combining Corollary \ref{cor:integrability}, Theorem \ref{thm:moment_upper_bound_2}, and Corollary \ref{co:detemination}.

\medskip

Asymptotics for moments in the case of the Gaussian Free Field (or more precisely for $g=0$) have been proven in \cite{GP} by scaling and space partition arguments. Below we show how to slightly alter their method to deal with a general covariance $C_X(x,y)$ and obtain the following lower bounds for the moments. One should note that the main term in the estimate is the same as for the upper bound.

\smallskip

\begin{proposition}\label{prop:moment_lower_bound}
  Let $f\in C^\infty_c(U)$ be  non-negative and not identically zero. Then for $\mu$ from Theorem \ref{th:existuniq},
  \[\log \E |\mu(f)|^{2N} \ge \frac{\beta^2}{d} N \log N + \mathcal{O}(N).\]
\end{proposition}
\begin{proof} By the assumption we may choose a cube $K\subset U$ so that $f \geq c_0>0$ on $K$. 
 With a simple scaling and translation argument we may assume that $K = [0,1]^d$ and $c_0=1.$
Let us denote
\begin{align*}
  Z_{\beta,2N}(\Omega)&=``\E |\mu(\mathbf{1}_\Omega)|^{2N}"=\int_{\Omega^{2N}}\frac{\prod_{i,j=1}^Ne^{\beta^2 C_X(x_i, y_j)}}{\prod_{1\leq i<j\leq N}e^{\beta^2C_X(x_i, x_j)+\beta^2C_X(y_i, y_j)}}\prod_{i=1}^Ndx_i\prod_{j=1}^N dy_j,
\end{align*}
for any measurable subset $\Omega \subset K$ and integer $N\geq 0$. Here we wrote $``\E |\mu(\mathbf{1}_\Omega)|^{2N}"$ to indicate that we ignore the discussion about whether or not $\mathbf{1}_\Omega$ is a suitable test function, since it's only the integral we are interested in. Note that $\E |\mu(f)|^{2N}\geq Z_{\beta,2N}(K)$.

Assume that  $0\leq N_1\leq N$ is an integer and write $N_2=N-N_1$. Let also $\Omega_1,\Omega_2\subset K$ be two measurable subsets (with positive $2N$-dimensional measure) satisfying $\Omega_1\cap\Omega_2=\emptyset$. Then the total integral defining $Z_{\beta,2N}(K)$, can be bounded from below by restricting to the subset of $\Omega^{2N}$ where precisely $N_1$ of both the $x$- and the $y$-variables are in $\Omega_1$ and $N_2$ of them are in $\Omega_2$. There are ${N\choose N_1}^2$ ways to choose the variables in this way and we find the following bound:

\begin{align*}
Z_{\beta,2N}(K)&\geq {N\choose N_1}^2Z_{\beta, 2N_1}(\Omega_1) Z_{\beta,2N_2}(\Omega_2)\E_\nu e^{\beta^2 U}\geq {N\choose N_1}^2Z_{\beta, 2N_1}(\Omega_1) Z_{\beta,2N_2}(\Omega_2)e^{\beta^2\E_\nu  U},
\end{align*} 

\noindent where in the last step we used Jensen's inequality, and we have also introduced the following notation: $\nu$ is a probability measure on $\Omega_1^{2N_1}\times \Omega_2^{2N_2}$ of the form 

\begin{align*}
\nu(dx^{(1)},dy^{(1)},dx^{(2)},dy^{(2)})&=\frac{1}{Z_{\beta,2N_1}(\Omega_1)}\frac{1}{Z_{\beta,2N_2}(\Omega_2)}\frac{\prod_{i,j=1}^{N_1}e^{\beta^2 C_X(x_i^{(1)}, y_j^{(1)})}}{\prod_{1\leq i<j\leq N_1}e^{\beta^2 C_X(x_i^{(1)}, x_j^{(1)})+\beta^2 C_X(y_i^{(1)}, y_j^{(1)})}}\\
&\quad \times \frac{\prod_{i,j=1}^{N_2}e^{\beta^2 C_X(x_i^{(2)}, y_j^{(2)})}}{\prod_{1\leq i<j\leq N_2}e^{\beta^2 C_X(x_i^{(2)}, x_j^{(2)})+\beta^2 C_X(y_i^{(2)}, y_j^{(2)})}}dx^{(1)}dy^{(1)}dx^{(2)}dy^{(2)},
\end{align*}

\noindent where $dx^{(i)}$ and $dy^{(i)}$ denote the Lebesgue measure on $\Omega_i^{N_i}$, and we write 

$$
U=\log \frac{\prod_{i=1}^{N_1}\prod_{j=1}^{N_2}e^{C_X(x^{(1)}_i, y^{(2)}_j)+C_X(y_i^{(1)}, x_j^{(2)})}}{\prod_{i=1}^{N_1}\prod_{j=1}^{N_2}e^{C_X(x_i^{(1)}, x_j^{(2)})+C_X(y_i^{(1)}, y_j^{(2)})}}.
$$

\noindent We point out that the density of $\nu$ (as well as the domain of $\nu$) is invariant under the transformation $x^{(2)}\leftrightarrow y^{(2)}$, but under this transformation $U$ is mapped to $-U$, so we see that $\E_\nu U=0$. We conclude that

$$
Z_{\beta,2N}(K)\geq {N\choose N_1}^2Z_{\beta, 2N_1}(\Omega_1) Z_{\beta,2N_2}(\Omega_2),
$$

\noindent or in other words

$$
\frac{1}{[N!]^2}Z_{\beta,2N}(K)\geq \frac{1}{[N_1!]^2}Z_{\beta, 2N_1}(\Omega_1) \frac{1}{[N_2!]^2}Z_{\beta,2N_2}(\Omega_2).
$$

\noindent By induction, if $(\Omega_j)_{j=1}^k$ are non-empty disjoint positive measure subsets of $K$ and $(N_j)_{j=1}^k$ are non-negative integers such that $N_1+...+N_k=N$, then 

\begin{equation}\label{eq:ineq}
\frac{1}{[N!]^2}Z_{\beta,2N}(K)\geq \prod_{j=1}^k \frac{1}{[N_j!]^2}Z_{\beta,2N_j}(\Omega_j).
\end{equation}

Let us now apply this inequality to the case where $k=\lceil N^{1/d}\rceil^d$, $N_j=1$ for all $j= 1,\dots,N$, $N_j = 0$ for $j = N+1,\dots,k$, and $\Omega_j$ is a translate of $[0,\lceil N^{1/d}\rceil^{-1})^d$. This yields 
  \[\log Z_{\beta,2N}(K)\geq \log [N!]^2+\sum_{i=1}^N \log Z_{\beta,2}(\Omega_i).\]
We now have for some vector $v_i\in[0,1)^d$
\begin{align*}
Z_{\beta,2}(\Omega_i)&=\int_{[0,\lceil N^{1/d}\rceil^{-1})^{2d}}\frac{e^{\beta^2 g(v_i+x,v_i+y)}}{|x-y|^{\beta^2}}dxdy
  \; \geq\;  e^{-\beta^2\|g\|_{L^\infty(K)}}\lceil N^{1/d}\rceil^{\beta^2-2d}\int_{[0,1)^{2d}}\frac{1}{|x-y|^{\beta^2}}dxdy
\end{align*}
so that 
\begin{align*}
\log Z_{\beta,2N}(K)&\geq 2 N\log N-2N +o(N)+\sum_{i=1}^N \left[\left(\frac{\beta^2}{d}-2\right)\log N+\mathcal{O}(1)\right]\\
&=\frac{\beta^2}{d}N\log N+\mathcal{O}(N).\qedhere
\end{align*}
\end{proof}

As an application of the moment bounds we close this subsection by proving Theorem \ref{th:tails}.

\smallskip

\begin{proof}[Proof of Theorem \ref{th:tails}]
  Fix $\lambda >1.$ By Chebyshev's inequality and Theorem~\ref{thm:moment_upper_bound_2} we have for any $N \ge 1$ that
  \begin{align*}
    \log \P(|\mu(\varphi)| > \lambda) & \le \log \frac{\E |\mu(\varphi)|^{2N}}{\lambda^{2N}} \\
    & \le \frac{\beta^2}{d} N \log(N) - 2N \log(\lambda) + c N
  \end{align*}
  for some $c > 0$.
  Letting $N = \left\lfloor \lambda^{\frac{2d}{\beta^2}} e^{-1 - \frac{cd}{\beta^2}}\right\rfloor$ and using the fact that the map $x \mapsto \frac{\beta^2}{d} x \log(x) - 2x \log(\lambda) + c x$ has Lipschitz constant of order $1$ when $x \approx \lambda^{\frac{2d}{\beta^2}}$, we get
  \begin{align*}
    \log \P(|\mu(\varphi)| > \lambda) & \le \frac{\beta^2}{d} \lambda^{\frac{2d}{\beta^2}} e^{-1-\frac{c d}{\beta^2}} \frac{2d}{\beta^2} \log(\lambda) - \frac{\beta^2}{d} \lambda^{\frac{2d}{\beta^2}} e^{-1-\frac{cd}{\beta^2}}(1 + \frac{cd}{\beta^2}) \\
    & \quad \quad - 2\lambda^{\frac{2d}{\beta^2}} e^{-1-\frac{cd}{\beta^2}} \log(\lambda) + c \lambda^{\frac{2d}{\beta^2}} e^{-1-\frac{cd}{\beta^2}} + O(1) \\
    & = -\frac{\beta^2}{d}\lambda^{\frac{2d}{\beta^2}}e^{-1-\frac{cd}{\beta^2}} + O(1).
  \end{align*}

  To prove the lower bound, assume that there exist arbitrarily large numbers $\lambda > 0$ such that
  \[\log \P(|\mu(\varphi)| > \lambda) \le - \lambda^{\frac{2d}{\beta^2} + \varepsilon}\]
  and fix some large enough $\lambda > 0$ (how large $\lambda$ is needed will be implicitly determined during the proof). By assuming that $\lambda$ is so large that $b := (\frac{1}{c})^{\frac{\beta^2}{2d}} \lambda^{1 + \frac{\beta^2}{2d} \varepsilon} >\lambda, $ we may compute for any $N \ge 1$ that
  \begin{align*}
    \E |\mu(\varphi)|^{2N} 
    & = 2N \Big(\int_0^{\lambda} +\int_{\lambda}^b + \int_b^\infty\Big) x^{2N-1} \P(|\mu(\varphi)| > x) \, dx \\
    & \le 2N a \int_0^{\lambda} x^{2N-1} e^{-c x^{\frac{2d}{\beta^2}}} \, dx + 2N a \int_{\lambda}^b x^{2N-1} e^{-\lambda^{\frac{2d}{\beta^2} + \varepsilon}} \, dx + 2N a \int_b^\infty x^{2N-1} e^{-c x^{\frac{2d}{\beta^2}}} \, dx,
  \end{align*}
  where we have used the bound $\P(|\mu(\varphi)| > x) \le a e^{-c x^{2d/\beta^2}}$ (for some $c > 0$ and $a >1$) coming from the first part of the proof, and applied  the monotonicity of $\P(|\mu(\varphi)| > x)$ and the fact $a>1$  when $x \in [\lambda, b]$. The length of the interval $[\lambda,b]$ is of the order $\lambda^{1 + \frac{\beta^2}{2d} \varepsilon}$.
  By differentiation it is easy to check that the function $x \mapsto x^{2N-1} e^{-c x^{\frac{2d}{\beta^2}}}$ has a unique maximum at $x_0 = \big(\frac{\beta^2(2N-1)}{2dc}\big)^{\frac{\beta^2}{2d}}$. 
  Fix some $\delta \in (0, \frac{\beta^2}{2d} \varepsilon)$. 
  If we now choose $N \in [\frac{1}{2} + \frac{dc}{\beta^2} \lambda^{\frac{2d(1+\delta)}{\beta^2}}, 2(\frac{1}{2} + \frac{dc}{\beta^2} \lambda^{\frac{2d(1+\delta)}{\beta^2}})]$ to be an integer, (this is possible for large enough $\lambda$), then
  by this choice of $N$, the function $x \mapsto x^{2N-1} e^{-c x^\frac{2d}{\beta^2}}$ is increasing on the interval $[0,\lambda]$ (simply due to the fact that with this choice of $N$, we have $x_0\geq \lambda^{1+\delta}$). The first integral is thus bounded by 
 $\displaystyle 2N a \lambda^{2N} e^{-c \lambda^{\frac{2d}{\beta^2}}}.$ 
  The second integral can be evaluated as 
  $\displaystyle a e^{-\lambda^{\frac{2d}{\beta^2} + \varepsilon}} (b^{2N} - \lambda^{2N})$, and
  finally the third integral has the upper bound
  \[2Na \int_b^\infty x^{2N-1} e^{-c x^{\frac{2d}{\beta^2}}} \, dx \le 2Na \int_b^\infty \frac{b^{2N+1} e^{-c b^{\frac{2d}{\beta^2}}}}{x^2} \, dx \le 2Na b^{2N+1} e^{-c b^{\frac{2d}{\beta^2}}},\]
  where we have used the fact that $b > 1$ for large enough $\lambda$, and also that the unique maximum of $x \mapsto x^{2N+1} e^{-c x^{\frac{2d}{\beta^2}}}$, which is at the point $\big(\frac{\beta^2 (2N+1)}{2dc}\big)^{\frac{\beta^2}{2d}}$, lies in $[\lambda,b]$ for large enough $\lambda$.
  Our choice of $N$ shows that both $\lambda^{2d/\beta^2+\varepsilon}$  and $b^{2d/\beta^2}$ grow quicker than $N^{1+\delta'}$ for some $\delta'>0$, and hence the second and the third integrals converge to zero as $\lambda\to\infty$ as $\log b$ is of the order $\log N$. From the first integral we obtain that by increasing $\lambda$, we can find arbitrarily large integers $N=N(\lambda)$ for which
  \[\E |\mu(\varphi)|^{2N} \lesssim e^{\frac{\beta^2 N}{d (1+\delta')} \log(N)}.\]
  This contradicts the lower bound given by Proposition~\ref{prop:moment_lower_bound}, and concludes the argument.
\end{proof}

\subsection{Regularity properties of imaginary chaos}\label{subsec:regularity}  In this section we continue our study of analytic properties of imaginary chaos, namely we shall study to which classical function spaces imaginary chaos belongs -- this corresponds to Theorem \ref{th:regularity}. We shall obtain essentially sharp results in Besov and Triebel--Lizorkin scales of function spaces, which include e.g. negative index H\"older spaces. As described in more detail in Section \ref{sec:spaces}, this gives much more combined size and smoothness information on the chaos than obtained by just considering the Hilbert--Sobolev spaces $H^s(\R^d)$.

We start by proving that we are dealing with true generalised functions, instead of say honest functions or even complex measures. This is the first component of Theorem \ref{th:regularity}. Though this is an important fact, it seems not to have been proven in the literature before.

\begin{theorem}\label{thm:l2_not_measure}
The imaginary chaos $\mu$ from Theorem \ref{th:existuniq} is almost surely not a complex measure.
\end{theorem}
\begin{proof}
  What the claim means is that the total variation of $\mu$ is almost surely infinite. To prove this,  it is enough to find a sequence of smooth functions $(h_k)_{k \ge 1}$ on $U$ such that almost surely $\sup_{k \ge 1} \|h_k\|_\infty \le 1$ but $\sup_{k \ge 1} |\mu(h_k)| = \infty$. A suitable candidate turns out to be a subsequence of the random sequence
  \[f_k(x) = e^{- i \beta X_{1/k}(x)} \psi(x),\]
  where $X_{1/k}$ are standard mollifications of $X$, and the real-valued test function $\psi \in C_c^\infty(U)$ satisfies $\mathbf{1}_B \le \psi(x) \le \mathbf{1}_{2B}$, where $B=B(x_0,r_0)$ is a ball such that the double sized ball $2B:=B(x_0,2r_0)$ is compactly contained in $U$. The idea of the proof is to calculate $\E\mu(f_k)$ and $\E |\mu(f_k)|^2$ and argue by Paley--Zygmund that the total variation must be infinite with probability $1$.
 
  To simplify the notation, denote $g_\delta(x) = e^{-i\beta X_\delta(x)} \psi(x)$ so that $f_k(x) = g_{1/k}(x)$. Let us begin by computing $\E \mu(g_\delta)$. Using Proposition \ref{prop:mu_convergence}, we can pick a sequence $\varepsilon_n \searrow 0$ such that $\mu_{\varepsilon_n}\to \mu$ almost surely in say $H^{-d/2-1}(\R^d)$. Moreover, using the fact that $\E\|\mu_{\varepsilon_n}\|_{H^{s}(\R^{d})}^{2}$ is bounded, which was part of the proof of Proposition \ref{prop:mu_convergence}, to justify a standard dominated convergence argument below,  we see that 

 \begin{align*}
    \E \mu(g_\delta) & = \lim_{n \to \infty} \int_{2B} \E e^{i \beta X_{\varepsilon_n}(x) - i \beta X_{\delta}(x)} e^{\frac{\beta^2}{2} \E X_{\varepsilon_n}(x)^2} \psi(x)\, dx \\
    & = \lim_{n \to \infty} \int_{2B} e^{-\frac{\beta^2}{2} \E X_{\delta}(x)^2 + \beta^2 \E X_{\varepsilon_n}(x) X_{\delta}(x)} \psi(x)\, dx \\
    & = \int_{2B} e^{-\frac{\beta^2}{2} \E X_{\delta}(x)^2 + \beta^2 \E X(x) X_{\delta}(x)} \psi(x) \, dx =: A_\delta,
  \end{align*}  
  where $\E X(x) X_{\delta}(x) = \lim_{\varepsilon \to 0} \E X_\varepsilon(x) X_\delta(x)$. Note that by Lemma~\ref{le:standard} we have $A_\delta \gtrsim \delta^{-\frac{\beta^2}{2}}$.
  
  To compute $\E|\mu(g_\delta)|^2$, we argue in a similar way, but now $L^2$-boundedness is not sufficient for us to conclude. The remedy comes from Proposition \ref{prop:electrostatic_inequality}(ii) which can be used to check that say $\sup_{n\geq 1}\E\|\mathbf{1}_{2B}\mu_{\varepsilon_n}\|_{H^{-d/2-1}(\R^d)}^6<\infty$. In turn, by the smoothness of the covariance $C_{X_{\delta}}$ one easily verifies that $\E \|g_\delta\|_{H^{d/2+1}(\R^d)}^6 < \infty$ for each fixed $\delta > 0$. Thus one finds that we can interchange the order of the limit and integration and we now obtain
 
  \begin{align*}
    \E |\mu(g_\delta)|^2 & = \lim_{n \to \infty} \int_{2B} \int_{2B} \E e^{i \beta X_{\varepsilon_n}(x) - i \beta X_{\varepsilon_n}(y) - i \beta X_{\delta}(x) + i \beta X_{\delta}(y)} \cdot \\
    & \quad \quad \quad \quad \quad \quad \quad e^{\frac{\beta^2}{2} \E X_{\varepsilon_n}(x)^2 + \frac{\beta^2}{2} \E X_{\varepsilon_n}(y)^2} \psi(x)\psi(y)\, dx \, dy \\
    & = \lim_{n \to \infty} \int_{2B} \int_{2B} e^{- \frac{\beta^2}{2} \E X_{\delta}(x)^2 - \frac{\beta^2}{2} \E X_{\delta}(y)^2
    + \beta^2 \E X_{\varepsilon_n}(x) X_{\delta}(x) + \beta^2 \E X_{\varepsilon_n}(y) X_{\delta}(y)} \psi(x)\psi(y) \cdot \\
    & \quad \quad \quad \quad \quad \quad \quad e^{\beta^2 \E X_{\delta}(x) X_{\delta}(y) + \beta^2 \E X_{\varepsilon_n}(x) X_{\varepsilon_n}(y) - \beta^2 \E X_{\varepsilon_n}(x) X_{\delta}(y) - \beta^2 \E X_{\delta}(x) X_{\varepsilon_n}(y)} \, dx \, dy\\
    & = \int_{2B} \int_{2B} e^{-\frac{\beta^2}{2} \E X_\delta(x)^2 - \frac{\beta^2}{2} \E X_{\delta}(y)^2 + \beta^2 \E X_\delta(x) X(x) + \beta^2 \E X_\delta(y) X(y)} \psi(x) \psi(y) \cdot \\
    & \quad \quad \quad \quad \quad \quad \quad e^{\beta^2 \E X_{\delta}(x) X_{\delta}(y) + \beta^2 \E X(x) X(y) - \beta^2 \E X_\delta(x) X(y) - \beta^2 \E X_{\delta}(y) X(x)} \, dx \, dy \\
    & =: B_\delta.
  \end{align*} 
  Our aim is to show that $\lim_{\delta \to 0} \frac{A_\delta^2}{B_\delta} = 1$. Let
  \[a_\delta(x) := \delta^{\beta^2/2} e^{-\frac{\beta^2}{2} \E X_\delta(x)^2 + \beta^2 \E X_\delta(x) X(x)} \psi(x)\]
  and
  \[b_\delta(x,y) := e^{\beta^2 \E X_\delta(x) X_\delta(y) + \beta^2 \E X(x) X(y) - \beta^2 \E X_\delta(x) X(y) - \beta^2 \E X_\delta(y) X(x)}.\]
  Then we have
  \begin{equation}\label{eq:abfraction} \frac{B_\delta}{A_\delta^2} = \frac{\int_{2B} \int_{2B} a_\delta(x)a_\delta(y)b_\delta(x,y) \, dx \, dy}{\Big(\int_{2B} a_\delta(x) \, dx\Big)^2} = 1 + \frac{\int_{2B} \int_{2B} a_\delta(x)a_\delta(y) (b_\delta(x,y) - 1) \, dx \, dy}{\Big(\int_{2B} a_\delta(x) \, dx\Big)^2}.
  \end{equation}
  By Lemma~\ref{le:standard} we know that $a_\delta(x)$ is bounded both from above and away from $0$, uniformly in $\delta$ and $x$. Moreover, $b_\delta(x,y)$ has an integrable majorant of the form $C |x-y|^{-\beta^2}$ for some $C > 0$, and it converges to $1$ pointwise. Thus by the dominated convergence theorem the right hand side of \eqref{eq:abfraction} tends to $1$ as $\delta \to 0$, as desired.
  
  Since $\E |\mu(f_k)| \ge \E \mu(f_k)$, the Paley--Zygmund inequality shows that we have
  \[\P(|\mu(f_k)| > \theta \E |\mu(f_k)|) \ge (1 - \theta)^2 \frac{(\E \mu(f_k))^2}{\E |\mu(f_k)|^2}\]
  for any $\theta \in (0,1)$. Choosing $\theta = (\E |\mu(f_k)|)^{-\varepsilon}$ for some $\varepsilon > 0$ we thus see that
  \[\P(|\mu(f_k)| > (\E |\mu(f_k)|)^{1 - \varepsilon}) \ge (1 - (\E |\mu(f_k)|)^{-\varepsilon})^2 \frac{A_{1/k}^2}{B_{1/k}} \to 1\]
  as $k \to \infty$. As we noted above that $A_\delta \gtrsim k^{(1 - \varepsilon) \frac{\beta^2}{2}}$, this implies that
  \begin{equation}\label{eq:blowup_event}
    \P\big(|\mu(f_k)| \ge C k^{(1-\varepsilon)\frac{\beta^2}{2}} \text{ for infinitely many } k\big) = 1
  \end{equation}
  for some constant $C > 0$.
  This provides us with the desired subsequence $h_k$ and proves the claim. We note that for our purposes here one could have chosen for instance $\varepsilon = 1/2$, but we stated \eqref{eq:blowup_event} for later use in the proof of Theorem~\ref{th:besov} below.
\end{proof}

The following general result can be used to show that the imaginary chaos belongs to $C^s_{loc}(U)$ or $H^s_{loc}(U)$\footnote{The definition of localised functions spaces with subscript $loc$ was given in Subsection~\ref{sec:spaces}. We also recall that for general $s\in\R$, the interpretation of $C^s$ is $B_{\infty,\infty}^s$ -- see again Section \ref{sec:spaces}.} for indices $s<-\beta^2/2$, and this range is essentially optimal. Moreover, the optimality is not due to some special boundary effects since  it is shown using  localisations that lie compactly inside the domain $U$. This is the second part of Theorem~\ref{th:regularity}.

\begin{theorem}\label{th:besov} Assume that $\beta \in (0,\sqrt{d})$ and fix $1\leq p,q\leq\infty$. Moreover, let $X$ be a log-correlated field satisfying our basic assumptions \eqref{eq:cov} and \eqref{eq:assumptions}. Let $\mu$ be the imaginary chaos given by Theorem~\ref{th:existuniq}. Then the following are true. 
 \begin{enumerate}[label={\rm (\roman*)}]

\item We have almost surely $\mu \in B_{p,q,loc}^s(U)$ when $s < -\frac{\beta^2}{2}$, and $\mu\notin B_{p,q,loc}^{s}(U)$ for $s>-\frac{\beta^{2}}{2}$. 

\item Assume moreover that $g\in L^\infty (U\times U)$ or that $X$ is the $2d$ GFF with zero boundary conditions. Then almost surely $\mu \in B_{p,q}^s(\R^d)$ when $s < -\frac{\beta^2}{2}$. 

\item Analogous statements hold for the Triebel spaces in the case $p,q\in [1,\infty)$.

\end{enumerate}
\end{theorem} 
\begin{proof}
(i).  Fix $\psi\in C_c^\infty (U),$ and denote the support of $\psi$ by $K$ so that $K$ is a compact subset of $U$. In view of the inclusions \eqref{eq:betrie},\eqref{eq:bebe} and the embedding \eqref{eq:beem}, in order to prove the claim it is enough to establish that for any $s<-\beta^2/2$ and for  arbitrary large positive integers $n$ it holds  that  $\psi\mu \in B_{2n,2n}^s(\R^d)$ almost surely.

We fix a large $n$ and compute a suitable moment of the Besov-norm as follows 
  \begin{align*}
    \E \|\psi \mu\|_{{B}_{2n,2n}^s}^{2n} & = \E \sum_{j=0}^\infty 2^{2nsj} \int_{\reals^d} |((\psi\mu) * {\phi_j})(x)|^{2n} \, dx,
  \end{align*}
  where the $\phi_j$:s are as in the discussion leading to \eqref{eq:besov}.  
  By Proposition \ref{prop:electrostatic_inequality}(ii), and using the fact that the integrand is invariant under permutations of the whole set of variables $x_1,\dots,x_n,y_1,\dots,y_n$, we see that it is enough to check that 
  \begin{align*}
    \sum_{j=0}^\infty 2^{2nsj} \int_{\reals^d} \int_{K^{2n}} \frac{|{\phi_j}(x - x_1) \dots {\phi_j}(x - x_n) {\phi_j}(x - y_1) \dots {\phi_j}(x - y_n)|}{|x_1 - y_1|^{\beta^2} \dots |x_n - y_n|^{\beta^2}} \, dx_1 \dots dx_n \, dy_1 \dots dy_n \, dx
  \end{align*}
  is finite. As for $j\geq 1$, the functions $\phi_j$ are built from $\phi_1$, we consider separately $j=0$ and $j\geq 1$. The summand for $j=0$ is clearly finite (by compact support and the fact that $\beta^2<d$). For $j\geq 1 $, pick a ball $B$  centered at the origin such that $K \subset B$. For the rest of the sum the change of variables $x_k \mapsto 2^{-j} x_k$, $y_k \mapsto 2^{-j} y_k$ and $x \mapsto 2^{-j} x$ yields the upper bound
  \begin{align*}
    \sum_{j=1}^\infty 2^{2nsj + n j \beta^2 - j d} \int_{\reals^d} \left(\int_{(2^j B) \times (2^j B)} \frac{|\phi_1(x - x_1) {\phi}_1(x - y_1)|}{|x_1 - y_1|^{\beta^2}} \, dx_1 \, dy_1\right)^n \, dx,
  \end{align*}
  where we have used the fact that ${\phi}_j(x) = 2^{dj} {\phi}_1(2^j x)$.
  Comparing with our statement, we see that it is enough to check that 
  \[\int_{\reals^d} 2^{-jd} \left(\int_{(2^j B) \times (2^j B)} \frac{|{\phi}_1(x - x_1) {\phi}_1(x - y_1)|}{|x_1 - y_1|^{\beta^2}} \, dx_1 \, dy_1\right)^n \, dx\]
  is uniformly bounded in $j$. Notice that for $x \in 2^{j+1}B$ we have 
  \begin{align*}
    & \int_{(2^j B) \times (2^j B)} \frac{|{\phi}_1(x - x_1) {\phi}_1(x - y_1)|}{|x_1 - y_1|^{\beta^2}} \, dx_1 \, dy_1 \\
    \lesssim&  \int_{\reals^d} \frac{1}{(1 + |x - x_1|^{2d})(1 + |x - y_1|^{2d})|x_1-y_1|^{\beta^2}} \, dx_1 \, dy_1 \le c',
  \end{align*}
  as the integral is constant in $x$.
  Moreover, for $x \notin 2^{j+1}B$ we have
  \begin{align*}
    & \int_{x \notin 2^{j+1} B} 2^{-jd} \left( \int_{(2^j B) \times (2^j B)} \frac{|{\phi}_1(x - x_1) {\phi}_1(x - y_1)|}{|x_1 - y_1|^{\beta^2}}  \right)^n \, dx \\
    \le & \int_{x \notin 2^{j+1} B} 2^{-jd} \left( \int_{(2^j B) \times (2^j B)} \frac{1}{(1 + |x - x_1|^{2d})(1 + |x - y_1|^{2d})|x_1-y_1|^{\beta^2}} \, dx_1 \, dy_1 \right)^n \, dx \\
    \le & \int_{x \notin 2 B} 2^{(2jd - j \beta^2)n} \left(\int_{B \times B} \frac{1}{(1 + 2^{2dj } |x - x_1|^{2d })(1 + 2^{{2dj }} |x - y_1|^{2d }) |x_1 - y_1|^{\beta^2}} \, dx_1 \, dy_1\right)^n \, dx ,     \\        
    \lesssim & \int_{x \notin 2 B} \frac{2^{n(-2jd - j \beta^2)}}{|x|^{4dn}} \left(\int_{B \times B} \frac{1}{|x_1 - y_1|^{\beta^2}} \, dx_1 \, dy_1\right)^n \, dx
  \end{align*}
    which goes to $0$ as $j \to \infty$. This concludes the proof of $\psi\mu\in B_{2n,2n}^{s}(\R^{d})$ almost surely, and thus by our discussion at the beginning of the proof, this implies that for $s<-\frac{\beta^{2}}{2}$, $\mu\in B_{p,q,loc}^{s}(U)$ almost surely.
    
  We then turn to the converse direction. In this case one deduces from \eqref{eq:betrie},\eqref{eq:bebe} and \eqref{eq:beem} that it is enough to verify for any fixed $s < \frac{\beta^2}{2}$  that almost surely $\psi\mu \notin B^{-s}_{1,1}$. From \eqref{eq:blowup_event} we know that if we let $f_k(x) = \psi(x) e^{-i \beta X_{1/k}(x)}$ with $\psi$ as in the proof of Theorem~\ref{thm:l2_not_measure}, then for any $\delta > 0$ there exists a deterministic constant $C$ and a stochastic sequence $n_k \to \infty$ such that $|\mu(f_{n_k})| \ge Cn_k^{\frac{\beta^2}{2} - \delta}$ with probability one. By the duality of $B^{-s}_{1,1}$ and $B^s_{\infty,\infty}$ we thus have
  \[\|\mu\|_{B^{-s}_{1,1}} \ge \frac{Cn_k^{\frac{\beta^2}{2} - \delta}}{\|f_{n_k}\|_{B^s_{\infty,\infty}}},\]
  and hence it is enough to show that for all fixed $\delta > 0$  the inequality $\|f_n\|_{B^s_{\infty,\infty}} \le n^{s + 2\delta}$ holds almost surely for large enough $n$.
  We will prove this bound first in the case when $s < 1$. The norm $\|f_{k}\|_{B^s_{\infty,\infty}}$ is equivalent to the H\"older norm of $f_{k}$, and since $t \mapsto e^{-i\beta t}$ is Lipschitz, it is enough to consider the $C^s$-norm of $X_{1/k}$.
  In order to bound this, we note first that for a fixed $\delta\in (0,(1-s)/2)$ 
  \[ \|X_{1/n}\|_{C^s} \sim c  \|I^{\delta} X_{1/n}\|_{C^{s + \delta}} \sim \|(I^{\delta} X)_{1/n}\|_{C^{s + \delta}},\]
  where $I^{\delta}$ is the  standard lift operator \eqref{eq:lift} (see the definition in Subsection \ref{sec:spaces})  $I^{\delta}\colon C^{s} \to C^{s + \delta}$, and $c > 0$ is a constant. By Lemma \ref{le:tarkempi} we have $I^{\delta} X \in C^{\delta/2}$ almost surely, and thus by Fernique's theorem
  \[\E \exp\big({a \|I^{\delta} X\|^2_{C^{\delta/2}}}\big) < \infty\]
  for some $a > 0$. Moreover, we may compute directly from the definition of a convolution that
  \begin{align}
    |(I^\delta X)_{1/n}(x) - (I^\delta X)_{1/n}(y)| & \le  \|I^\delta X\|_{\infty}\int |\eta_{1/n}(x-u) - \eta_{1/n}(y-u)| \, du \label{eq:Idelta_bound}\\
    & \le b\|I^\delta X\|_{\infty} \min(1, n |x-y|)  \nonumber
  \end{align}
  for some constant $b > 0$.
  Thus
  \begin{align*}
    \|(I^\delta X)_{1/n}\|_{C^{s + \delta}} & \le b\sup_{|x-y| \le 1} |x-y|^{-s-\delta} \min(1, n |x-y|) \|I^\delta X\|_{\infty} + \|(I^\delta X)_{1/n}\|_{\infty}\\
    & \lesssim (n^{s + \delta} + 1) \|I^\delta X\|_{C^{\delta/2}}.
  \end{align*}
  By the Fernique bound we have
  \begin{align*}
    \P (\|(I^\delta X)_{1/n}\|_{C^{s + \delta}} > n^{s + 2\delta}) \le \P((n^{s + \delta} + 1) \|I^\delta X\|_{C^{\delta/2}} \ge n^{s + 2\delta}) \le e^{-b' n^{\delta}}
  \end{align*}
  for some constant $b' > 0$.
  Finally,  by Borel--Cantelli  $\|f_{n}\|_{C^s} \le n^{s + 2\delta}$ for all large enough $n\geq n(\omega)$. This is precisely what we set out to prove, so we are done in the $s<1$ case.

 In the case of $s \ge 1$, we may actually choose $s > 1$ and we need to get an estimate for the H\"older norm of the derivatives of $X_{1/n}$. This is obtained by applying estimates like \eqref{eq:Idelta_bound} by replacing the test function $\eta$ by its derivatives. We leave the details for the reader.

    (ii) \; The proof is identical to that in case (i) as one invokes Proposition \ref{prop:electrostatic_GFF}. 

  (iii)\;  The claims for the Triebel--Lizorkin spaces follow easily from those for the Besov spaces by employing the embeddings \eqref{eq:betrie}.
\end{proof}
Combining Theorem \ref{thm:l2_not_measure} and Theorem \ref{th:besov} yields Theorem \ref{th:regularity}, so this concludes our study of regularity properties of imaginary chaos and we turn to what we refer to as universality properties.

\subsection{Universality properties}\label{subsec:universality} The goal of this section is to study the following question: For which periodic functions $H$ can we make sense of $H(X)$ (through a suitable regularization and renormalization procedure) when $X$ is a log-correlated field? To give an intuitive answer to this question, let us assume that $H$ is a $2\pi/\beta$-periodic\footnote{This is simply a notationally convenient way to write the arbitrary period of the function as it will work well with the notation we have used previously.}  function and let us expand $H(X_n)$ as a Fourier series $H(X_n(x))=\sum_{k\in \Z}H_ke^{ik\beta X_n(x)}$. Now if $H_0\neq 0$, we would expect from Proposition \ref{prop:mu_convergence} that $H(X_n(x))\to H_0$ as $n\to \infty$. If on the other hand $H_0=0$, and $\beta$ is small enough, then one would expect that multiplying by $e^{\frac{\beta^2}{2}\E [X_n(x)^2]}$ and letting $n\to\infty$ would pick out the $k=\pm 1$-terms and yield $H_1 e^{i\beta X(x)}+H_{-1}e^{-i\beta X(x)}$. If $H_{\pm 1}=0$ and $\beta$ is small enough, one would expect convergence to $H_2 e^{2i\beta X(x)}+H_{-2}e^{-2i\beta X(x)}$ and so on. To make this argument rigorous, one needs to control the contribution of the higher Fourier modes. For simplicity we shall assume from now on that $H$ is real, even, $H_0=0$, and $H_1\neq 0$, though these assumptions can be relaxed, see Remark \ref{re:generalH} below. 

Before proceeding, let us address a technicality that might concern a careful reader. If $H$ is not very regular, say just measurable instead of continuous, one might worry whether or not $\int H(X_n(x))\varphi(x)dx$ is a well defined random variable. That is, if $\widetilde{H}=H$ Lebesgue almost everywhere, do we have $\int H(X_n(x))\varphi(x)dx=\int \widetilde{H}(X_n(x))\varphi(x)dx$ almost surely?  To see that this is the case,  note that if  $X_n$ is  a centered Gaussian field with continuous realisations on the bounded domain $U\subset\R^d$, pointwise non-degenerate (i.e.
$\E X_n(x)^2>0$ for each $x\in U$), and   $H:\R\to\R$ is a locally  bounded function, then for any bounded compactly supported measurable function $\varphi$, the evaluation
$$
Y:=\int_UH(X_n(x))\varphi(x)dx
$$
is well-defined as a random variable. Indeed, we may choose Borel measurable representatives for the functions $H$ and $\varphi$, and it follows that $(x,\omega)\mapsto H(X_n(x,\omega))\varphi(x)$ is jointly measurable. Moreover, given another Borel measurable representative $\widetilde H$,  one has  a.s. $H(X_n(x,\omega))=\widetilde H(X_n(x,\omega))$ for almost every $x\in U$  by Fubini's theorem and the fact that Gaussians have continuous density on $\R^d.$ Hence moving to $\widetilde H$ does not change the value of $Y$, and we do not need to assume much regularity from $H$ to pose a meaningful question.

In what follows we assume again that $(X_n)$ are standard convolution approximations of our log-correlated field $X$ on the domain $U\subset\R^d$. To be more precise, we write $X_n:=X*\eta_{c_n}$ for some sequence $c_n\to 0$  as in Lemma \ref{le:standard},  and we recall that the covariance $C_n(x,y):=C_{X_n}(x,y)$  satisfies for any compact subset $K\subset U$, that there exists a $M=M(K)$ such that
\begin{equation}\label{eq:a}
\Big |C_n(x,y)-\log \big(\frac{1}{\max (c_n, |x-y|)}\big)\Big|\leq M\quad\textrm{for all}\quad x,y\in K,
\end{equation}
as $n\to\infty$.

The following lemma is instrumental in controlling the contribution of higher order Fourier modes. We are able to obtain a result for a slightly larger class of functions $H$ when specializing to two dimensions and assuming some further regularity from $g$, and for this reason, we also prove a slightly stronger version of our control of higher Fourier modes in the case of $d=2$.

\begin{lemma}\label{le:a} Let $X$ be a log-correlated field satisfying assumptions \eqref{eq:cov} and \eqref{eq:assumptions} and let $(X_n)$ be a convolution approximation of it as described above. Assume that $\beta\in (0,\sqrt{d})$  and $\varphi\in L^\infty (U)$ has compact support.
Denote for $k\in \Z$
$$
Y_k:= \int_U \varphi(x)e^{\frac{1}{2}\beta^2C_n(x,x)}e^{ik\beta X_n(x)}dx.
$$
\noindent {\bf (i)}\quad
For all   integers $k$ with  $|k|\geq 2$  it holds that
\begin{equation}\label{eq:claim1}
 \E | Y_k|^2 \lesssim c_n^{\alpha}\|\varphi\|^2_{L^\infty},
\end{equation}
  where $c_n$ is as in \eqref{eq:a},  $\alpha =\min (3\beta^2, d-\beta^2),$ and in the special case $\beta=\frac{1}{2}\sqrt{d}$ the factor on the right hand side must be replaced by $c_n^{3\beta^2}\log \frac{1}{c_n}.$ The bounds are uniform in $k$.

\smallskip

\noindent{\bf (ii)}\quad Assume that $d=2$, $g\in C^2(U\times U)$ and assume that the bump function $\eta$ used to define the convolution approximations $X_n$  is additionally  non-negative,  radially decreasing and symmetric, so that  $\eta(0)>0$. Moreover, assume that the term $g$ in the covariance satisfies $g(x,x)=g(y,y)$  for all $x,y\in U$. Then for all integers $\ell,k$ with $|\ell|,|k|\geq 4/\beta$ and for $n$ large enough it holds that
\begin{equation}\label{eq:claim2}
| \E Y_k \overline {Y_\ell}| \lesssim \frac{(e^{2M}c_n)^{-\beta^2+2+\frac{\beta^2}{8}(\ell-k)^2}}
{(|\ell| \vee |k|)^{2}}\|\varphi\|^2_{L^\infty}.
\end{equation}
\end{lemma}
\begin{proof}
(i) \quad We may assume that $\|\varphi\|_{L^\infty}=1$ and denote $K:=\textrm{supp}(\varphi)\subset U$, so that $K$ is compact. A direct computation yields the upper bound
$$
\E|Y_n|^2\leq I_n:=\|\varphi\|_{L^\infty}^2\int_{K\times K}\exp\Big( \beta^2k^2C_n(x,y)-\frac{1}{2}\beta^2(k^2-1)(C_n(x,x)+C_n(y,y)\Big)dxdy.
$$
In the range $k^2\beta^2 <d$  the term  $\exp(\beta^2k^2C_n(x,y))$ is uniformly integrable in $n$, As $|C_n(x,x)-\log (1/c_n)|\lesssim 1$ for all $x$ we infer that $I_n\lesssim c_n^{(k^2-1)\beta^2}.
$ In the case $k^2\beta^2=d$ we obtain a similar bound where one just adds the extra factor 
\begin{eqnarray*}
  \int_{U\times U}e^{dC_n(x,y)}dxdy \sim \int_{|x-y|\leq c_n}c_n^{-d}+\int_{|x-y|\geq c_n}|x-y|^{-d}\sim \log(1/ c_n).
\end{eqnarray*}
In the generic situation $k^2\beta^2>d$ we observe first that due to the covariance inequality
  \begin{equation}\label{eq:cov_inequality}
    C_n(x,y)\leq \frac{1}{2}(C_n(x,x)+C_n(y,y)),
  \end{equation}
  the integrand in $I_n$ is upper bounded by $\exp (\beta^2 C_n(x,y)).$ We use this estimate in the part of the product domain where $|x-y|\leq e^{2M}c_n$ and note that for the remaining values $|x-y|> e^{2M}c_n$, where $M$ is  from \eqref{eq:a},  we have 
\begin{eqnarray*}
 &&\beta^2k^2C_n(x,y)-\frac{1}{2}\beta^2(k^2-1)(C_n(x,x)+C_n(y,y))\\
&\leq& \beta^2k^2(\log(|x-y|^{-1}) + M - \beta^2(k^2-1)(\log (1/c_n)-M)\\
&\leq& \beta^2k^2\log\big(\big|(x-y)e^{-2M}\big|^{-1}\big)- \beta^2(k^2-1)\log (1/c_n).
\end{eqnarray*}
Thus
\begin{eqnarray*}
I_n&\lesssim& \int_{|x-y|\leq e^{2M}c_n}|x-y|^{-\beta^2}dxdy+ 
c_n^{\beta^2(-1+k^2)}\int_{|x-y|>e^{2M}c_n}|(x-y)e^{-2M}|^{-k^2\beta^2}dxdy\\
&\lesssim& c_n^{-\beta^2+d},
\end{eqnarray*}
where in the latter integral one  performs a change of variables $(x,y)=(e^{2M}x',e^{2M}y')$.
The claim follows by combining our estimates for different values of $k$.

\smallskip

(ii)\quad We use the same notation as in the proof of part (i). Consider first the case where $\ell$ and $k$ have the same sign, so that we may assume $k,\ell >0.$ We claim first that given any constant $A>0$, for points $x,y\in K$  it holds with a constant $\delta= \delta (K,A,g)>0$ and large enough $n\geq n_0(K,A,g)$ that
\begin{equation}\label{eq:apu1a}
C_n(x,y)\leq  C_n(x,x)-\delta \big(|x-y|/c_n)^2\quad \textrm{if}\quad |x-y|\leq Ac_n.
\end{equation}
  This auxiliary result will be used later on in the proof.
In order to verify \eqref{eq:apu1a}, we fix $y_0\in K$  and note that 
\begin{eqnarray*}
C_n(x,y_0)&=& \big(\widetilde \eta_{c_n}*\log(|\cdot|^{-1})\big)(x-y_0)+ \big((\eta_{c_n}\otimes \eta_{c_n})*g)(x,y_0)\\
&=:& V_n(x)+ W_n(x)
\end{eqnarray*}
  where $\widetilde \eta :=\eta*\eta.$ Since $C_n(x,x)$ is independent of $x$, it follows from the covariance inequality \eqref{eq:cov_inequality} that $V_n(x) + W_n(x)$ has a maximum at $x = y_0$ and we have $\nabla (V_n+W_n)(y_0)=0$.
By symmetry considerations $\nabla V_n(y_0)=0$, whence also $\nabla W_n(y_0)=0.$ Since $D^2 W_n$ is  bounded in any compact subdomain  of $U$, uniformly in $n$, we may easily infer  the uniform bound
$$
W_n(x)-W_n(y_0)\leq C|x-y_0|^2,
$$
valid uniformly for   $(x,y_0)\in K\times K$, and $C=C(K).$
  On the other hand, the function $V_0:= \big(\widetilde \eta_1* \log(|\cdot|^{-1})\big)(x-y_0)$, defined for all $x\in\R^2$,   obtains its unique maximum at the point $y_0$ by (an integral version of) the Hardy--Littlewood rearrangement inequality, and as the  logarithm yields the fundamental solution of the Laplacian in the plane, we have $\Delta V_0(y_0)=-2\pi \widetilde \eta_1(0)<0.$ As $V_0$ is radial  with respect to $y_0$ it follows easily that for any given $A\geq 1$  there is  $\delta =\delta (A) >0$ such that $V_0(x)-V_0(y_0)\leq -2\delta |x-y_0|^2$ 
for $|x-y_0|\leq A$. Then the scaling properties of the logarithm yield that
$$
V_n(x)-V_n(y_0)\leq -2\delta (|x-y_0|/c_n)^2\quad \textrm{for}\quad |x-y_0|\leq Ac_n.
$$
By combining this with our previous estimate for $W_n(x)-W_n(y_0)$ the inequality \eqref{eq:apu1a}
follows for large enough $n.$

We now move to actually estimating $\E|Y_k\overline{Y_\ell}|$. We recall that $K\subset U$ is the topological support of $\varphi$ and compute 
\begin{align*}
\left| \E Y_k \overline {Y_\ell}\right| &= \left|\int_{K\times K}\varphi(x)\overline{\varphi(y)}\exp\bigg[
 \ell k\beta^2C_n(x,y)- \frac{\beta^2}{2}\big( (\ell^2-1)C_n(x,x)+(k^2-1) C_n(y,y)\big)
 \bigg]dxdy\right|\\
 &\lesssim  c_n^{-\beta^2}\int_{K\times K}\exp\bigg[
 \ell k\beta^2C_n(x,y)- \frac{\beta^2}{2}\big( \ell^2C_n(x,x)+k^2 C_n(y,y)\big)
 \bigg]dxdy\\
 &= c_n^{-\beta^2}\int_{\{|x-y|\leq e^{2M}c_n\}\cap K\times K}e^{
 \ell k\beta^2C_n(x,y)- \frac{\beta^2}{2}\big( \ell^2C_n(x,x)+k^2 C_n(y,y)\big)
 }dxdy \\
 &\qquad + c_n^{-\beta^2}\int_{\{|x-y|> e^{2M}c_n\}\cap K\times K}e^{
 \ell k\beta^2C_n(x,y)- \frac{\beta^2}{2}\big( \ell^2C_n(x,x)+k^2 C_n(y,y)\big)
 }dxdy \\
 &=: I^1_n+I^2_n.
\end{align*}

In the set $\{|x-y|> e^{2M}c_n\}\cap K\times K$ we may estimate 
\begin{eqnarray*}
 &&\ell k\beta^2C_n(x,y)- \frac{\beta^2}{2}\big( \ell^2C_n(x,x)+k^2 C_n(y,y)\big)\\
 &\leq& \ell k\beta^2(\log(|x-y|^{-1})+M)- \frac{\beta^2}{2}( \ell^2+k^2 )(\log (1/c_n)-M)\\
 &\leq& \ell k\beta^2\big(\log\big (|e^{2M}|x-y|^{-1}|\big)-M\big)- \frac{\beta^2}{2}( \ell^2+k^2 )(\log (1/c_n)-M)\\
 &\leq& \ell k\beta^2\log\big (|e^{2M}|x-y|^{-1}|\big)- \frac{\beta^2}{2}( \ell^2+k^2 )\log (1/c_n) +\frac{M\beta^2}{2}(\ell-k)^2.
\end{eqnarray*}
We denote $M':= e^{M}$ and perform the change of variables   $u=(x-y)/(M')^2$,  $v=(x+y)/(M')^2$. After integration first with respect to the variable $v$ it follows that
\begin{eqnarray*}
 I^2_n&\lesssim&c_n^{-\beta^2+\frac{\beta^2}{2}(\ell^2+k^2)} M'^{\beta^2(\ell-k)^2/2} 
\int_{|u|\geq c_n}|u|^{-k\ell\beta^2}du \\
&\lesssim&\frac{ (c_nM')^{-\beta^2+2+\frac{\beta^2}{2}(\ell-k)^2}}{k\ell\beta^2 - 2}\lesssim
\frac{ (c_nM')^{-\beta^2+2+\frac{\beta^2}{2}(\ell-k)^2}}{\ell k}\lesssim \frac{ (c_nM')^{-\beta^2+2+\frac{\beta^2}{4}(\ell-k)^2}}{(\ell \vee k)^2},
\end{eqnarray*}
where in the second last inequality we used the fact that $\frac{1}{4}k\ell \beta^2 \geq 2$, which follows from our assumption that $|k|,|\ell|\geq 4/\beta$. In turn,  the last inequality follows by noting that we may assume $\ell >k$, and by considering separately the cases $\ell\geq 2k$ and $2k>\ell$. Naturally, we need to assume that $n$ is large enough so that, say,  $c_n^{-1} >2M'.$

Next, for $I^1_n$ we have $|x-y|\leq e^{2M}c_n$. Using \eqref{eq:apu1a} yields that
\begin{eqnarray*}
 &&\ell k\beta^2C_n(x,y)- \frac{\beta^2}{2}\big( \ell^2C_n(x,x)+k^2 C_n(y,y)\big)\\
  &\leq& \ell k\beta^2(C_n(x,x)- \delta(|x-y|/c_n)^2) - \frac{\beta^2}{2}\big( \ell^2 + k^2 \big)C_n(x,x)\\
 &\leq& (\ell -k)^2\frac{1}{2}\beta^2(\log c_n+M)-\big((\delta \ell k)^{1/2}\beta |x-y|/c_n\big )^2.
\end{eqnarray*}
We thus obtain
\begin{eqnarray*}
  I^1_n&\lesssim&c_n^{-\beta^2} (c_nM')^{\frac{\beta^2}{2}(\ell -k)^2}\int_{K} \int_{\reals^2} e^{-((\delta \ell k)^{1/2} \beta |x-y|/c_n)^2} \,dx\,dy \lesssim
\frac{(c_nM')^{-\beta^2+2+\frac{\beta^2}{2}(\ell -k)^2}}{\ell k},
\end{eqnarray*}
and this is transformed to the desired form as before.

Finally, the case where $k$ and $\ell$ have different sign is much easier since then the term $\ell k\beta^2C_n(x,y)$ has negative sign and works to our favour.
\end{proof}

We are now in a position to prove our universality result.

\begin{theorem}\label{th:kosini} {\bf (i)}\quad
  Let $(X_n)_{n \ge 1}$ be a convolution approximation of a log-correlated field $X$ as in Lemma \ref{le:a} and let $0<\beta<\sqrt{d}$. Assume that $H:\R\to\R$ is a $2\pi/\beta$-periodic even function with absolutely convergent Fourier series and mean zero. Then there is a constant $a$ such that for every test function $\varphi\in C^\infty_c(U)$ we have
$$
  \int_U \varphi (x)e^{\frac{1}{2}\beta^2C_n(x,x)}H(X_n(x))dx\to \langle a\textrm{``}\cos (\beta X)\textrm{''}, \varphi\rangle ,
$$
in probability as $n\to\infty.$ 

\smallskip

 \noindent {\bf (ii)}\quad If $d=2$ and $X_n,X$ satisfy the condition of part $\mathrm{(ii)}$ of the previous lemma, we have the same conclusion as in part $\mathrm{(i)}$ of this theorem, but assuming only that $H$ is a locally integrable  $2\pi/\beta$-periodic even function with mean zero.
\end{theorem}
\begin{proof} (i)\quad Let $H(x)=\sum_{k=1}^\infty \widehat H_k\cos(\beta kx)$. By Theorem \ref{th:existuniq} it is enough to check that for a test function $\varphi$ the quantity
$$
R:= \sum_{|k|\geq 2} \widehat H_k\int_U e^{\frac{1}{2}\beta^2C_n(x,x)}\varphi(x)\big( e^{i k\beta X_n(x)}+ e^{-i k\beta X_n(x)})dx
$$
converges to zero in probability. Since $\sum_{|k|\geq 2} |\widehat H_k|<\infty$ by assumption, this follows from  Lemma
\ref{le:a}(i) combined with two basic Cauchy--Schwarz estimates as one then finds that  $\E R^2\to 0$ as $n\to\infty.$

\smallskip

(ii)\quad We aim to show that again $\E R'^2\to 0$ as $n\to\infty,$ where we now define
$$
R':= \sum_{|k|\geq k_0}\widehat H_k\int_U e^{\frac{1}{2}\beta^2C_n(x,x)}\varphi(x)\big( e^{i k\beta X_n(x)}+ e^{-i k\beta X_n(x)})dx,
$$
  where $k_0> 2\sqrt{d}/\beta$. The finite number of terms with $2\leq |k|<k_0$ can be handled as in case (i). Since e.g. Fej\'{e}r partial sums of the Fourier series converge to $H$ almost everywhere pointwise, Fatou's lemma allows us to assume that $H$ is a trigonometric polynomial and it is enough to prove a uniform bound for  $\E R^2$ over all trigonometric polynomials $H$ such that the modulus  of all of their Fourier coefficients is bounded by 1. However, by Lemma
\ref{le:a}(ii) we obtain in this situation
\begin{eqnarray*}
  \E R'^2 &\leq& \sum_{|k|,|\ell|\geq k_0} |\widehat H(k)\overline{\widehat H(\ell)} \E Y_k \overline {Y_\ell}| \lesssim c_n^{-\beta^2+2} \sum_{|k|,|\ell |\geq k_0}\frac{(c_n e^{2M})^{\frac{\beta^2}{4}(\ell-k)^2}}{(|\ell| \vee |k|)^{2}}\\
&\lesssim& c_n^{-\beta^2+2}\to 0\quad \textrm{as}\quad n\to\infty. 
\end{eqnarray*}
\end{proof}

\begin{remark}\label{re:generalH} The second part of the result applies to e.g.  $*$-scale invariant log-correlated fields since they typically have translation invariant covariance structure. The same proof of course yields that if $H$ is any complex valued $2\pi/\beta$-periodic function with zero mean and absolutely convergent Fourier series, the limit is a linear combination of the imaginary chaoses $\textrm{``}e^{\pm i\beta X}\textrm{''}$.  
\end{remark}

This concludes our study of universality and now we discuss the behavior of $\mu$ near $\beta_c$.

\subsection{Approach to the critical point}\label{subsec:critical} As we have mentioned before and as follows from results in \cite{LRV}, $e^{\frac{\beta^2}{2}\E [X_n(x)^2]+i\beta X_n(x)}$ does not converge for $\beta\geq \sqrt{d}$, at least if one assumes a bit more of $g$ and the approximation $X_n$. Nevertheless, if one multiplies this quantity by a suitable deterministic one, then one can prove convergence to white noise. In this section, we study how this fact that $\beta_c:=\sqrt{d}$ is a special point can be seen from the limiting objects $\mu$. In what follows, we find it convenient to write $\mu_\beta$ to indicate the dependence on $\beta$ and hope this notation causes no confusion. The main result of this section is the following which describes how $\mu_\beta$ blows up as $\beta$ increases to $\sqrt{d}$. The theorem complements in a natural manner some results in \cite{LRV}, and the methods used in the proof are somewhat similar to the ones already employed in that paper.

\begin{theorem}\label{thm:crit} Let  $X$ be a log correlated field on the bounded subdomain $U\subset\R^d$ satisfying the standard assumptions \eqref{eq:assumptions} as before. Fix any test function $f\in C_c^\infty (U).$
  As $\beta \nearrow\sqrt{d}$, we have 
  $$
\sqrt{\frac{d - \beta^2}{|S^{d-1}|}} \mu_\beta(f) \to \int_U f(x)e^{\frac{\beta^2}{2} g(x,x)} W(dx)
  $$
   in law, where $W$ is the standard complex white noise on $U$,\footnote{Our notation here is slightly formal; $Z_h:=\int h(x)W(dx)$ denotes a centered complex Gaussian random variable satisfying $\E Z_h^2=0$ and $\E|Z_h|^2=\int_U |h(x)|^2dx$.} and $|S^{d-1}|$ denotes the ``area'' of the unit sphere of $\R^{d}$.
\end{theorem}

\begin{proof}
  As we are dealing with Gaussian random variables, it is enough to show that the moments converge and we start by computing the second absolute one; we will implicitly be using constantly the results from Section \ref{subsec:moments} which allow us to write all the moments as suitable integrals. We have
  \begin{align*}
    \frac{d - \beta^2}{|S^{d-1}|} \E |\mu_\beta(f)|^2 & = \frac{d - \beta^2}{|S^{d-1}|} \int_{|x-y| < (d - \beta^2)^{\frac{1}{2d}}} f(x)\overline{f(y)} \frac{e^{\beta^2 g(x,y)}}{|x-y|^{\beta^2}} \, dx \, dy \\
    & \quad + \frac{d - \beta^2}{|S^{d-1}|} \int_{|x-y| > (d - \beta^2)^{\frac{1}{2d}}} f(x)\overline{f(y)} \frac{e^{\beta^2 g(x,y)}}{|x-y|^{\beta^2}} \, dx \, dy.
  \end{align*}
  The trivial estimate $\frac{1}{|x-y|^{\beta^2}} \le \frac{1}{(d - \beta^2)^{\frac{\beta^2}{2d}}}$ shows that the second term goes to $0$ as $\beta \nearrow\sqrt{d}$.
  This and uniform continuity of our test function $f$ and the function $g$ on the support of $f$  easily  gives us
  \begin{align*}
    \lim_{\beta \nearrow \sqrt{d}} \frac{d - \beta^2}{|S^{d-1}|} \E |\mu_\beta(f)|^2 & = \lim_{\beta \nearrow \sqrt{d}} \frac{d - \beta^2}{|S^{d-1}|} \int_{|x-y| < (d - \beta^2)^{\frac{1}{2d}}} \frac{|f(x)|^2 e^{\beta^2 g(x,x)}}{|x-y|^{\beta^2}} \, dx \, dy \\
    & = \lim_{\beta \nearrow \sqrt{d}} \frac{d - \beta^2}{|S^{d-1}|} \int_U |f(x)|^2 e^{\beta^2 g(x,x)}\int_{y \in B(x,(d-\beta^2)^{\frac{1}{2d}})} |x-y|^{-\beta^2} \, dy \, dx \\
    & = \lim_{\beta \nearrow\sqrt{d}} \frac{d - \beta^2}{|S^{d-1}|} \int_U |f(x)|^2 e^{\beta^2 g(x,x)} |S^{d-1}| \int_0^{(d - \beta^2)^{\frac{1}{2d}}} r^{d-1-\beta^2} \, dr \, dx \\
    & = \lim_{\beta \nearrow\sqrt{d}} (d - \beta^2)^{\frac{d - \beta^2}{2d}} \int_U |f(x)|^2 e^{\beta^2 g(x,x)} \, dx \\
    & = \int_U |f(x)|^2 e^{\beta^2 g(x,x)} \, dx.
  \end{align*}
  Next note that for mixed moments  we have by Lemma~\ref{lemma:moments} and the above computation that 
  \begin{align*}
    \Big(\frac{d - \beta^2}{|S^{d-1}|}\Big)^{\frac{a+b}{2}} \left|\E \mu_\beta(f)^a \overline{\mu_\beta(f)}^b \right| & \le C_{a,b} \Big(\frac{d - \beta^2}{|S^{d-1}|}\Big)^{\frac{a+b}{2}} (\E |\mu_\beta(f)|^2)^{\min(a,b)} \\
    & \lesssim (d - \beta^2)^{\frac{a+b}{2} - \min(a,b)},
  \end{align*}
  where the right hand side tends to $0$ as $\beta \nearrow \sqrt{d}$. Thus it remains to check that the moments $\Big(\frac{d - \beta^2}{|S^{d-1}|}\Big)^a \E |\mu_\beta(f)|^{2a}$ behave correctly. We have
  \[\E |\mu_\beta(f)|^{2a} = \int_{U^{2a}} \Big(\prod_{j=1}^a dx_j dy_j f(x_j) \overline{f(y_j)}\Big) \frac{\prod_{1 \le j < k \le a} |x_j - x_k|^{\beta^2} |y_j - y_k|^{\beta^2}e^{-\beta^2 g(x_j, x_k) - \beta^2 g(y_j, y_k)}}{\prod_{1 \le j, k \le a} |x_j - y_k|^{\beta^2} e^{-\beta^2 g(x_j, y_k)}}.\]
  We may split the integration domain into the $a!$ disjoint sets $A_\sigma$, $\sigma \in S_a$, and the  complement of their union, where
  \begin{align*}
    A_\sigma & = \{|x_i - y_{\sigma_i}| < (d - \beta^2)^{\frac{1}{2d}} \text{ for all } 1 \le i \le a\} \\
             & \quad \cap \{|x_i - x_j| > (d - \beta^2)^{\frac{1}{3d}} \text{ for all } 1 \le i < j \le a\}.
  \end{align*}
  Consider the integral over $A_e$, where $e$ is the identity permutation. In $A_e$ we have for $j < k$ that
  \[\frac{|x_j - x_k|}{|x_j - y_k|} \ge \frac{|x_j - x_k|}{|x_j - x_k| + |x_k - y_k|} \ge \frac{1}{1 + \frac{(d - \beta^2)^{\frac{1}{2d}}}{(d - \beta^2)^{\frac{1}{3d}}}}\]
  and
  \[\frac{|x_j - x_k|}{|x_j - y_k|} \le \frac{|x_j - x_k|}{|x_j - x_k| - |x_k - y_k|} \le \frac{1}{1 - \frac{(d - \beta^2)^{\frac{1}{2d}}}{(d - \beta^2)^{\frac{1}{3d}}}}\]
  from which we deduce that in $A_e$ 
  \[\frac{|x_j - x_k|}{|x_j - y_k|} \to 1\]
  as $\beta \to \sqrt{d}$. Similar reasoning shows that
  \[\frac{|y_j - y_k|}{|x_k - y_j|} \to 1.\]
  Hence again by uniform continuity of $g$ and $f$
  \begin{align*}
    & \lim_{\beta \nearrow \sqrt{d}} \Big(\frac{d - \beta^2}{|S^{d-1}|}\Big)^{a} \int_{A_e} \Big(\prod_{j=1}^a dx_j dy_j f(x_j) \overline{f(y_j)}\Big) \frac{\prod_{1 \le j < k \le a} |x_j - x_k|^{\beta^2} |y_j - y_k|^{\beta^2} e^{-\beta^2 g(x_j, x_k) - \beta^2 g(y_j, y_k)}}{\prod_{1 \le j, k \le a} |x_j - y_k|^{\beta^2} e^{-\beta^2 g(x_j, y_k)}} \\
    & = \lim_{\beta  \nearrow \sqrt{d}} \Big(\frac{d - \beta^2}{|S^{d-1}|}\Big)^{a} \int_{A_e} \frac{\prod_{j=1}^a dx_j dy_j |f(x_j)|^2 e^{\beta^2 g(x_j, x_j)}}{\prod_{1 \le j \le a} |x_j - y_j|^{\beta^2}} \\
    & = \lim_{\beta  \nearrow \sqrt{d}} \Big(\frac{d - \beta^2}{|S^{d-1}|}\Big)^{a} \int_{|x_i - x_j| > (d - \beta^2)^{\frac{1}{3d}}} \Big(\prod_{j=1}^a dx_j |f(x_j)|^2 e^{\beta^2 g(x_j, x_j)}\Big) \prod_{j=1}^a \int_{|y_j - x_j| < (d - \beta^2)^{\frac{1}{2d}}} \frac{dy_j}{|x_j - y_j|^{\beta^2}} \\
    & = \lim_{\beta  \nearrow \sqrt{d}} \Big(\frac{d - \beta^2}{|S^{d-1}|}\Big)^{a} \int_{|x_i - x_j| > (d - \beta^2)^{\frac{1}{3d}}} \Big(\prod_{j=1}^a dx_j |f(x_j)|^2 e^{\beta^2 g(x_j, x_j)}\Big) |S^{d-1}|^a \Big(\int_0^{(d - \beta^2)^{\frac{1}{2d}}} r^{d-1-\beta^2} \, dr\Big)^a \\
    & = \left(\int_U |f(x)|^2 e^{\beta^2 g(x, x)} \, dx \right)^a.
  \end{align*}
  
By relabelling $y_i$, we see that the result does not depend on the permutation chosen, so we get the same outcome $a!$ times.  Thus the moments converge to Gaussian ones as soon as we  check that the contribution from the complement of the sets $A_\sigma$ goes to $0$. The complement is covered by the sets
  \[B_1 = \{|x_j - x_k| \le (d - \beta^2)^{\frac{1}{3d}} \text{ for some } 1 \le j < k \le a\}\]
  and
  \[B_{2,k} = \{|x_k - y_j| > (d - \beta^2)^{\frac{1}{2d}} \text{ for all } j \neq k\}.\]

  We have
 \begin{align*}
& \lim_{\beta \to \sqrt{d}} \Big(\frac{d - \beta^2}{|S^{d-1}|}\Big)^{a} \int_{B_1} \Big(\prod_{j=1}^a dx_j dy_j f(x_j) \overline{f(y_j)}\Big) \frac{\prod_{1 \le j < k \le a} |x_j - x_k|^{\beta^2} |y_j - y_k|^{\beta^2} e^{-\beta^2 g(x_j, x_k) - \beta^2 g(y_j, y_k)}}{\prod_{1 \le j, k \le a} |x_j - y_k|^{\beta^2} e^{-\beta^2 g(x_j, y_k)}} \\
  &= 0
  \end{align*}
  because we may use Lemma~\ref{lemma:moment_inequality} and Fubini's theorem to integrate out the variables $y_k$, leaving a term of size $\lesssim (d - \beta^2)^{-a}$ that cancels the factor in front. The remaining integral over the variables $x_k$ is over a domain whose measure goes to $0$. Finally, again using Lemma~\ref{lemma:moment_inequality} we have
  \begin{align*}
    & \Big(\frac{d - \beta^2}{|S^{d-1}|}\Big)^{a} \int_{B_{2,a}} \Big(\prod_{j=1}^a dx_j dy_j f(x_j) \overline{f(y_j)}\Big) \frac{\prod_{1 \le j < k \le a} |x_j - x_k|^{\beta^2} |y_j - y_k|^{\beta^2} e^{-\beta^2 g(x_j, x_k) - \beta^2 g(y_j, y_k)}}{\prod_{1 \le j, k \le a} |x_j - y_k|^{\beta^2} e^{-\beta^2 g(x_j, y_k)}} \\
    & \lesssim \|f\|_\infty^{2a} \sum_{\sigma \in S_a} \Big(\frac{d - \beta^2}{|S^{d-1}|}\Big)^{a} \int_{B_{2,a}} \Big(\prod_{j=1}^a dx_j dy_j \Big) \frac{1}{\prod_{1 \le j \le a} |x_j - y_{\sigma_j}|^{\beta^2}} \\
    & \lesssim \sum_{\sigma \in S_a} \Big(\frac{d - \beta^2}{|S^{d-1}|}\Big)^{a} (d - \beta^2)^{-\frac{\beta^2}{2d}} \int_{B_{2,a}} \Big(\prod_{j=1}^a dx_j dy_j \Big) \frac{1}{\prod_{1 \le j \le a-1} |x_j - y_{\sigma_j}|^{\beta^2}} \\
    & \lesssim (d - \beta^2)^{1 - \frac{\beta^2}{2d}},
  \end{align*}
  which goes to $0$. A similar calculation holds for $B_{2,k}$, $1 \le k \le a-1$.
\end{proof}

This concludes the portion of this article dealing with basic properties of imaginary chaos. We now turn to discussing the Ising model.

\section{The Ising model and multiplicative chaos: the scaling limit of the critical and near critical planar XOR-Ising spin field}\label{sec:ising}

The goal of this section is to  prove Theorem \ref{th:ising} and Theorem \ref{th:perturb}. We begin by first recalling the definition of the Ising model (with $+$ boundary conditions) on a finite part of the square lattice as well as recent results concerning the scaling limit of correlation functions of the spin field for the critical Ising model on a finite part of the square lattice. We then define the XOR-Ising model on the square lattice and using the results concerning the correlation functions (along with some rough estimates for the behavior of the correlation functions on the diagonals), we prove Theorem \ref{th:ising}, namely that in zero magnetic field, the scaling limit of the critical XOR-Ising spin field is the real part of an imaginary multiplicative chaos distribution. After this, we prove that if we add a magnetic field to the XOR-Ising model, then the scaling limit of the spin field can be seen as the cosine of the sine-Gordon field, which is Theorem~\ref{th:perturb}.

\subsection{The Ising model and spin correlation functions for the critical planar Ising model} \label{sec:Isingdef}

Let $U\subset \C$ be a simply connected bounded planar domain, and for $\delta>0$, let $F_\delta$ be the set of faces of the lattice graph $\delta \integers^2$ that are contained in $U$. To avoid overlap, let us say that the faces are half-open, i.e. of the form $\delta([n,n+1)\times[m,m+1))$ for some $m,n \in \integers$. Following \cite{CHI}, we will define our Ising model on the faces $F_\delta$. We also define the set of boundary faces $\partial F_\delta$ as the set of those faces in $\delta \integers^2$ which are adjacent to a face in $F_\delta$ but not in $F_\delta$ themselves.

We call a function $\sigma\colon F_\delta \cup \partial F_\delta \to \lbrace -1,1\rbrace$, $a\mapsto \sigma_a$ a spin configuration on $F_\delta \cup \partial F_\delta$ and we define the Ising model on $F_\delta$ with $+$ boundary conditions, inverse temperature $\beta$, and zero magnetic field to be a probability measure on the set of spin configurations on $F_\delta \cup \partial F_\delta$ such that the law of the spin configuration is

$$
\P_\delta(\sigma)=\P_{\delta,\beta,U}^+(\sigma)=\frac{1}{Z_\beta}e^{\beta \sum_{a,b\in F_\delta \cup \partial F_\delta, a\sim b}\sigma_a\sigma_b}\mathbf{1}\lbrace \sigma_{|\partial F_\delta}=1\rbrace,
$$

\noindent where by $a\sim b$ we mean that $a,b\in F_\delta \cup \partial F_\delta$ are neighboring faces, and $Z_\beta$ is a normalizing constant. We count each pair $a,b$ of nearest neighbor faces only once. We will want to talk about the spin at an arbitrary point $x \in U$, so we define a function $\sigma_\delta(x) = \sigma_f$ if $x \in f \in F_\delta$, and $\sigma_\delta(x) = 1$ otherwise. 

As discussed in the introduction, a fundamental fact about the planar Ising model with zero magnetic field is that the model has a phase transition. From now on, we will focus on the critical model, namely when $\beta=\beta_c=\frac{\log(1+\sqrt{2})}{2}$ -- see \cite[Section 7.12]{Baxter}. We will also write from now on $\P_\delta=\P_{\delta,\beta_c,U}^+$ for the law of the critical Ising model (on the faces $F_\delta \cup \partial F_\delta$ with the $+$ boundary conditions as indicated above) as well as the law of the induced spin field $\sigma_\delta \colon U \to \{-1,1\}$.

We next turn to the analysis of the correlation functions of $\sigma_\delta$, which as we discussed in Section \ref{subsec:isingresults} have a non-trivial scaling limit and are connected to conformal field theory. The precise statement concerning the scaling limit is a recent result of Chelkak, Hongler, and Izyurov (see \cite[Theorem 1.2]{CHI} and the discussion leading to it):

\begin{theorem}[Chelkak, Hongler, and Izyurov]\label{th:CHI}
  Let $x_1,...,x_n\in U$ be distinct and the spin field $\sigma_\delta$ be distributed according to $\P_\delta$ {\rm (}as defined above{\rm )}. Then for $\mathcal{C}=2^{5/48}e^{\frac{3}{2}\zeta'(-1)}$,

\begin{align*}
\lim_{\delta\to 0^+}\delta^{-\frac{n}{8}}\E \left[\prod_{j=1}^n \sigma_\delta(x_j)\right]&=\mathcal{C}^n \prod_{j=1}^n \left(\frac{|\varphi'(x_j)|}{2 \mathrm{Im}\varphi(x_j)}\right)^{1/8}\\
&\qquad \times\left(2^{-n/2}\sum_{\mu\in\lbrace -1,1\rbrace^n}\prod_{1\leq k<m\leq n}\left|\frac{\varphi(x_k)-\varphi(x_m)}{\varphi(x_k)-\overline{\varphi(x_m)}}\right|^{\frac{\mu_k\mu_m}{2}}\right)^{1/2},
\end{align*}

\noindent where $\varphi:U\to \mathbf{H}=\lbrace x+iy\in \C:y>0\rbrace$ is any conformal bijection and for any $\varepsilon>0$, the convergence is uniform in $\lbrace x_1,...,x_n\in \Omega: \min_{i\neq j}|x_i-x_j|>\varepsilon, \min_{i}d(x_i,\partial U)>\varepsilon\rbrace$.
\end{theorem}

\begin{remark}{\rm
  We note that in \cite{CHI}, the authors consider actually the square lattice rotated by $\pi/4$ and with diagonal mesh $2\delta$ in which case the lattice spacing is $\sqrt{2}\delta$ instead of $\delta$ as in our case. Rotating the lattice plays a role only in the value of the constant $\mathcal{C}$. Our version follows by replacing their $\delta$ with $\delta/\sqrt{2}$.
  We also note that in \cite{CHI} there appears to be a sign error in the exponent of $e^{\frac{3}{2}\zeta'(-1)}$. We offer here a brief suggestion on how the interested reader might convince themselves of this fact. First of all, as pointed out in \cite[Remark 1.4]{CHI}, one can recover the (continuum) whole plane spin-correlation functions from the finite volume ones through a suitable limiting process. In particular, the scaling limit of the whole plane two-point function equals $\mathcal C^2|x-y|^{-1/4}$ (see \cite[(1.6)]{CHI}). On the other hand, it is known that on the whole plane $\integers^2$-lattice the \emph{diagonal} two point function has an explicit product representation -- see e.g. \cite[(XI.4.18)]{MCW}. This product can be written in terms of Barnes $G$ functions, and using their known asymptotics, one can recover the correct value of $\mathcal C$. We thank Antti Suominen for pointing this sign error out to us. \hfill$\blacksquare$
}
\end{remark}

\subsection{The critical XOR-Ising model and its magnetic perturbation}\label{sec:XOR} Following Wilson \cite{Wilson}, see also \cite{BdT}, we consider now the so called XOR-Ising model, which is again a probability measure on spin configurations, but now the spin configurations are given by a pointwise product of two independent Ising spin configurations. We focus on the critical case again and we thus make the following definitions: let $\sigma_\delta,\widetilde{\sigma}_\delta$ be independent and distributed according to $\P_\delta$ and define for $x\in U$, $\mathcal{S}_\delta(x)=\sigma_\delta(x)\widetilde{\sigma}_\delta(x)$. Also write for $a\in F_\delta$, $\mathcal{S}_a=\sigma_a\widetilde{\sigma}_a$. Let us write $\mathcal{P}_\delta$ for the law of $\mathcal{S}$ (both the spin configuration and spin field, and as for the normal Ising model, we don't care what space of functions $\mathcal{S}$ lives on). Perhaps slightly artificially, but as discussed in Section \ref{subsec:isingresults}, motivated by wanting to study scaling limits of near critical models of statistical mechanics, we also add a coupling to a (non-uniform) magnetic field to this law: for a function $\psi\in C_c^\infty(U)$, define
\begin{align*}
\mathcal{P}_{\psi,\delta}(\mathcal{S})&=\frac{1}{\mathcal{Z}_{\psi,\delta}}e^{\delta^{2-\frac{1}{4}}\sum_{a\in F_\delta \cup \partial F_\delta}(\delta^{-2}\int_a \psi(x)dx)\mathcal{S}_a}\mathcal{P}_\delta(\mathcal{S})\\
&=\frac{1}{\mathcal{Z}_{\psi,\delta}}e^{\delta^{-1/4}\int_{U} \psi(x)\mathcal{S}_\delta(x)dx}\mathcal{P}_\delta(\mathcal{S}),
\end{align*}
where $\mathcal Z_{\psi,\delta}$ is a normalizing constant. The reason to view this as a coupling to a magnetic field is that typically in spin models of statistical mechanics, the part of the energy of a spin configuration $(\sigma_a)$ coming from an interaction with a magnetic field $(h_a)$ is given by $-\sum_a h_a\sigma_a$, and in the Gibbs measure of the model in a non-zero magnetic field is obtained by biasing the zero-magnetic field Gibbs measure with a quantity $ \frac{1}{Z_{\beta,h}}e^{\beta \sum_a h_a\sigma_a}$, where $Z_{\beta,h}$ is a normalizing constant. In this picture, our model corresponds roughly to choosing $h_a=\delta^{2-\frac{1}{4}}\psi(a)$ (where $\psi(a)$ means the value at the center of the face, which is close to $\delta^{-2}\int_a \psi(x)dx$ due to the smoothness of $\psi$). Since $h_a\to 0$ as $\delta\to 0$, one sometimes calls this type of model near-critical in that it is close to the critical case of $h=0$.

\subsection{Convergence to multiplicative chaos} \label{sec:Isingconv}
 The goal of this subsection is to prove Theorem~\ref{th:ising}. The main point in
the proof is to obtain a $\delta$-independent integrable upper bound for the
$n$-point correlation function $\E \sigma_\delta(x_1) \dots
\sigma_\delta(x_n)$, which makes it possible to use the dominated convergence
theorem and Theorem~\ref{th:CHI} to find asymptotics of moments of $\delta^{-1/4}\int_U \mathcal S_\delta(x)f(x)dx$ and then using the method of moments, justified by Theorem \ref{th:moments}, conclude the convergence. Such an upper bound is obtained
by proving a variant of Onsager's inequality for the Ising model, after which
integrability is obtained again from Lemma~\ref{lemma:combinatorial_argument}.

The precise statement about the moments of $\mathcal S_\delta$ is the following.

\begin{lemma}\label{lemma:ising_moments}
For each $f\in C_c^\infty(U)$ and integer $k\geq 0$

\begin{align}\label{eq:ising_moments}
\lim_{\delta\to 0}&\E \left(\delta^{-1/4}\int_U f(x) \mathcal{S}_\delta(x)dx \right)^k\\
&=\left(\frac{\mathcal{C}^2}{\sqrt{2}}\right)^k\int_{U^k}\prod_{j=1}^k \left[f(x_j) \left(\frac{|\varphi'(x_j)|}{2 \mathrm{Im}\, \varphi(x_j)}\right)^{1/4}\right]\sum_{\mu\in\lbrace -1,1\rbrace^n}\prod_{i<j}\left|\frac{\varphi(x_i)-\varphi(x_j)}{\varphi(x_i)-\overline{\varphi(x_j)}}\right|^{\frac{\mu_i\mu_j}{2}}\prod_{j=1}^k dx_j. \nonumber
\end{align}

\noindent and for each $\lambda>0$

  \begin{equation}\label{eq:ising_laplace}
\sup_{\delta>0}\E e^{\lambda\left|\delta^{-1/4}\int_U \mathcal{S}_\delta(x)dx\right|}<\infty.
  \end{equation}
\end{lemma}
Our proof will be based on the following lemma.

\begin{lemma}\label{lemma:spinbound}
  Let $a_1,\dots,a_k \in F_\delta$ be distinct faces lying inside a fixed compact set $K \subset U$ and identify each face with its center. Then for some constant $C > 0$ we have
  \[\delta^{-k/8} \E \sigma_{a_1} \dots \sigma_{a_k} \le C^k \prod_{i=1}^k \big( \min_{j \neq i} |a_i - a_j|\big)^{-1/8}.\]
  The constant $C$ is independent of the points $a_i$, $k$ and $\delta$, but it may depend on $K$.
\end{lemma}

\begin{proof}
  This inequality essentially appears in the proof of Proposition~3.10 in \cite{FM}, where the authors show (\cite[last line on p. 20]{FM}) that
  \[\E \sigma_{a_1} \dots \sigma_{a_k} \le \prod_{i=1}^p \phi_{B_i}^+ (a_i \leftrightarrow \partial B_i).\]
  Here $B_i = a_i + [-\ell_i/4, \ell_i/4]^2$ are disjoint boxes with $\ell_i = \min_{j\ge 0, j \neq i} d(a_i,a_j)$ being the $\delta \integers^2$-distance (we have added the factor $\delta$ compared to \cite{FM} because we are working on the scaled lattice) from $a_i$ to its closest neighbour or to the boundary $\partial U$ which is denoted by $a_0$.
  The quantity $\phi_{B_i}^+(a_i \leftrightarrow \partial B_i)$ denotes the 
  probability that $a_i$ is connected to the boundary of $B_i$ in the FK--Ising model (see e.g. \cite[Section 3.1 and Section 3.2]{FM} and references therein), and this probability is
  less than $C \ell_i^{-1/8}$ by \cite[Lemma~3.9]{FM}. Our claim then follows from the elementary inequality $d(a,b) \le \sqrt{2}|a - b|/\delta$
  and the fact that by compactness $d(K,\partial U)$ is bounded from below for small enough $\delta$.
\end{proof}

This allows us to give the proof of Lemma \ref{lemma:ising_moments}.

\begin{proof}[Proof of Lemma~\ref{lemma:ising_moments}]
Let $K \subset U$ be a fixed compact set and let $x_1,\dots,x_k\in K$. We claim that for some $C>0$ independent of $x_i$, $k$, and $\delta$,

\begin{equation}\label{eq:fieldbound}
\delta^{-k/8}\E \sigma_\delta(x_1)\cdots \sigma_\delta(x_k)\leq C^k\prod_{i=1}^k\left(\min_{j\neq i}|x_i-x_j|\right)^{-1/8}
\end{equation}
which can be seen as a variant of Onsager's inequality for the Ising model.

  Let us write $a_i^{(x)}$ for the face $x$ lies in (with the convention that we count in it the southern and western boundary without corners as well as the south-western corner). If we first assume that all of the $a_i^{(x)}$ are distinct ($|a_i^{(x)}-a_j^{(x)}|\geq \delta$) and note that $|x_i-x_j|\leq |x_i-a_i^{(x)}|+|a_i^{(x)}-a_j^{(x)}|+|x_j-a_j^{(x)}|\leq 2\delta +|a_i^{(x)}-a_j^{(x)}|\leq 3|a_i^{(x)}-a_j^{(x)}|$, then \eqref{eq:fieldbound} follows immediately from Lemma~\ref{lemma:spinbound}.

Consider then the case where not all of the $a_i^{(x)}$ are distinct. After using $\sigma_a^2=1$ to reduce the number of spins from the correlation function and possibly relabelling the spins, let us assume that we have 

$$
\delta^{-k/8}\E \sigma_\delta(x_1)\cdots \sigma_\delta(x_k)=\delta^{-k/8}\E \sigma_{\delta}(x_1)\cdots \sigma_\delta(x_l)
$$

\noindent with $l<k$ and $(a_i^{(x)})_{i=1}^l$ distinct. From the case where all faces were distinct, we find 

\begin{align*}
\delta^{-k/8}\E \sigma_\delta(x_1)\cdots \sigma_\delta(x_k)&\leq \delta^{-(k-l)/8}C^l\prod_{j=1}^l \left(\min_{1\leq i\leq l,i\neq j}|x_i-x_j|\right)^{-1/8}\\
&\leq \delta^{-(k-l)/8}C^l\prod_{j=1}^l \left(\min_{1\leq i\leq k,i\neq j}|x_i-x_j|\right)^{-1/8}
\end{align*}

\noindent where the second step comes from the fact that we minimize over a larger set. Now for the remaining points $x_{l+1},...,x_k$, for each of them, there is another $x_i$ such that both points belong to the same face, implying that for $j>l$, 

$$
\min_{1\leq i\leq k,i\neq j}|x_i-x_j|\leq \sqrt{2}\delta
$$

\noindent so that 

$$
\delta^{-(k-l)/8}\leq 2^{\frac{k-l}{16}}\prod_{j=l+1}^k\left(\min_{1\leq i\leq k,i\neq j}|x_i-x_j|\right)^{-1/8},
$$

\noindent which concludes the proof of \eqref{eq:fieldbound}.

We may now compute
\begin{align*}
  \lim_{\delta \to 0} \E \Big( \delta^{-1/4} \int_U f(x) \mathcal{S}_\delta(x) \, dx \Big)^k & = \delta^{-k/4} \int_{U^k} f(x_1) \dots f(x_k) \E \mathcal{S}_\delta(x_1) \dots \mathcal{S}_\delta(x_k) \, dx \\
  & = \delta^{-k/4} \int_{U^k} f(x_1) \dots f(x_k) (\E \sigma_\delta(x_1) \dots \sigma_\delta(x_k))^2 \, dx.
\end{align*}
Using \eqref{eq:fieldbound}, we see that the absolute value of the integrand is at most
\[C^{2k} \|f\|_\infty^k \prod_{i=1}^k \Big(\min_{j \neq i} |x_i - x_j|\Big)^{-1/4}.\]
  By Lemma~\ref{lemma:combinatorial_argument} this is integrable, so we may apply the dominated convergence theorem and Theorem~\ref{th:CHI} to get
  \begin{align*}
    & \lim_{\delta \to 0} \E \Big( \delta^{-1/4} \int_U f(x) \mathcal{S}_\delta(x) \, dx \Big)^k \\
    & = \int_{U^k} f(x_1) \dots f(x_k) \mathcal{C}^{2k} \prod_{j=1}^k \Big( \frac{|\varphi'(x_j)|}{2 \Im \varphi(x_j)} \Big)^{1/4} 2^{-k/2} \sum_{\mu \in \{-1,1\}^n} \prod_{1 \le i < j \le k} \left| \frac{\varphi(x_i) - \varphi(x_j)}{\varphi(x_i) - \overline{\varphi(x_j)}}\right|^{\frac{\mu_i \mu_j}{2}} \, dx,
  \end{align*}
  which proves \eqref{eq:ising_moments}.
  Moreover, the uniform bound obtained from Lemma~\ref{lemma:combinatorial_argument} also implies \eqref{eq:ising_laplace}.
\end{proof}

Having Lemma \ref{lemma:ising_moments} in our hand, we can now turn to the proof of convergence to chaos.

\begin{proof}[Proof of Theorem~\ref{th:ising}]
  By (the proof of) Theorem \ref{th:moments}, the moments of
  \[\int_U \mathcal{C}^2\left(\frac{2|\varphi'(x)|}{\Im \varphi(x)}\right)^{1/4}\cos(2^{-1/2}X(x))f(x)\,dx\]
  are precisely the right side of \eqref{eq:ising_moments}.
Thus  by Lemma~\ref{lemma:ising_moments} the moments of the XOR-Ising field converge to those of the real part of the imaginary chaos and 
 by Theorem \ref{th:moments}, the moments of the imaginary chaos grow slowly enough so that they determine its distribution
  and the convergence of moments implies convergence in law -- see Corollary~\ref{co:detemination}.
\end{proof}

\subsection{The sine-Gordon model}\label{sec:sg}

Let us now introduce the sine-Gordon type model appearing in the statement of Theorem~\ref{th:perturb}. In the theoretical physics literature, a definition of the sine-Gordon model could be representing the correlation functions of the sine-Gordon field as a functional integral, which might be written as

$$
\langle X(x_1)\cdots X(x_k)\rangle_{\mathrm{sG}(\lambda,\beta)}=\frac{1}{Z(\lambda,\beta)}\int X(x_1)\cdots X(x_k)e^{\lambda \int_{\R^2} \cos \beta X(x)dx-\int_{\R^d}\nabla X(x)\cdot \nabla X(x)dx}\mathcal{D}X.
$$

\noindent Above $\mathcal{D}X=\prod_{x\in \R^2}dX(x)$ is formally the (non-existent) infinite dimensional Lebesgue measure and the integral is over $\R^{\R^2}$. This is of course ill-defined, but the way one mathematically makes sense of this is through understanding the combination $e^{-\int_{\R^2}\nabla X(x)\cdot \nabla X(x)dx}\mathcal{D}X$ as the probability distribution of the (whole plane) Gaussian free field.  Then one could try to view this as biasing the law of the Gaussian free field with something again related to imaginary multiplicative chaos. For our purposes, it is more convenient to work in a finite domain with zero boundary conditions on the free field (this also avoids the problem with the zero mode or the fact that the whole plane free field is well defined only up to a random additive constant). Also instead of having just the quantity $\lambda\int \cos \beta X(x)dx$, our purposes require generalizing slightly and replacing the constant $\lambda$ by a weight in the integral. We thus make the following definition.

\begin{definition}\label{def:sG}
  Let $U\subset \R^2$ be a bounded simply connected domain, let $X$ be the zero boundary Gaussian free field in $U$ -- see Example \ref{ex:fields} -- with law $\P_{\mathrm{GFF}}$ on (say) $H^{-\varepsilon}(\reals^2)$\footnote{Recall that the field $X$ is actually supported in $U$ -- see Proposition~\ref{prop:karhunen_loeve}. As is often done, one could also consider $X$ as a random element of $H^{-\varepsilon}(U)$.}. For $\psi\in C_c^\infty(U)$, $\beta\in(0,\sqrt{2})$, the sine-Gordon$(\psi,\beta)$ model in domain $U$ with zero boundary condition is a probability distribution on $H^{-\varepsilon}(\reals^2)$ of the form 

$$
\P_{\mathrm{sG}(\psi,\beta)}(dX)=\frac{1}{Z(\psi,\beta)}e^{\int_U \psi(x)\cos\beta X(x)dx}\P_{\mathrm{GFF}}(dX),
$$
where again the integral in the exponential is formal notation for testing the random generalized function $\cos(\beta X)$ against the test function $\psi$.  \hfill $\blacksquare$
\end{definition}

\begin{remark}
For the above definition to make sense, $\cos\beta X$ has to be measurable w.r.t. $X$ and we need $\E e^{\int_U \psi(x)\cos\beta X(x) \, dx}$ to be finite. The first property follows simply from our convergence in probability in Theorem \ref{th:existuniq}, while the second one follows from Theorem \ref{th:moments}. \hfill $\blacksquare$
\end{remark}

This definition allows us to construct the cosine of the sine-Gordon field, namely the proposed limiting object from Theorem \ref{th:perturb}. We do not really need to construct it as a random generalized function, we simply need to know that for each test function, there exists a random variable that can be viewed as the cosine of the sine-Gordon field tested against this test function. 

\begin{definition}\label{def:sgchaos}
Let $U\subset \R^2$ be a bounded simply connected domain. For each $\beta,\gamma\in(0,\sqrt{2})$ and $f,\psi\in C_c^\infty(U)$, let us write 
$$
\int_U f(x)\cos \left(\gamma X_{\mathrm{sG}(\psi,\beta)}(x)\right)dx
$$
for the random variable whose law is characterized by the condition that for each bounded continuous $F:\R\to \R$,
\begin{align*}
\E&\left[F\left(\int_U f(x)\cos \left(\gamma X_{\mathrm{sG}(\psi,\beta)}(x)\right)dx\right)\right]\\
&\qquad =\frac{1}{Z(\psi,\beta)}\E_{\mathrm{GFF}}\left[F\left(\int_U f(x)\cos\left(\gamma X(x)\right)dx\right) e^{\int_U \psi(x)\cos(\beta X(x))dx}\right],
\end{align*}
where $\int_U f(x)\cos(\gamma X(x))dx$ and $\int_U \psi(x)\cos(\beta X(x))dx$ denote the action of the real parts of imaginary chaos distributions built from the GFF on $U$ with zero boundary conditions provided by Theorem \ref{th:existuniq}. \hfill $\blacksquare$
\end{definition}
To see that this is a valid definition, first note from Theorem \ref{th:existuniq} that we can simultaneously construct both of the random variables $\int_U f(x)\cos(\gamma X(x))dx$ and $\int_U \psi(x)\cos(\beta X(x))dx$ on the same probability space. Moreover, as $F$ is bounded, we have from Theorem \ref{th:moments} that the expectation on the right hand side of the equation in the definition is finite. Thus, by the standard argument of interpreting this as a positive linear functional of $F$, the Riesz--Markov--Kakutani representation theorem provides the existence of the desired probability distribution. We note that one could also construct the same object starting from regularizations of the free field.

We are now in a position to move on to the proof of Theorem \ref{th:perturb}.

\subsection{Convergence of the magnetically perturbed critical XOR-Ising to cosine of the sine-Gordon field} \label{sec:pertconv} Proving that the spin field of the magnetically perturbed XOR-Ising model converges to the cosine of the sine-Gordon field, or Theorem \ref{th:perturb}, now follows rather easily from Theorem \ref{th:existuniq}.

\begin{proof}[Proof of Theorem \ref{th:perturb}]
What we wish to show is that for each bounded continuous $F:\R\to \R$, 
\begin{align*}
  \lim_{\delta\to 0}&\E_{\psi,\delta}\left[F\left(\delta^{-1/4}\int_U f(x)\mathcal{S}_\delta(x)dx\right)\right]\\
&=\E \left[F\left(\mathcal C^2\int_U \left(\frac{2|\varphi'(x)|}{\mathrm{Im}\, \varphi(x)}\right)^{1/4} f(x)\cos \left(2^{-1/2}X_{\mathrm{sG}(\widetilde \psi,1/\sqrt{2})}(x)\right)dx\right)\right],
\end{align*}
where on the left hand side we have the spin field of the magnetically perturbed XOR-Ising model, with law $\mathcal P_{\psi,\delta}$ and expectation $\E_{\psi,\delta}$, and on the right hand side we have the random variable defined in Definition \ref{def:sgchaos}. 

Recall that we wrote $\mathcal P_\delta$ for the law of the spin field of the zero-magnetic field XOR-Ising model and let us write $\E_\delta$ for the corresponding expectation. By the definition of $\mathcal{P}_{\psi,\delta}$ we thus have
  \[\E_{\psi,\delta} F\Big(\delta^{-1/4} \int_U \mathcal{S}_\delta(x) f(x) \, dx \Big) = \frac{1}{\mathcal Z_{\psi,\delta}} \E_\delta \left[F\Big(\delta^{-1/4} \int_U \mathcal{S}_\delta(x) f(x) \, dx \Big) e^{\delta^{-1/4} \int_U \psi(x) \mathcal{S}_\delta(x) \, dx}\right] \]
  By Theorem~\ref{th:ising} we know that under $\mathcal{P}_\delta$, $\delta^{-1/4} \mathcal{S}_\delta$ tested against an arbitrary test function converges in law to $\cos(\frac{1}{\sqrt{2}} X)$ (where $X$ is the free field) tested against $\mathcal C^2(2\frac{|\varphi'(x)|}{\mathrm{Im}\, \varphi(x)})^{1/4}$ times that same test function, so by linearity and the Cram\'er-Wold theorem, the random variables $A = \delta^{-1/4} \int_U \mathcal{S}_\delta(x) f(x) \, dx$ and $B = \delta^{-1/4} \int_U \mathcal{S}_\delta(x) \psi(x) \, dx$ converge jointly in law (to the corresponding random variables expressed in terms of the free field. From the continuity of the map $(x,y) \mapsto F(x) e^{y}$ it follows that also $F(A) e^{B}$ converges in law by the continuous mapping theorem \cite[Lemma~4.27]{K}  to the random variable 
$$
F\left(\mathcal C^2\int_U\left(\frac{2|\varphi'(x)|}{\mathrm{Im}\, \varphi(x)} \right)^{1/4}f(x)\cos(2^{-1/2}X(x))dx\right)e^{\int_U \widetilde{\psi}(x)\cos(2^{-1/2}X(x))dx}.
$$  
Moreover, by the (exponential) uniform integrability provided by boundedness of exponential moments proven in Lemma \ref{lemma:ising_moments}, $\mathcal Z_{\psi,\delta}$ converges to $Z(\psi,1/\sqrt{2})$ as $\delta \to 0$. These remarks combined  with another application of the boundedness of exponential moments from Lemma \ref{lemma:ising_moments} shows that also the expectation of $F(A)e^B$ converges to the correct quantity as $\delta\to 0$ and we deduce that
\begin{align*}
\lim_{\delta\to 0}&\E_{\psi,\delta}\left[F\left(\delta^{-1/4}\int_U f(x)\mathcal{S}_\delta(x)dx\right)\right]\\
&=\E \left[F\left({\mathcal C^2}\int_U \left(\frac{2|\varphi'(x)|}{\mathrm{Im}\, \varphi(x)}\right)^{1/4} f(x)\cos \left(2^{-1/2}X_{\mathrm{sG}(\widetilde \psi,1/\sqrt{2})}(x)\right)dx\right)\right],
\end{align*}
as was desired.
\end{proof}
This concludes our study of the Ising model.

\section{Random unitary matrices and imaginary multiplicative chaos -- a cautionary tale}\label{sec:rmt}

We begin this section with a review of the connection between random unitary matrices, log-correlated fields, and real multiplicative chaos. Based on this connection, it is natural to expect that imaginary multiplicative chaos also appears naturally in random matrix theory, and we indeed formulate a result of this flavor in the setting of random unitary matrices. As the proof is so similar to the case of real multiplicative chaos, and its essential ingredients are well documented in the literature, we omit the details and simply offer the reader the relevant references for reproducing the proof. With this example of random matrices, we illustrate both that imaginary chaos appears naturally in various models of probability and mathematical physics along with some of the subtleties one can expect to encounter when constructing multiplicative chaos from such models.

In the last two decades, the connection between random matrix theory and log-correlated fields has been observed in various random matrix models -- see e.g. \cite{HKO,RidVir,FKS}. More precisely, as the size of the matrix tends to infinity, the real part of the logarithm of the characteristic polynomial of a random matrix drawn from various distributions of unitary, Hermitian, or normal matrices, is known to converge to a variant of the Gaussian free field after suitable recentering. As first utilized in \cite{FHK,FK} on a heuristic level, one would then naturally expect that powers of the (absolute value of the) characteristic polynomial of such a random matrix should be related to the exponential of the Gaussian free field -- multiplicative chaos. This type of results have since been proven for some models of random matrix theory -- see e.g. \cite{BWW,LOS,W} -- though focusing on cases where the limiting object is a real multiplicative chaos measure, such as the case of real powers of the absolute value of the characteristic polynomial.

In this article, we will consider large random unitary matrices drawn from the Haar measure on the unitary group $\mathrm U(N)$, or in other words, we consider the so-called Circular Unitary Ensemble. Let us write $U_N$ for such a random $N\times N$ unitary matrix\footnote{Being one of the classical compact groups, it is a classical fact that there exists a unique probability measure $\P_N$ on $\mathrm{U}(N)$ such that for any Borel set ${B}\subset \mathrm{U}(N)$ and any fixed $U\in \mathrm{U}(N)$, $\P_N(U{B})=\P_N({B}U)=\P_N({B})$ -- this probability measure is the one we take for the distribution of the random matrix $U_N$. We write simply $\E$ for integration with respect to $\P_N$.} and consider the following two fields defined on the unit circle
$$
X_N(\theta)=\log |\det(I-e^{-i\theta}U_N)| \quad \text{and} \quad Y_N(\theta)=\lim_{r\to 1^-}\mathrm{Im}\mathrm{Tr}\log(I-re^{-i\theta}U_N), \quad \theta\in[0,2\pi],
$$
where $I$ denotes the $N\times N$ identity matrix, and in the definition of $Y_N$, what we mean by $\mathrm{Tr}\log (I-re^{-i\theta}U_N)$ is $\sum_{j=1}^N \log (1-re^{i(\theta_j-\theta)})$, where $(e^{i\theta_j})_{j=1}^N$ are the eigenvalues of $U_N$, and the branch of the logarithm is the principal one -- namely it is given by $\log(1-z)=-\sum_{k=1}^\infty \frac{1}{k}z^k$ for $|z|<1$. Note that in this case, the limit defining $Y_N$ exists almost surely e.g. in $L^2([0,2\pi],d\theta)$. Thus the fields can be interpreted as the real and imaginary parts of the logarithm of the characteristic polynomial of $U_N$ evaluated on the unit circle.

It was proven in \cite{HKO} that as $N\to\infty$, $X_N$ and $Y_N$ converge in law to $2^{-1/2}$ times the 2d Gaussian free field restricted to the unit circle, namely a centered log-correlated Gaussian field $X$ with covariance $\E X(\theta)X(\theta')=-\log |e^{i\theta}-e^{i\theta'}|$ -- for details about this field, see Example \ref{ex:fields}. Moreover, this convergence was in the Sobolev space $H^{-\varepsilon}$ for arbitrary $\varepsilon>0$ -- this is essentially as nicely as a sequence of random generalized functions could converge.  It was then proven in \cite{W} that for $-\frac{1}{2}<\alpha<\sqrt{2}$ and $-\sqrt{2}<\beta<\sqrt{2}$, $\frac{e^{\alpha X_N(\theta)}}{\E e^{\alpha X_N(\theta)}}d\theta$ and $\frac{e^{\beta Y_N(\theta)}}{\E e^{\beta Y_N(\theta)}}d\theta$ converge in law to the multiplicative chaos measures formally written as $e^{\frac{\alpha}{\sqrt{2}} X(\theta)}d\theta$ and $e^{\frac{\beta}{\sqrt{2}} X(\theta)}d\theta$. In this article, we consider an analogue of this result for imaginary $\alpha$ and $\beta$. More precisely, the result is the following:

\begin{proposition}\label{th:rmtgmc}
Let $X(\theta)$ be the log-correlated Gaussian field on $[0,2\pi]$ with covariance $\E X(\theta)X(\theta')=-\log |e^{i\theta}-e^{i\theta'}|$ $($see Example \ref{ex:fields} for details$)$, and $e^{i\beta X(\theta)}$ the associated imaginary multiplicative chaos distribution provided by Theorem \ref{th:existuniq}. Then for any smooth and $2\pi$-periodic $f:\R\to \C$
$$
\int_0^{2\pi}\frac{e^{i\beta X_N(\theta)}}{\E e^{i\beta X_N(\theta)}}f(\theta)d\theta\stackrel{d}{\to}\int_0^{2\pi}e^{i\frac{\beta}{\sqrt{2}}X(\theta)}f(\theta)d\theta
$$
\noindent as $N\to\infty$, for $\beta\in(-\sqrt{2},\sqrt{2})$. Moreover, as $N\to\infty$,
$$
\int_0^{2\pi}\frac{e^{i\beta Y_N(\theta)}}{\E e^{i\beta Y_N(\theta)}}f(\theta)d\theta\stackrel{d}{\to}\int_0^{2\pi}e^{i\frac{\beta}{\sqrt{2}}X(\theta)}f(\theta)d\theta
$$
\noindent for $\beta\in(-1,1)$. In both statements, the integrals on the right hand side are formal notation meaning that the distribution $e^{i\frac{\beta}{\sqrt{2}} X(\theta)}$ is tested against $f$.
\end{proposition}

As mentioned above, the proof of this theorem is essentially repeating the proof from the real case in the $L^2$-phase -- see e.g. \cite{BWW,LOS,W} for this type of arguments. The main issue is to control moments of the form e.g. $\E e^{i\beta X_N(\theta)-i\beta X_N(\theta')}$. This is done through a connection to certain Toeplitz determinants and their known asymptotics. Again this is essentially identical to the real case -- see e.g. \cite{W} for the argument and \cite[Theorem 1.11]{CK}\footnote{We mention here that while in \cite{W} the proof requires nearly the full extent  of the results of \cite{CK}, one can actually simplify the proof slightly through an easy Cauchy-Schwarz estimate and one can manage with just making use of \cite[Theorem 1.11]{CK} and \cite[Theorem 1.1 and Remark 1.4]{DIK2} also in the real case.} and \cite[Theorem 1.1 and Remark 1.4]{DIK2} for the asymptotics of the relevant Toeplitz determinant.

Being nearly identical to the real case, we think that the reader would not find the proof very illuminating. Indeed, instead of the proof, what we hope readers will find interesting here is the discrepancy between the parameter values for which convergence is obtained for the two fields. We maintain that this is not a technical issue simply requiring better estimates, but truly that for $Y_N$ one does not have convergence for larger values of $|\beta|$ despite the fact that $Y_N$ converges to a log-correlated field essentially as nicely as one might hope and that the corresponding multiplicative chaos exists. We think these remarks should be viewed as a warning that one ought to take some care when hoping to prove that something converges to multiplicative chaos, and hence we will next elaborate on it slightly.

\subsection{Pitfalls in proving convergence to multiplicative chaos, from the point of view of random matrices}
Let us first discuss the case of real multiplicative chaos briefly. As mentioned above, it was proven in \cite{W} that for $-\frac{1}{2}<\alpha<\sqrt{2}$, $\frac{e^{\alpha X_N(\theta)}}{\E e^{\alpha X_N(\theta)}}d\theta$ converges in law to $e^{\frac{\alpha}{\sqrt{2}}X(\theta)}d\theta$. A natural question that one might then have is what happens for $-\sqrt{2}<\alpha\leq -\frac{1}{{2}}$. After all, the multiplicative chaos measure is perfectly well defined here and for the field $Y_N$, one has convergence. This issue is at least partly resolved by recalling the definition of $X_N$: for $\alpha<0$, one has $e^{\alpha X_N(\theta)}=|\det(I-e^{-i\theta}U_N)|^{-|\alpha|}$. If $\alpha<-1$, as a function of $\theta$, this will have non-integrable singularities at each eigenangle, and even after a deterministic normalization, $e^{\alpha X_N(\theta)}d\theta$ can not converge to $e^{\frac{\alpha}{\sqrt{2}}X(\theta)}d\theta$ or any other finite measure -- note that for such $\alpha$, even the normalization constant $\E e^{\alpha X_N(\theta)}$ is infinite. Thus we have an example of a sequence of fields which approximate a log-correlated field essentially as nicely as one could hope for -- that is we have convergence in any Sobolev space of negative regularity index -- yet do not give rise to a real multiplicative chaos measure in the full regime where one would naively expect. This example demonstrates that the field may take extremely large values, but in a very small set, while  maintaining  convergence on the level of log-correlated fields but not on the level of real multiplicative chaos measures -- and this happens already in a portion of the $L^2$-phase.

Let us turn to imaginary multiplicative chaos, and try to qualitatively understand the discrepancy in Proposition \ref{th:rmtgmc}, which we expect to arise due to  a mechanism of a different nature compared to the real case. Noting now that for the field $X_N$, we have convergence in the whole $L^2$-regime, this suggests that any possible lack of convergence for $Y_N$ is not due to the size of the field. To explain the lack of convergence, let us point out that a simple exercise in trigonometry shows that for $\theta\neq \theta_j$ for all $j$
\[
Y_N(\theta)=\sum_{j=1}^N \frac{\theta_j-\theta}{2}-\frac{N\pi}{2}+\pi\sum_{j=1}^N \mathbf{1}\lbrace \theta_j<\theta\rbrace,
\]
so apart from the first sum (which is a rather simple function of $\theta$) and a factor of $\pi$, the field is essentially integer valued. If we only cared about this integer valued ``eigenangle counting function" $\widetilde Y_N(\theta):=\pi\sum_{j=1}^N \mathbf{1}\lbrace \theta_j<\theta\rbrace$, then $e^{i \beta \widetilde Y_N(\theta)}$ would be $2$-periodic in $\beta\in (-\sqrt{2},\sqrt{2})$. This periodic structure is of course somehow present also in $e^{i\beta Y_N(\theta)}$, though not in its exact form. As the limiting multiplicative chaos distribution $e^{i\frac{\beta}{\sqrt{2}}X(\theta)}$ does not have any periodicity properties when viewed as a function of $\beta$, this suggests that indeed it is not reasonable to expect convergence for $\beta\notin (-1,1)$, since outside of this regime, periodicity will kick in. Also, at $\beta=\pm 1$ something special obviously happens as the exponential $\exp(i\beta \widetilde Y_N(\theta))$ takes only values $\pm 1$. This is also seen in the Toeplitz determinants corresponding to the moments -- see e.g. \cite[Theorem 1.1 and Theorem 1.13]{DIK1} (our $\beta=\pm 1$ corresponds precisely to their $|||\beta|||=1$ e.g. for the second moment $\E e^{i\beta Y_N(\theta)-i\beta Y_N(\theta')}$, though note the slight difference in notation: our $\beta$ corresponds to $\beta/2$ in \cite{DIK1}.).

Such ``essentially integer valued" approximations to log-correlated fields are common in various models of probability and mathematical physics -- indeed any height functions or interfaces of discrete models, like  dimer models or random partitions of integers, are  inherently of this type, so when attempting to study these fields through the associated (real or imaginary) multiplicative chaos,  the phenomenon described above  is good to keep in mind.

\appendix

\section{Auxiliary results}\label{app:moments}

In this appendix we record some basic facts needed to control moments of imaginary chaos near the critical point. The first one is something that gives a rough estimate required for controlling mixed moments.

\begin{lemma}\label{lemma:moment_inequality} Let $U\subset \reals^d$ be bounded and $0<\beta<\sqrt{d}$. Then for any indices $a \ge b$ and
$x_1,\dots,x_a,y_1,\dots,y_b \in U$, $a \ge b$, we have the inequality
  \begin{align*}
    & \frac{\prod_{1 \le j < k \le a} |x_j - x_k|^{\beta^2} \prod_{1 \le j < k \le b} |y_j - y_k|^{\beta^2}}{\prod_{1 \le j \le a} \prod_{1 \le k \le b} |x_j - y_k|^{\beta^2}}
    \; \le \; \sum_{\substack{f \colon \{1,\dots,b\} \to \{1,\dots,a\},\\  \textrm{ injective}}}
     \frac{C}{\prod_{1 \le j \le b} |x_{f(j)} - y_j|^{\beta^2}} \label{eq:partial_fraction} \nonumber 
  \end{align*}
  for some constant $C$ depending only $U$, $a$, and $b$ -- not $\beta$.
\end{lemma}

\begin{proof}
The result can be obtained by using a Gale--Shapley matching (see e.g. the appendix in \cite{LRV} -- we provide a proof here for the reader's convenience).
  For given $x_1,\dots,x_a$ and $y_1,\dots,y_b$ we may form a matching $f \colon \{1,\dots,b\} \to \{1,\dots,a\}$ via the following algorithm:
  Among the remaining pairs $(x_j,y_k)$ choose one with minimal distance $|x_j - y_k|$, set $f(k) = j$, remove the points $x_j$ and $y_k$ from the set of remaining points and repeat.
By permutation invariance of the original expression we may assume that the points matched by the algorithm are   $(y_1,x_1), \dots, (y_b,x_b)$, and they are matched in this order. We may then write
  \begin{align*}
    & \frac{\prod\limits_{1 \le j < k \le a} |x_j - x_k|^{\beta^2} \prod\limits_{1 \le j < k \le b} |y_j - y_k|^{\beta^2}}{\prod\limits_{\substack{1 \le j \le a \\ 1 \le k \le b}} |x_j - y_k|^{\beta^2}} \\
    & = \frac{\prod\limits_{{b+1 \leq j < k \le a  }} |x_j - x_k|^{\beta^2}}{\prod\limits_{1 \le j \le b} |x_{j} - y_j|^{\beta^2}} \cdot \frac{\prod\limits_{\substack{1 \le j < k \le a \\  j \leq b}} |x_j - x_k|^{\beta^2} \prod\limits_{1 \le j < k \le b} |y_j - y_k|^{\beta^2}}{\prod\limits_{\substack{1 \le j \le a \\ 1 \le k \le b \\ j \neq k}} |x_j - y_k|^{\beta^2}}.
  \end{align*}
We next write the second factor as
  \begin{align*}
    & = \prod_{\ell = 1}^{b}\left( \prod_{\ell < k \le b} \frac{|y_\ell - y_k|^{\beta^2}}{|x_k-y_{\ell} |^{\beta^2}} \prod_{\ell < j \le a} \frac{|x_\ell - x_{j}|^{\beta^2}}{| x_j-y_\ell |^{\beta^2}}\right)
  \end{align*}
  and using the inequalities
  \[\frac{|y_\ell - y_k|}{|x_k- y_{\ell} |} \le \frac{|y_\ell - x_{k}| + |x_{k } - y_k|}{|x_k -y_\ell |
  } \le 2,\]
where we use that $y_k$ was matched  before $y_\ell$, and
  \[\frac{|x_\ell - x_{j}|}{|x_j-y_\ell|} \le \frac{|x_\ell - y_\ell| + |y_\ell - x_j|}{|x_j-y_\ell|} \le 2\]
  implied in turn by the fact that $x_\ell$ was matched before $x_j$, we see that
  \[\frac{\prod\limits_{1 \le j < k \le a} |x_j - x_k|^{\beta^2} \prod\limits_{1 \le j < k \le b} |y_j - y_k|^{\beta^2}}{\prod\limits_{\substack{1 \le j \le a \\ 1 \le k \le b}} |x_j - y_k|^{\beta^2}} \le 2^{\beta^2(a-1)b} \frac{\prod\limits_{b+1 \le j < k \le a } |x_j - x_k|^{\beta^2}}{\prod\limits_{1 \le j \le b} |x_{j} - y_j|^{\beta^2}}\]
  under the assumption that the points were matched according to $f$. Summing over the possible matchings and bounding $\beta^{2}$ by $d$ in the prefactor yields the result. \end{proof}

The following lemma is used for studying the behavior of imaginary multiplicative chaos near the critical point.

\begin{lemma}\label{lemma:moments}
Let $\mu$ be the random generalized function from Theorem \ref{th:existuniq}. For any test function $\varphi \in C^\infty_c(U)$ we have
  \[\left|\E \mu(\varphi)^a \overline{\mu(\varphi)}^b \right|\le C (\E |\mu(\varphi)|^2)^{\min(a,b)}\]
  for all integers $a, b \ge 0$ and some constant $C$ possibly depending on $\varphi$, $g$ from \eqref{eq:cov}, $a$, and $b$, but not on $\beta$.
\end{lemma}

\begin{proof}
By the proof of Theorem \ref{th:moments} and a direct computation
  \[\left|\E \mu(\varphi)^a \overline{\mu(\varphi)}^b \right|\lesssim C_{a,b} \int_{U^{a \times b}} \frac{\prod_{1 \le j < k \le a} |x_j - x_k|^{\beta^2} \prod_{1 \le j < k \le b} |y_j - y_k|^{\beta^2}}{\prod_{1 \le j \le a} \prod_{1 \le k \le b} |x_j - y_k|^{\beta^2}} \, dx_1 \dots dx_a dy_1 \dots dy_b.\]
 Here $C_{a,b}$ depends on $\varphi$ and $g$, and initially also on $\beta$, since the natural estimate one uses involves terms like $e^{\beta^{2}\|g\|_{L^{\infty}(\mathrm{supp}(\varphi)\times \mathrm{supp}(\varphi))}}$, but we can always bound this from above by replacing $\beta^{2}$ with $d$, so we get a bound independent of $\beta$.
  We may assume that $a \ge b$, the other case is handle in the same way.
  It then readily follows by applying Lemma~\ref{lemma:moment_inequality} and integrating that
  \[\left|\E \mu(\varphi)^a \overline{\mu(\varphi)}^b\right| \le C (\E |\mu(\varphi)|^2)^b\]
  for some constant $C$ independent of $\beta$.
\end{proof}

Finally we conclude with a proof of Lemma \ref{lemma:combinatorial_argument}.

\begin{proof}[Proof of Lemma \ref{lemma:combinatorial_argument}]
For fixed $x_1,\dots,x_N \in B(0,1)$, let $F \colon \{1,\dots,N\} \to \{1,\dots,N\}$ be the nearest neighbour function mapping $i \mapsto j$, where $j$ is the index of the closest point $x_j$ to the point $x_i$. By removing a set of measure $0$ from $B(0,1)^N$, we may assume that $F$ is uniquely defined. The integral then becomes
  \[\sum_{F} \int_{U_F} e^{\frac{\beta^2}{2}\sum_{j=1}^{N} \log \frac{1}{\frac{1}{2} |x_j - x_{F(j)}|}} \, dx_1 \dots dx_{N},\]
  where $U_F \subset B(0,1)^N$ is the set of those point configurations $(x_1,\dots,x_{N}) \in B(0,1)^N$ whose nearest neighbour function equals $F$.
  Each nearest neighbour function $F$ can be uniquely represented by a directed graph with vertices $\{1,\dots,N\}$ and an arrow from $i$ to $F(i)$.
  This graph is of the following form: It consists of $k \le \lfloor N/2 \rfloor$ components, and each component consists of a $2$-cycle (the two mutually closest points in the component, by the triangle inequality there can be no longer cycles) with two trees connected to the two vertices in the cycle.
  Without loss of generality we may assume that $(x_1, x_2), \dots, (x_{2k-1}, x_{2k})$ are the vertices forming the cycles.
  Perform now the change of variables $u_j = \frac{1}{2}(x_j - x_{F(j)})$ for $j=2k+1,\dots,N$, $u_1 = \frac{1}{2}(x_1 - x_2)$, $u_2 = \frac{1}{2}x_2$, \dots, $u_{2k-1} = \frac{1}{2}(x_{2k-1} - x_{2k})$ and $u_{2k} = \frac{1}{2} x_{2k}$. Then we get the integral
  \begin{align*}
    \int_{\tilde{U}_F} \frac{2^{N}}{|u_1|^{\beta^2} |u_3|^{\beta^2} \dots |u_{2k-1}|^{\beta^2} |u_{2k+1}|^{\beta^2/2} \dots |u_{N}|^{\beta^2/2}} \, du_1 \dots du_{N}
  \end{align*}
  for some new integration domain $\tilde{U}_F$.
  We have $|u_j| \le 1$ for all $j$ and moreover the balls $B_j = \{y \in \reals^d : |y - x_j| \le |u_j|\}$, $j=1,3,\dots,2k-1,2k+1,2k+2,\dots,N$ are disjoint (since $|u_j|$ is half the distance from $x_j$ to its nearest neighbour). Each such ball is contained in $B(0,2)$, and thus by comparing volumes we get the inequality
  \[|u_1|^d + |u_3|^d + \dots + |u_{2k-1}|^d + |u_{2k+1}|^d + \dots + |u_N|^d \le 2^d.\]
  In particular the new integration domain $\tilde{U}_F$ is contained in
  \[\{|u_1|^d + |u_3|^d + \dots + |u_{2k-1}|^d + |u_{2k+1}|^d + \dots + |u_N|^d \le 2^d, |u_2|, \dots, |u_{2k}| \le 1\}.\]
  Hence we get the upper bound
  \begin{align*}
    &\int_{\tilde{U}_F} \frac{2^N}{|u_1|^{\beta^2} |u_3|^{\beta^2} \dots |u_{2k-1}|^{\beta^2} |u_{2k+1}|^{\beta^2/2} \dots |u_N|^{\beta^2/2}} \, du_1 \dots du_N \\
    & \le c^{N} \int_{(\partial B(0,1))^{N-k}} \int_{r_1^d + \dots + r_{N - k}^d \le 2^d} r_1^{-\beta^2 + d - 1} \dots r_k^{-\beta^2 + d - 1} r_{k+1}^{-\frac{\beta^2}{2} + d - 1} \dots r_{N-k}^{-\frac{\beta^2}{2} + d - 1} \, d r_1 \dots d r_{N-k} \\
    & \le c^{N} \int_{t_1 + \dots + t_{N-k} \le 1} t_1^{-\frac{\beta^2}{d}} \dots t_k^{-\frac{\beta^2}{d}} t_{k+1}^{-\frac{\beta^2}{2d}} \dots t_{N-k}^{-\frac{\beta^2}{2d}} \, dt_1 \dots dt_{N-k} \\
    & \le c^{N} \frac{\Gamma(1 - \frac{\beta^2}{d})^k \Gamma(1 - \frac{\beta^2}{2d})^{N-2k}}{\Gamma(k(1 - \frac{\beta^2}{d}) + (N-2k)(1 - \frac{\beta^2}{2d}))} \int_0^1 t^{N-k - k \frac{\beta^2}{d} - (N-2k) \frac{\beta^2}{2d} - 1} \, dt \\
    & \le \frac{c^{N}}{\Gamma(k(1 - \frac{\beta^2}{d}) + (N - 2k)(1 - \frac{\beta^2}{2d}) + 1)}, 
  \end{align*}
  where $c$ is some constant that may get bigger on each line of the above and following computations, and which is allowed to depend on $\beta^2$ and $d$ but not on $N$ or $k$. Above we used Dirichlet's integral formula, see e.g. \cite[Section 12.5]{WW}. Thus we have
  \begin{equation}\label{eq:dirichlet_bound}
    \int_{U_F} e^{\frac{\beta^2}{2}\sum_{j=1}^{N} \log \frac{1}{\frac{1}{2} |x_j - x_{F(j)}|}} \, dx_1 \dots dx_{N} \le \frac{c^N}{\Gamma(N(1 - \frac{\beta^2}{2d}) - k + 1)},
  \end{equation}
  where the right hand side only depends on $F$ via the number of components in the directed graph associated with $F$.

  Next we bound the number of nearest neighbour functions whose graphs have $k$ components. As already mentioned above, each component consists of a $2$-cycle augmented with two trees, or a simpler way to think of them might be as unordered pairs of rooted trees whose roots form the cycle.
  It is worth noting that the map from the nearest neighbour functions to their associated graphs is not a surjection since geometrical reasons limit the number of incoming edges each vertex may have.
  However, since we are only concerned with an upper bound, we will ignore this fact and simply count all possible labeled graphs with $N$ vertices and $k$ components of the above prescribed type, with labels corresponding to the variables $x_1,\dots,x_N$.
  This is a fairly straightforward task to which standard counting methods using generating functions apply. Here we have written the argument using combinatorial species, see for example \cite{BLL} for an introduction to the subject. For an argument formulated in more elementary terms, we refer to \cite{GP}. Let $E_k$ be the species of (unordered) sets of $k$ elements and let $T$ be the species of rooted trees. The species of a single component in the graph is then $E_2 \circ T$ (an unordered pair of rooted trees, whose roots correspond to the cycle). A set of $k$ of these gives us then the required species $G_k$ of nearest neighbour graphs with $k$ components, $G_k = E_k \circ (E_2 \circ T)$.
  The labeled generating function of $E_k$ is given by $E_k(x) = \frac{x^k}{k!}$ and hence
  \[G_k(x) = \frac{(T(x)^2/2)^k}{k!} = \frac{T(x)^{2k}}{2^k k!}.\]
  The species $T$ itself satisfies the equation $T = X \cdot (E \circ T)$, where $E$ is the species of sets (a rooted tree consists of a root and a set of subtrees).
  Since $E(x) = e^x$, the labeled generating function of $T$ satisfies the equation $T(x) = x e^{T(x)}$. In particular, if we let $f(x) = x e^{-x}$, then $f$ is the compositional inverse of $T$, and we may use the Lagrange inversion formula to compute for $N \ge 2k$ that
  \[[x^N] T(x)^{2k} = \frac{2k}{N} [x^{-2k}] f(x)^{-N} = \frac{2k}{N} [x^{-2k}] \frac{e^{N x}}{x^N} = \frac{2k}{N} [x^{-2k}] \sum_{j=0}^\infty \frac{N^j x^{j-N}}{j!} = \frac{2k N^{N - 2k - 1}}{(N - 2k)!},\]
  where $[x^k] g(x)$ is the coefficient of $x^k$ in some power series $g$. Hence the number of nearest neighbour graphs with $N$ vertices and $k$ components (ignoring the geometrical restrictions) is
  \begin{equation}\label{eq:graph_bound}
    \frac{N! 2k N^{N - 2k - 1}}{2^k k! (N - 2k)!} \le c^N \frac{N!}{k!} \le c^N (N - k)!,
  \end{equation}
  where the first inequality follows by Stirling's approximation and the second follows from the fact that ${N \choose k} \le 2^N$.

  The proof is easily finished by combining \eqref{eq:dirichlet_bound} and \eqref{eq:graph_bound} with another application of Stirling: 
    \begin{align*}
    & \int_{B(0,1)^N} \exp\Big(\frac{\beta^2}{2} \sum_{j=1}^{N} \log \frac{1}{\frac{1}{2} \min_{k \neq j} |x_j - x_k|}\Big) \, dx_1 \dots dx_{N} \\
    & = \sum_{F} \int_{U_F} e^{\frac{\beta^2}{2}\sum_{j=1}^{N} \log \frac{1}{\frac{1}{2} |x_j - x_{F(j)}|}} \, dx_1 \dots dx_{N} \\
    & \le c^N \sum_{k=1}^{\lfloor N/2\rfloor} \frac{(N - k)!}{\Gamma(N(1 - \frac{\beta^2}{2d}) - k + 1)} \le c^N \sum_{k=1}^{\lfloor N/2\rfloor} N^{N\frac{\beta^2}{2d}} \le c^N N^{N\frac{\beta^2}{2d}},
  \end{align*}
where again the value of $c$ may not be same in each of the places it appears.
\end{proof}

\end{document}